\documentclass{amsart}

\usepackage{amsmath, amssymb, tikz}
\tikzstyle{n}=[circle, draw, fill, minimum size=6, inner sep=0]
\tikzstyle{ngr}=[circle, draw, fill, gray, minimum size=6, inner sep=0]
\tikzstyle{int}=[draw, circle, fill,  minimum size=3, inner sep=1]
\tikzstyle{intgr}=[draw, circle, fill, gray,  minimum size=3, inner sep=1]
\tikzstyle{ext}=[draw, circle,  minimum size=3, inner sep=1]
\usetikzlibrary{matrix}
\usetikzlibrary{arrows} 
\usetikzlibrary{decorations.pathreplacing}

\usepackage{hyperref}

\newtheorem{thm}{Theorem}[section]
 \newtheorem{prop}[thm]{Proposition}
 \newtheorem{corollary}[thm]{Corollary}
 \newtheorem{lemma}[thm]{Lemma}

\newtheorem{cond}[thm]{Condition}
\newtheorem{example}[thm]{Example}
\newtheorem{remark}[thm]{Remark}
\newtheorem{pty}[thm]{Property}

\numberwithin{equation}{section}

\numberwithin{figure}{section}

 \newcommand{\dontprint}[1]{\relax}

\theoremstyle{definition}
\newtheorem{defi}[thm]{Definition}

\newcommand{\wt}[1]{{\widetilde{#1}}}
\newcommand{\wh}[1]{{\widehat{#1}}}
\newcommand{\und}[1]{{\underline{#1}}}
\newcommand{\mc}{{\mathfrak{c}}}

\newcommand{\ma}{{\mathfrak{a}}}

\newcommand{\br}{{\}\hspace{-0.07cm}.\hspace{-0.03cm}.\hspace{-0.07cm}\} }}

\newcommand{\Av}{{\mathrm {Av}}}

\newcommand{\mF}{{\mathcal{F}}}
\newcommand{\Alg}{\mathrm{{Alg}}}
\newcommand{\alg}[1]{\mathfrak{{#1}}}

\newcommand{\co}[2]{\left[{#1},{#2}\right]} 

\newcommand{\Br}{{\mathsf{Br} }} 
\newcommand{\BT}{{\mathsf{BT} }}
\newcommand{\Com}{{\mathsf{Com} }}
\newcommand{\coCom}{{\mathsf{coCom} }}
\newcommand{\Lie}{{\mathsf {Lie}}}
\newcommand{\As}{{\mathsf {As}}}
\newcommand{\coAs}{{\mathsf{coAs}}}
\newcommand{\coLie}{{\mathsf {coLie}}}
\newcommand{\Ger}{{\mathsf {Ger}}}
\newcommand{\BV}{{\mathsf {BV}}}

\newcommand{\Ch}{{\mathsf {Ch}}}
\newcommand{\grVect}{{\mathsf {grVect}}}
\newcommand{\Conv}{{\mathrm {Conv}}}

\newcommand{\Tree}{{\mathsf{T r e e}}}

\newcommand{\Cbu}{C^{\bullet}}

\newcommand{\ve}{{\varepsilon}}

\newcommand{\vr}{{\varrho}}

\newcommand{\ml}{{\mathfrak{l}}}

\newcommand{\mi}{{\mathfrak{i}}}
\newcommand{\mC}{{\mathfrak{C}}}
\newcommand{\mT}{{\mathfrak{T}}}
\newcommand{\mG}{{\mathfrak{G}}}
\newcommand{\mD}{{\mathfrak{D}}}
\newcommand{\mfF}{{\mathfrak{F}}}

\newcommand{\mm}{{\mathfrak{m}}}

\newcommand{\ATw}{{\mathrm{ATw}\,}}
\newcommand{\Tw}{{\mathrm{Tw}\,}}
\newcommand{\tw}{{\mathrm{tw}\,}}
\newcommand{\lab}{{\mathrm{lab}\,}}

\newcommand{\bb}{{\bullet\bullet}}
\newcommand{\cc}{{\circ\circ}}
\newcommand{\cced}{{\circ\hspace{-0.05cm}-\hspace{-0.05cm}\circ}}

\newcommand{\bfzero}{{\bf 0}}

\newcommand{\id}{{\mathrm {i d} }}

\newcommand{\sgn}{{\mathrm {s g n}}}

\newcommand{\La}{{\Lambda}}
\newcommand{\D}{{\Delta}}
\newcommand{\Te}{{\Theta}}
\newcommand{\al}{{\alpha}}
\newcommand{\la}{{\lambda}}
\newcommand{\ga}{{\gamma}}
\newcommand{\de}{{\delta}}
\newcommand{\ka}{{\kappa}}
\newcommand{\si}{{\sigma}}

\newcommand{\vf}{{\varphi}}

\newcommand{\vs}{{\varsigma}}
\newcommand{\pa}{{\partial}}
\newcommand{\Cobar}{{\mathrm {C o b a r} }}

\newcommand{\Op}{{\mathbb{O P}}}

\renewcommand{\c}{{\circ}}
\newcommand{\bul}{{\bullet}}

\newcommand{\bs}{{\bf s}}
\newcommand{\bsi}{{\bf s}^{-1}\,}
\newcommand{\bu}{{\bf u}}
\newcommand{\bt}{{\bf t}}

\renewcommand{\L}{\langle}
\newcommand{\R}{\rangle}

\newcommand{\MC}{{\mathrm{MC}}}
\newcommand{\filtr}{{\mathrm{filtr}}}

\newcommand{\Hom}{{\mathsf{Hom}}}
\newcommand{\Ind}{{\mathsf{Ind}}}
\newcommand{\emb}{{\mathrm{emb}}}
\newcommand{\End}{\mathsf{End}}

\newcommand{\cT}{\mathcal{T}}

\newcommand{\cO}{\mathcal{O}}
\newcommand{\cG}{\mathcal{G}}
\newcommand{\cF}{\mathcal{F}}
\newcommand{\cC}{\mathcal{C}}
\newcommand{\cL}{\mathcal{L}}

\newcommand{\bbK}{{\mathbb{K}}}
\newcommand{\bbZ}{{\mathbb{Z}}}

\newcommand{\nod}{{\mathrm{nod}}}

\newcommand{\Sh}{{\mathrm {S h} }}

\DeclareMathOperator{\Gr}{Gr}

\newcommand{\Operads}{{\mathit{Operads}}}

\title{Operadic Twisting -- with an application to Deligne's conjecture }

\author{Vasily Dolgushev}
\author{Thomas Willwacher}

\begin{document}


\date{}

\maketitle

~\\[-0.5cm]

\begin{flushright} 
{\it To the memory of Jean-Louis Loday}\\
~~~\\
\end{flushright}

\begin{abstract}
We study categorial properties of the operadic twisting functor $\Tw$. In particular, we show that $\Tw$ is a comonad. Coalgebras of this comonad are operads for which a natural notion of twisting by Maurer-Cartan elements exists. 
We give a large class of examples, including the classical cases of the Lie, associative and Gerstenhaber operads, and their infinity-counterparts $\Lie_\infty$, $\As_\infty$, $\Ger_\infty$. 
We also show that $\Tw$ is well behaved with respect to the homotopy theory of operads. As an application we show that every solution of Deligne's conjecture is homotopic to a solution that is compatible with twisting.
\end{abstract}

\tableofcontents

\section{Introduction}

The idea of twisting algebraic objects by Maurer-Cartan elements has been 
a folklore for a very long time and the original nudge for this idea probably belongs 
to D. Quillen \cite{Quillen}. 
 
For example, let $\alg{g}$ be a differential graded Lie algebra and 
let $m\in \alg{g}$ be a Maurer-Cartan element, i.e., a degree 1 element satisfying 
the Maurer-Cartan equation
\[
dm+\frac{1}{2}\co{m}{m} = 0.
\]
Then the twisted Lie algebra $\alg{g}^m$ is the graded space $\alg{g}$ with differential $d+\co{m}{\cdot}$ and 
the same bracket $\co{\cdot}{\cdot}$. In a similar manner the notion of Maurer-Cartan element and twisting can be defined for associative and Gerstenhaber algebras, and their homotopy versions, i.e., $\Lie_\infty$, $\As_\infty$ and $\Ger_\infty$ algebras (see, e. g. \cite{C-Lazarev09}, \cite{CEFT}, \cite[Section 2.4]{thesis}, 
and \cite{Yek}). 

In this paper we investigate in how far the operation of twisting by Maurer-Cartan elements can be defined in general for algebras $A$ over an operad $\cO$. We assume that there is a map from the $\Lie_\infty$ operad (or a degree shifted version thereof) into $\cO$. This endows $A$ with a $\Lie_\infty$ algebra structure. A Maurer-Cartan element $m$ of $A$ is then defined as a Maurer-Cartan element of the $\Lie_\infty$ algebra $A$. 
We may twist the $\Lie_\infty$ algebra $A$ by the Maurer-Cartan element $m$, but in general the twisted $\Lie_\infty$ algebra structure does not necessarily lift to an $\cO$ algebra structure in a natural way. 

We provide an algebraic formalism to deal with operads $\cO$ for which there is a natural lift. In fact, we show in Theorem \ref{thm:Twcomonad} that the operadic twisting functor $\Tw$ introduced in \cite{grt} is naturally equipped with the structure of a comonad. This comonad furthermore has the following properties\footnote{Here we postpone the precise statements to later sections, since they require terminology introduced in the definition of $\Tw$.}:
\begin{itemize}

\item  The twisting procedure by Maurer-Cartan elements can be defined for algebras 
over an operad which is, in turn, a coalgebra for the comonad $\Tw$, see Section \ref{sec:twisting-algebras}.

\item A map $\cO \to \cO'$ of operads induces a functor from the category of $\cO'$-algebras to the category of $\cO$-algebras. If $\cO\to \cO'$ gives rise to a map of $\Tw$-coalgebras, then this functor is compatible with the operation of twisting by Maurer-Cartan elements, see Theorem \ref{thm:atwfunctor}.

\item The functor $\Tw$ preserves quasi-isomorphisms between operads, see Theorem 
\ref{thm:Twpreservesqiso}.
\end{itemize}

Concretely, the twisted operad $\Tw \cO$ is a completion of the operad generated by $\cO$ and a unary element, representing the Maurer-Cartan element, with a suitable differential. A pair consisting of an $\cO$ algebra $A$ and a Maurer-Cartan element $m$ defines a $\Tw \cO$ algebra, and moreover the category $\Alg^{\MC}_{\cO}$ of such pairs is isomorphic to the category of (filtered) $\Tw\cO$ algebras, see Theorem \ref{thm:twisting-cont}. A $\Tw$ coalgebra structure on $\cO$ is described by map $\cO\to \Tw\cO$. The twisted $\cO$ algebra $A^m$ is obtained by pulling back the $\Tw\cO$ algebra structure along this map. 

The axioms for a coalgebra over a comonad translate into the statements that twisting with the zero Maurer-Cartan element retains the original algebra and twisting by a Maurer-Cartan element $m$, followed by twisting by a Maurer-Cartan element $m'$ is the same as twisting by the Maurer-Cartan element $m+m'$.


We apply the developed machinery to study solutions of (our version of) the Deligne conjecture. 
By this we mean, abusing notation, a quasi-isomorphism 
\begin{equation}
\label{Gerinfty-to-Br}
F \colon \Ger_\infty \to \Br
\end{equation}
from the operad $\Ger_{\infty}$ governing homotopy Gerstenhaber algebras 
to the braces operad $\Br$\,.

Both operads $\Ger_{\infty}$ and $\Br$
are $\Tw$-coalgebras, i.e., their algebras can be twisted 
by Maurer-Cartan elements. However, in general, a quasi-isomorphism 
\eqref{Gerinfty-to-Br} does not need to be a morphism of $\Tw$-coalgebras. 

Using our machinery, we establish that every homotopy class
of a quasi-isomorphism \eqref{Gerinfty-to-Br} has at least one representative 
which respects the structures of $\Tw$-coalgebras on  $\Ger_{\infty}$ and $\Br$.
More precisely, 
\begin{thm}
\label{thm:main}
Suppose that $F: \Ger_\infty \to \Br$ is a quasi-isomorphism of dg operads. Then there 
exists a quasi-isomorphism of dg operads $F': \Ger_\infty \to \Br$ such that 
\begin{itemize}

\item $F'$ is a homomorphism of $\Tw$-coalgebras, and

\item $F'$ is homotopy equivalent to $F$.

\end{itemize}

Furthermore, the morphism $F'$ is given by the explicit formula \eqref{F-pr}.
\end{thm}

Let us recall that, since the operad $\Br$ acts on the Hochschild cochain 
complex $\Cbu(A)$ of an $A_{\infty}$-algebra $A$,  a quasi-isomorphism \eqref{Gerinfty-to-Br} gives
a $\Ger_{\infty}$-structure on $\Cbu(A)$\,.
Theorem \ref{thm:main} has an interesting consequence related to this 
$\Ger_{\infty}$-structure which we will describe now.

Let $A$ be an $A_{\infty}$-algebra and let $F_A$
denote the composition
\begin{equation}
\label{F-A}
F_A = \ma_{A}  \circ F : \Ger_{\infty} \to \End_{\Cbu(A)},
\end{equation}
where $ \End_{\Cbu(A)}$ is the endomorphism operad of $\Cbu(A)$  and 
$$
\ma_{A} : \Br \to  \End_{\Cbu(A)}
$$
is the operad morphism coming from the action of $\Br$ on $\Cbu(A)$\,.
Let us assume, in addition, that the dg Lie algebra $\Cbu(A)$ carries a complete descending 
filtration\footnote{Such a filtration may come from a formal deformation parameter 
$\ve$ which enters the game when we extend the base field to the ring 
$\bbK[[\ve]]$\,.} $\cF_{\bul} \Cbu(A)$ and let $\al$ be a Maurer-Cartan element in $\cF_1 \Cbu(A)$\,.

Such a Maurer-Cartan element gives us a new $A_{\infty}$-structure on $A$ and 
we denote this new $A_{\infty}$-algebra by $A^{\al}$\,. In addition, the  Maurer-Cartan element 
$\al$ can be used to twist the dg Lie algebra structure on $\Cbu(A)$ and form a new 
cochain complex 
$$
\Cbu(A)^{\al}\,.
$$

It is easy to see that the cochain complex $\Cbu(A)^{\al}$ coincides with the Hochschild
cochain complex $\Cbu(A^{\al})$ of the new $A_{\infty}$-algebra $A^{\al}$\,. 
Furthermore, there are two natural ways of obtaining a $\Ger_{\infty}$-structure 
on $\Cbu(A^{\al}) = \Cbu(A)^{\al}$\,. First, we may compose the operad morphism
$$
\ma_{A^{\al}} : \Br \to \End_{\Cbu(A^{\al}) }
$$
with $F: \Ger_{\infty} \to \Br$ and, second, we may twist $F_A$ \eqref{F-A} by 
$\al$ and obtain
$$
F^{\al}_A :  \Ger_{\infty} \to \End_{\Cbu(A)^{\al}}\,.
$$

We claim that
\begin{corollary}
\label{cor:main}
For every solution $F$ \eqref{Gerinfty-to-Br} of Deligne's conjecture, 
the maps  $F_{A^{\al}}$ and $F_A^{\al}$ are homotopic. Furthermore, if 
the morphism $F$ is compatible with the coalgebra structures on 
$\Ger_{\infty}$ and $\Br$ over $\Tw$, then  $F_{A^{\al}}$ and $F_A^{\al}$
coincide. 
\end{corollary}
For the proof of this corollary we refer the reader to Section \ref{sec:cor-main}.

\subsection{Structure of the paper} 
In Section \ref{sec:preliminaries}, we fix notation and recall some facts about operads
and their dual versions. In Section \ref{sec:twist}, we introduce the functor $\Tw$ \cite{grt}
and give an alternative description of the category 
of algebras over an operad $\Tw \cO$.  In this section, we also describe in detail
the dg operad $\Tw \La \Lie_{\infty}$. 

Section \ref{sec:categorial} is devoted to categorical properties of the functor $\Tw$\,. 
We show that the functor $\Tw$ forms a comonad on the under-category 
$\La \Lie_{\infty}\downarrow \Operads$. At the end of this section, we prove
that the dg operad $\Ger_{\infty}$ governing homotopy Gerstenhaber algebras   
has the canonical coalgebra structure over the comonad $\Tw$.

In Section  \ref{sec:twisting-homotopy}, we describe homotopy theoretic 
properties of the functor $\Tw$ and introduce the notion of homotopy fixed 
point for $\Tw$ (see Definition \ref{dfn:homot-fixed}). In this section, we also 
describe a large class of dg operads which are simultaneously $\Tw$-coalgebras
and homotopy fixed points for $\Tw$. This class includes the classical operads 
$\Lie$, $\Ger$ and the operad governing Poisson algebras.  
At the end of this section, we show that the operads $\As$ and 
$\As_{\infty}$ are $\Tw$-coalgebras and are homotopy fixed points for $\Tw$\,.

In Section \ref{sec:twisting-algebras}, we consider a dg operad $\cO$
which carries a coalgebra structure over $\Tw$\,. We introduce the notion 
of twisting by Maurer-Cartan elements for $\cO$-algebras and describe 
categorial and homotopy theoretical properties of this twisting procedure. 

Sections \ref{sec:BT}, \ref{sec:TwBT} and \ref{sec:Br} treat the braces operad $\Br$\,. 
In this section, we show how to recover the dg operad $\Br$ (essentially) as the twisted version 
of  a simpler operad $\BT$. The proofs of Theorem \ref{thm:main} and Corollary \ref{cor:main} 
are given in Section \ref{sec:theproof}.


In Appendix \ref{app:Ger}, we recall the operad $\Ger$ governing 
Gerstenhaber algebras and, in Appendix \ref{app:action}, we describe in detail the action 
of the operads $\BT$, $\Tw\BT$ and $\Br$ on the Hochschild cochain 
complex $\Cbu(A)$ of an $A_{\infty}$-algebra $A$\,.
Appendices  \ref{app:invar-coinvar} and \ref{app:filtered-lem} are devoted to 
two auxiliary technical statements. 
In Appendix \ref{app:brhomfixed}, we sketch a different 
proof of the fact that the braces operad $\Br$ is a homotopy fixed 
point for $\Tw$. 


\subsection*{Memorial note} With great sadness, we learned that
Jean-Louis Loday died as a consequence of an accident on June 6, 2012. 
He is widely known for his contributions to algebraic $K$-theory, for  
his research linking cyclic homology, $K$-theory, and combinatorics.
His foundational work on operads has and will have a lasting impact on the 
development of algebra and topology. We devote our paper to the memory of 
this great mathematician.

\subsection*{Acknowledgements}
We would like to thank Martin Markl, Ezra Getzler,  Dmitry Tamarkin and Bruno 
Vallette for helpful discussions.
V.D. acknowledges the NSF grants DMS-0856196, DMS-1161867 and 
the grant FASI RF 14.740.11.0347. V.D.'s NSF grant  DMS-0856196 was supported by 
American Recovery and Reinvestment Act (ARRA) and it could not be transferred 
to Temple University from the UC Riverside by ARRA regulations. V.D. would like to 
thank Charles Greer, Tim LeFort, and Charles Louis from the UCR Office 
of Research for allowing V.D. to save this grant 
by creating a subaward from the UC Riverside to Temple University. 
V.D. also thanks John Baez for serving formally as the PI on DMS-0856196
starting July 1, 2010.
Part of this work was conducted while T.W. held a junior fellowship of the Harvard Society of Fellows. T.W. furthermore acknowledges partial support of the Swiss National Science Foundation, grants PDAMP2\_137151 and 200021\_150012.

Both authors are very grateful to the referee for her/his suggestions and patience in reading this long paper, and for providing the references \cite{Fresse98,KapranovManin,M-Smith2} that the authors were unaware of.

\section{Preliminaries}
\label{sec:preliminaries}
The underlying field $\bbK$ has characteristic zero. 
For a set $X$ we denote by $\bbK\L X\R$ the $\bbK$-vector 
space of finite linear combinations of elements in $X$\,.

The notation $S_{n}$ is reserved for the symmetric group 
on $n$ letters and  $\Sh_{p_1, \dots, p_k}$ denotes 
the subset of $(p_1, \dots, p_k)$-shuffles 
in $S_n$, i.e.  $\Sh_{p_1, \dots, p_k}$ consists of 
elements $\si \in S_n$, $n= p_1 +p_2 + \dots + p_k$ such that 
$$
\begin{array}{c}
\si(1) <  \si(2) < \dots < \si(p_1),  \\[0.3cm]
\si(p_1+1) <  \si(p_1+2) < \dots < \si(p_1+p_2), \\[0.3cm]
\dots   \\[0.3cm]
\si(n-p_k+1) <  \si(n-p_k+2) < \dots < \si(n)\,.
\end{array}
$$

The underlying symmetric monoidal category $\mC$ is 
either the category $\grVect_{\bbK}$ of $\bbZ$-graded
$\bbK$-vector spaces or the category 
$\Ch_{\bbK}$ of (possibly) unbounded cochain 
complexes of $\bbK$-vector spaces.  
We frequently use the ubiquitous
combination ``dg'' (differential graded) to refer to algebraic objects 
in $\Ch_{\bbK}$\,. 
For a homogeneous vector $v$ in 
a cochain complex, $|v|$ denotes the degree of $v$\,.
Furthermore, we denote 
by $\bs$ (resp. $\bs^{-1}$) the operation of suspension (resp. desuspension)
on $\grVect_{\bbK}$ or $\Ch_{\bbK}$. 

For a groupoid $\cG$ the notation $\pi_0(\cG)$ is reserved for 
the set of its isomorphism classes.

\subsection{Preliminaries on operads, pseudo-operads and their dual versions}
\label{sec:prelim-oper}
By a {\it collection} we mean the sequence $\{ P(n) \}_{n \ge 0}$ 
of objects of the underlying symmetric monoidal category $\mC$
such that for each $n$, the object $P(n)$ is equipped
with a left action of the symmetric group $S_n$\,.  

The notation $\Lie$ (resp. $\As$, $\Com$, $\Ger$) is reserved for the 
operad  governing Lie algebras (resp. associative algebras without unit, commutative 
(and associative) algebras without unit, Gerstenhaber algebras without unit).
Dually, the notation $\coLie$ (resp. $\coAs$, $\coCom$) is reserved for the 
cooperad governing Lie coalgebras (resp.  coassociative coalgebras without 
counit, cocommutative (and coassociative)
coalgebras without counit).

For an operad $\cO$ (resp. a cooperad $\cC$) and a cochain complex $V$, the notation 
$\cO(V)$ (resp. $\cC(V)$) is reserved for the free $\cO$-algebra (resp. cofree $\cC$-coalgebra). 
Namely, 
\begin{equation}
\label{cO-V}
\cO(V) : = \bigoplus_{n \ge 0} \Big( \cO(n) \otimes V^{\otimes \, n} \Big)_{S_n}\,,
\end{equation}
\begin{equation}
\label{cC-V}
\cC(V) : = \bigoplus_{n \ge 0} \Big( \cC(n) \otimes V^{\otimes \, n} \Big)^{S_n}\,.
\end{equation}

For an operad $\cO$, we denote by $\Alg_{\cO}$ the category of 
$\cO$-algebras.

Let $ a_1, a_2, \dots, a_n $ be degree $0$ dummy variables
and $\cO$ be a dg operad.
It is clear that
$\cO(n)$ is naturally identified with the subspace of the free 
$\cO$-algebra 
$$
\cO\big( \bbK \L a_1, a_2, \dots, a_n \R \big) 
$$
spanned by $\cO$-monomials in which each variable from the set 
$\{ a_1, a_2, \dots, a_n \} $ 
appears exactly once. We often use this identification 
in our paper. 

Let us denote by $\ast$ the following collection (in $\Ch_{\bbK}$)
\begin{equation}
\label{ast}
\ast(n) = \begin{cases}
 \bbK \qquad {\rm if} ~~ n = 1\,, \\
 \bfzero \qquad {\rm otherwise}.
\end{cases}
\end{equation}
It is easy to see that $\ast$
is equipped with the unique structure of an operad and 
a unique structure of a cooperad. 

Recall that an operad $\cO$ (resp. a cooperad $\cC$) is called 
{\it augmented} (resp. {\it coaugmented}) if it comes with an operad map
$\cO \to \ast$ (resp. a cooperad map $\ast \to \cC$).
For an augmented operad $\cO$ (resp. a coaugmented cooperad $\cC$) the 
notation $\cO_{\c}$ (resp. $\cC_{\c}$) is reserved for the kernel 
(resp. the cokernel) of the augmentation (resp. coaugmentation).

Recall that a pseudo-operad\footnote{Note that a pseudo-operad
\und{is not} a non-unital operad. For more details about this subtlety, 
see paper \cite{Markl} in which a pseudo-operad is 
called a non-unital Markl's operad.} (resp. pseudo-cooperad) is  vaguely
``an operad (resp. a cooperad) without the unit (resp. counit) 
axiom''. For the more precise definition of a pseudo-operad (resp. 
a pseudo-cooperad) we refer the reader to \cite[Sections 3.2, 3.4]{notes}
or \cite[Section 1.3]{MSS-book}. We remark that for every 
augmented operad $\cO$ (resp. coaugmented cooperad $\cC$)  the kernel 
of the augmentation $\cO_{\c}$ (resp. the cokernel of the coaugmentation 
$\cC_{\c}$) is naturally a pseudo-operad  (resp. pseudo-cooperad). In fact, 
due to \cite[Proposition 21]{Markl}, the assignment $\cO \mapsto \cO_{\c}$
(resp. $\cC \mapsto \cC_{\c}$) upgrades to an equivalence between the 
category of augmented operads (resp. coaugmented cooperads) and 
the category of pseudo-operads (resp. pseudo-cooperads). 

For a pseudo-operad (or an operad) $P$ in $\Ch_{\bbK}$ we denote 
by $\circ_i$ the {\it $i$-th elementary insertion}\footnote{Following the 
general convention, the numbers $n$ and $k$ are suppressed from 
the notation $\circ_i$.}
\begin{equation}
\label{circ-i}
\circ_i :   P (n) \otimes P(k) \to P(n+k-1)\,. 
\end{equation}
Here $n \ge 1$, $k \ge 0$, and $1 \le i \le n$.  

For an operad (resp. a cooperad) $P$ in $\Ch_{\bbK}$ we denote by $\La P$ the 
operad (resp. the cooperad) with the spaces of $n$-ary operations: 
\begin{equation}
\label{La-P}
\La P (n) = \bs^{1-n} P(n) \otimes \sgn_n\,,
\end{equation}
where $\sgn_n$ denotes the sign representation of $S_n$\,.
For example, an algebra over $\La\Lie$ is
a graded vector space $V$ equipped with the binary operation: 
$$
\{\,,\, \} : V \otimes  V \to V 
$$
of degree $-1$ satisfying the identities:  
$$
\{v_1,v_2\} = (-1)^{|v_1| |v_2|} \{v_2, v_1\}
$$
$$
\{\{v_1, v_2\} , v_3\} +  
(-1)^{|v_1|(|v_2|+|v_3|)} \{\{v_2, v_3\} , v_1\} +
(-1)^{|v_3|(|v_1|+|v_2|)} \{\{v_3, v_1\} , v_2\}  = 0\,,
$$
where $v_1, v_2, v_3 $ are homogeneous vectors in $V$\,.

$\Ger^{\vee}$ denotes the Koszul dual cooperad \cite{Fresse}, 
\cite{GJ}, \cite{GK}, \cite{Loday-Vallette} for $\Ger$\,. It is known \cite{Hinich} that 
\begin{equation}
\label{Ger-vee}
\Ger^{\vee} = \La^2 \Ger^{*}\,,
\end{equation}
where $\Ger^{*}$ is obtained from the operad 
$\Ger$ by taking the linear dual.
In other words, algebras over the linear dual 
$(\Ger^{\vee})^*$ are very much like Gerstenhaber algebras 
except that the bracket carries degree $1$ and the multiplication 
carries degree $2$\,.

\subsubsection{Intrinsic derivations of an operad}
Let $P$ be a dg operad. Then the operation
$\circ_1$ equips $P(1)$ with a structure of an associative 
algebra. We consider $P(1)$ as a Lie algebra with 
the Lie bracket being the (graded) commutator and we claim that  
\begin{prop}
\label{prop:P-1-deriv}
The formula 
\begin{equation}
\label{P-1-deriv}
\de_b (a) = b \circ_1 a - (-1)^{|a| |b|} \sum_{i=1}^n a \circ_i b   
\end{equation}
with 
$$
b \in P(1), \qquad \textrm{and} \qquad  a\in P(n)
$$ 
defines an operadic derivation of $P$ for every $b \in P(1)$\,.
Furthermore, for every pair $b_1, b_2 \in P(1)$, we have
$$
[\de_{b_1}, \de_{b_2}] = \de_{[b_1, b_2]}\,.
$$
\end{prop}
For the proof of this proposition we refer the reader to \cite[Section 6.1]{notes}. An operadic derivation 
of the form \eqref{P-1-deriv} is called {\it intrinsic}.

\subsection{Conventions related to trees}
\label{sec:trees}

By a {\it tree} we mean a connected graph without cycles with a marked vertex 
called {\it the root}.  In this paper, we assume that the root of 
every tree has valency $1$ (such trees are sometimes called  {\it planted}). 
The edge adjacent to the root is called the {\it root edge}.
Non-root vertices of valency $1$ are 
called {\it leaves}.  A vertex is called {\it internal} if it is 
neither a root nor a leaf. We always orient trees in the 
direction towards the root. Thus every internal vertex 
has at least $1$ incoming edge and exactly $1$ outgoing edge. 
An edge adjacent to a leaf is called {\it external}.   
A tree $\bt$ is called {\it planar} if, for every internal vertex $v$ of $\bt$, the set 
of edges terminating at $v$ carries a total order.

Let us recall (see \cite[Section 2]{notes}) that for every  planar tree $\bt$ 
the set $V(\bt)$ of all its vertices is equipped with a natural total order
such that the root is the smallest vertex of the tree.

We have an obvious bijection between the set of 
edges $E(\bt)$ of a tree $\bt$ and the subset of vertices: 
\begin{equation}
\label{no-root}
V(\bt) \setminus \{\textrm{root vertex}\}\,.
\end{equation}
This bijection assigns to a vertex $v$ in (\ref{no-root}) its 
outgoing edge. 
Thus the canonical total order on the set (\ref{no-root}) gives 
us a natural total order on the set of edges $E(\bt)$\,. 

For a non-negative integer $n$, an $n$-labeled planar tree $\bt$ is 
a planar tree equipped with 
an injective map 
\begin{equation}
\label{labeling}
\ml : \{1,2, \dots, n\} \to L(\bt)
\end{equation}
from the set $\{1,2, \dots, n\}$ to the set $L(\bt)$ of leaves of $\bt$\,.
Although the set $L(\bt)$ has a natural total order we do not require 
that the map \eqref{labeling} is monotonous. 

The set $L(\bt)$ of leaves of an $n$-labeled planar tree $\bt$
splits into the disjoint union of the image $\ml(\{1,2, \dots, n\})$
and its complement. We call leaves in the image of $\ml$
{\it labeled}.

A vertex $x$ of an $n$-labeled planar tree $\bt$ is called 
{\it nodal} if it is neither the root, nor a labeled leaf. 
We denote by $V_{\nod}(\bt)$ the set of all nodal vertices of 
$\bt$. Keeping in mind the canonical total order on 
the set of all vertices of $\bt$ we can say things like
``the first nodal vertex'', ``the second nodal vertex'', and
``the $i$-th nodal vertex''. 

It is convenient to talk about (co)operads and pseudo-(co)operads 
using the groupoid  $\Tree(n)$ of $n$-labeled planar trees. 
Objects of $\Tree(n)$ are $n$-labeled planar trees and 
morphisms are \und{non-planar} isomorphisms of the corresponding 
(non-planar) trees compatible with labeling. 
The groupoid   $\Tree(n)$ is equipped with an obvious left action 
of the symmetric group $S_n$\,. 

Following \cite[Section 3.2, 3.4]{notes}, for an $n$-labelled planar tree $\bt$ and 
pseudo-operad $P$ (resp. pseudo-cooperad $Q$) the notation
$\mu_{\bt}$ (resp. the notation $\D_{\bt}$) is reserved for the multiplication 
map 
\begin{equation}
\label{mu-bt}
\mu_{\bt} :  P (r_1) \otimes P(r_2) \otimes \dots \otimes P(r_k) \to P(n)
\end{equation}
and the comultiplication map 
\begin{equation}
\label{Delta-bt}
\D_{\bt} : Q(n) \to Q(r_1) \otimes Q(r_2) \otimes \dots \otimes Q(r_k)\,.
\end{equation}
respectively. Here, $k$ is the number of nodal vertices of the planar tree $\bt$
and $r_i$ is the number of edges (of $\bt$) which terminate at the 
$i$-th nodal vertex of $\bt$\,. 

For example, if $\bt_{n,k,i}$ is the labeled planar tree shown 
on figure\footnote{On figures, small white circles are used for 
nodal vertices and small black circles are used for 
all the remaining vertices.} \ref{fig:bt-n-k-i} then the map 
\begin{equation}
\label{mu-bt-nki}
\mu_{\bt_{n,k,i}} :  P (n) \otimes P(k) \to P(n+k-1)
\end{equation}
is precisely the $i$-th elementary insertion \eqref{circ-i}.
\begin{figure}[htp]
\centering 
\begin{tikzpicture}[scale=0.5]
\tikzstyle{w}=[circle, draw, minimum size=3, inner sep=1]
\tikzstyle{b}=[circle, draw, fill, minimum size=3, inner sep=1]
\node[b] (v1) at (0, 2) {};
\draw (0,2.6) node[anchor=center] {{\small $1$}};
\node[b] (v2) at (1.5, 2) {};
\draw (1.5,2.6) node[anchor=center] {{\small $2$}};
\draw (2.8,2) node[anchor=center] {{\small $\dots$}};
\node[b] (v1i) at (4, 2) {};
\draw (4,2.6) node[anchor=center] {{\small $i-1$}};
\node[w] (vv) at (6, 2) {};
\node[b] (vi) at (4.5, 4) {};
\draw (4.5,4.6) node[anchor=center] {{\small $i$}};
\draw (5.8,4) node[anchor=center] {{\small $\dots$}};
\node[b] (v1ik) at (7, 4) {};
\draw (7.6,4.6) node[anchor=center] {{\small $i+k-1$}};
\node[b] (vik) at (8, 2) {};
\draw (8,2.6) node[anchor=center] {{\small $i+k$}};
\draw (10,2) node[anchor=center] {{\small $\dots$}};
\node[b] (v1nk) at (12, 2) {};
\draw (12.6,2.6) node[anchor=center] {{\small $n+k-1$}};
\node[w] (v) at (6, 0.5) {};
\node[b] (r) at (6, -0.5) {};
\draw (vv) edge (vi);
\draw (vv) edge (v1ik);
\draw (v) edge (v1);
\draw (v) edge (v2);
\draw (v) edge (v1i);
\draw (v) edge (vv);
\draw (v) edge (vik);
\draw (v) edge (v1nk);
\draw (r) edge (v);
\end{tikzpicture}
\caption{The $(n+k-1)$-labeled planar tree $\bt_{n,k,i}$} 
\label{fig:bt-n-k-i}
\end{figure}
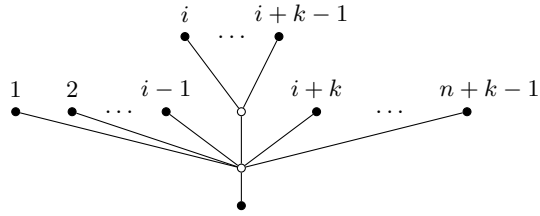

Let $\si$ be an element in $S_{n+k-1}$ and $\si(\bt_{n,k,i})$ be the 
$(n+k-1)$-labelled planar tree obtained from $\bt_{n,k,i}$ (on figure \ref{fig:bt-n-k-i})
via acting by the permutation $\si$. Sometimes it is convenient to use a different notation for 
the multiplication map $\mu_{\si(\bt^{n,k,i})}$ corresponding to
the tree $\si(\bt^{n,k,i})$\,.
More precisely, for a vector $v \in P(n)$ and $w \in P(k)$ of 
a pseudo-operad $P$ we will use this notation  
\begin{equation}
\label{new-notation}
v\big(\si(1), \dots \si(i-1), w\big(\si(i), \dots, \si(i+k-1) \big), 
\si(i+k), \dots, \si(n+k-1)\big) : = \mu_{\bt^{n,k,i}_{\si}} (v,w)\,. 
\end{equation}

The notation $\Tree_2(n)$ is reserved 
for the full sub-groupoid 
of $\Tree(n)$ whose objects are $n$-labelled planar trees 
with exactly $2$ nodal vertices. It is not hard to see that 
every object in  $\Tree_2(n)$ has at most $n+1$ leaves. 
Furthermore, isomorphism classes of $\Tree_2(n)$ are 
in bijection with the union 
\begin{equation}
\label{shuffles}
\bigsqcup_{0 \le p \le n} \Sh_{p, n-p}
\end{equation}
where $\Sh_{p, n-p}$ denotes the set of $(p, n-p)$-shuffles in 
$S_n$\,. The bijection assigns to a $(p, n-p)$-shuffle $\tau$ the 
isomorphism class of the planar tree depicted on figure \ref{fig:shuffle}. 
\begin{figure}[htp]
\centering
\begin{tikzpicture}[scale=0.5]
\tikzstyle{w}=[circle, draw, minimum size=3, inner sep=1]
\tikzstyle{b}=[circle, draw, fill, minimum size=3, inner sep=1]
\node[b] (l1) at (-1, 4) {};
\draw (-1,4.6) node[anchor=center] {{\small $\tau(1)$}};
\draw (0,3.8) node[anchor=center] {{\small $\dots$}};
\node[b] (lp) at (1, 4) {};
\draw (1,4.6) node[anchor=center] {{\small $\tau(p)$}};
\node[w] (vv) at (0, 2) {};
\node[b] (lp1) at (3, 2.5) {};
\draw (3,3.1) node[anchor=center] {{\small $\tau(p+1)$}};
\draw (4.25,2.4) node[anchor=center] {{\small $\dots$}};
\node[b] (ln) at (5.5, 2.5) {};
\draw (5.5,3.1) node[anchor=center] {{\small $\tau(n)$}};
\node[w] (v) at (3, 1) {};
\node[b] (r) at (3, 0) {};
\draw (vv) edge (l1);
\draw (vv) edge (lp);
\draw (v) edge (vv);
\draw (v) edge (lp1);
\draw (v) edge (ln);
\draw (r) edge (v);
\end{tikzpicture}
\caption{\label{fig:shuffle} Here $\tau$ is a $(p, n-p)$-shuffle}
\end{figure}
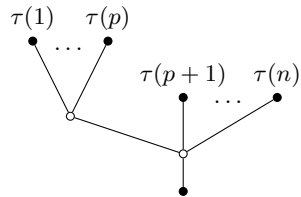

\subsection{The cobar construction and the convolution Lie algebra}
\label{sec:cobar-conv}

In this paper, we denote by $\Op(P)$ the free operad generated by a 
collection $P$\,. 

Let us recall that 
the cobar construction $\Cobar$ \cite{Fresse}, \cite{GJ}, \cite{GK}, 
\cite{Loday-Vallette} is a functor 
from the category of coaugmented dg cooperads  to the category 
of augmented dg operads. It is used to construct free resolutions 
for operads.  

More precisely, for a coaugmented dg cooperad $\cC$,
\begin{equation}
\label{Cobar-cC}
\Cobar(\cC) = \Op(\bs \, \cC_{\c})
\end{equation}
as the operad in the category $\grVect_{\bbK}$,
and the differential $\pa^{\Cobar}$ on $\Cobar(\cC)$
is define by the equations
\begin{equation}
\label{equ:cobarpa}
\pa^{\Cobar} = \pa' + \pa''\,, 
\end{equation}
with 
\begin{equation}
\label{pa-pr}
\pa' (X) = - \bs\, \pa_{\cC} \bsi X\,,
\end{equation}
and 
\begin{equation}
\label{pa-prpr}
\pa'' (X) = - \sum_{z \in \pi_0(\Tree_2(n)) } (\bs \otimes \bs)
\big( \bt_z ; \D_{\bt_z} (\bsi X) \big)\,,
\end{equation}
where $X \in \bs\, \cC_{\c}(n)$, $\bt_z$ is any 
representative\footnote{The axioms of a pseudo-cooperad imply that the right hand side of
\eqref{pa-prpr} does not depend on the choice of representatives 
$\bt_z$\,.} of the isomorphism class $z\in \pi_0(\Tree_2(n))$, and $\pa_{\cC}$ is 
the differential on $\cC$\,.

Let $P$ (resp. $Q$) be a pseudo-operad  (resp. pseudo-cooperad)  
in the category $\Ch_{\bbK}$. Following \cite{KapranovManin} and \cite{MV-nado}, we equip 
the cochain complex
\begin{equation}
\label{Conv}
\Conv(Q, P) =  \prod_{n \ge 0}
\Hom_{S_n} (Q(n), P(n))\,. 
\end{equation}
with the binary operation $\bullet$ defined by the formula
\begin{equation}
\label{Conv-bullet}
f \bullet g (X) =  
\sum_{z\in \pi_0\big(\Tree_2(n)\big)} 
\mu_{\bt_z} (f \otimes g  \,\, \D_{\bt_z} (X))\,, 
\end{equation}
$$
f, g \in  \Conv(Q, P), \qquad X \in Q(n)\,,  
$$
where $\bt_z$ is any representative of the 
isomorphism class $z\in  \pi_0\big(\Tree_2(n)\big)$\,.
The axioms of a pseudo-operad (resp. pseudo-cooperad)
imply that the right hand side of \eqref{Conv-bullet} does not 
depend on the choice of representatives $\bt_z$\,.

It follows directly from the definition that 
the operation $\bullet$ is compatible with the 
differential on $\Conv(Q, P)$ coming from 
$Q$ and $P$\,. Furthermore, the operation $\bullet$
satisfies the axiom of the pre-Lie algebra: 
$$
(f \bullet g ) \bullet h - f \bullet (g \bullet h) 
= (-1)^{|g| |h|} (f \bullet h) \bullet g  -  (-1)^{|g| |h|} f \bullet (h \bullet g)\,. 
$$
Therefore the bracket 
$$
[f,g] =  \big( f \bullet g - (-1)^{|f|\, |g|} g \bullet f \big)
$$
satisfies the Jacobi identity, and hence $\Conv(Q, P)$ is a dg Lie algebra. 

We refer to  $\Conv(Q, P)$ as the {\it convolution Lie algebra} of the pair $(Q, P)$\,.

We observe that for every dg operad $\cO$ the morphism (of dg operads) 
\begin{equation}
\label{from-Cobar}
F : \Cobar(\cC) \to \cO 
\end{equation}
is uniquely determined by its restriction
\begin{equation}
\label{F-restric}
F \Big|_{\bs\, \cC_{\c}} ~:~ \bs\, \cC_{\c} \to \cO\,.
\end{equation}
Thus, to every morphism \eqref{from-Cobar} we assign 
a degree $1$ element
\begin{equation}
\label{al-F}
\al_F \in \Conv(\cC_{\c}, \cO)
\end{equation}
such that 
$$
F \Big|_{\bs\, \cC_{\c}} = \al_F \circ \bs^{-1}\,.
$$

The construction of the dg Lie algebra $\Conv(\cC, \cO)$ is partially 
justified by the following statement:  
\begin{thm}[Proposition 5.2, \cite{notes}]
\label{thm:from-Cobar}
Let $\cC$ be a coaugmented dg cooperad
with $\cC_{\c}$ being the cokernel of the coaugmentation.  
Then, for every dg operad $\cO$, the above assignment 
$$
F ~~\mapsto~~ \al_F 
$$
is a bijection between the set of maps (of dg operads)
\eqref{from-Cobar} and Maurer-Cartan elements of the dg Lie algebra 
$$
\Conv(\cC_{\c}, \cO)\,. ~~~ \Box
$$   
\end{thm}

\subsubsection{Example: The dg operad $\La\Lie_{\infty}$}
\label{sec:LaLie-infty}

Let us recall from \cite{GK}, that the dg operad  $\La\Lie_{\infty}$
governing homotopy $\La\Lie$-algebras is 
\begin{equation}
\label{LaLie-infty}
\La\Lie_{\infty} = \Cobar(\La^2 \coCom)\,,
\end{equation}
where $\coCom$ is the cooperad governing cocommutative coalgebras 
without counit. 

Let us denote by $1^{\mc}_n$ the canonical generator
$$
1^{\mc}_n : = \bs^{2-2n}\, 1 \in \bs^{2-2n}\, \bbK \cong  \La^2 \coCom(n)\,.
$$

Using the identification between isomorphism classes of
objects in $\Tree_2(n)$ and the set \eqref{shuffles} of shuffles 
we get 
\begin{equation}
\label{diff-LaLie-infty}
\pa^{\Cobar} (\bs \, 1^{\mc}_n) =  - 
\sum_{p=2}^{n-1} \sum_{\si \in \Sh_{p, n-p}}
\si \big( (\bs\, 1^{\mc}_{n-p+1} ) \circ_1  (\bs\, 1^{\mc}_{p} ) \big)\,, \qquad n \ge 2\,.
\end{equation}

Let us also recall from \cite[Section 5.2]{notes} that the canonical quasi-isomorphism 
of dg operads
\begin{equation}
\label{LaLie-infty-LaLie}
U_{\La\Lie} : \La\Lie_{\infty}  = \Cobar(\La^2 \coCom) ~ \to ~ \La\Lie
\end{equation}
corresponds to the Maurer-Cartan element 
$$
\al_{\La\Lie} = \{a_1, a_2\} \otimes b_1 b_2 \in \Conv(\La^2\coCom_{\c}, \La\Lie) 
\cong  \prod_{n \ge 2} \Big( \La\Lie(n) \otimes \La^{-2}\Com(n) \Big)^{S_n}\,,
$$
where $\{a_1, a_2\}$ (resp. $b_1 b_2$)  denotes the natural generator of $\La\Lie(2)$
(resp. $\La^{-2}\Com(2)$).

In other words, 
$$
U_{\La\Lie} ( \bs 1^{\mc}_n ) = 
\begin{cases}
 \{a_1, a_2\} & {\rm if} ~~ n = 2\,,  \\
 0 & {\rm otherwise}\,.
\end{cases}
$$

\section{Twisting of operads}
\label{sec:twist}
Let $\cO$ be a dg operad equipped with a map
\begin{equation}
\label{from-hoLie}
\wh{\vf} : \La \Lie_{\infty} \to \cO
\end{equation}
and $V$ be an algebra over $\cO$\,.

Using the map $\wh{\vf}$, we equip $V$
with a $\La\Lie_{\infty}$-structure. 
If we assume, in addition, that $V$ is equipped with 
a complete descending filtration 
\begin{equation}
\label{filtr-V}
V \supset \cF_1 V \supset \cF_2 V \supset \cF_3 V \supset \dots \,, 
\qquad  V = \lim_{k} V  \Big / \cF_k V 
\end{equation}
and the $\cO$-algebra structure on $V$ is compatible 
with this filtration then we may define {\it Maurer-Cartan elements } of $V$ as 
degree $2$ elements $\al \in  \cF_1 V $ satisfying the equation
\begin{equation}
\label{MC-eq}
\pa (\al) + \sum_{n \ge 2} \frac{1}{n!} \{\al, \al, \dots, \al\}_n = 0
\end{equation}
where $\pa$ is the differential on $V$ and $\{\cdot, \cdot, \dots, \cdot\}_n$
are the operations of the $\La\Lie_{\infty}$-structure on $V$\,. 
Given such a  Maurer-Cartan element $\al$ we can twist the differential 
on $V$ and insert $\al$ into various $\cO$-operations on $V$\,. 
This way we get a new algebra structure on $V$\,. 
It turns out that this new algebra structure is governed by 
a dg operad $\Tw \cO$ which is built from the pair 
$(\cO, \wh{\vf})$\,. 
This section is devoted to the construction of $\Tw \cO$\,.

First, we recall that $\La\Lie_{\infty} = \Cobar(\La^2 \coCom)$. 
Hence, due to Theorem \ref{thm:from-Cobar},
the morphism \eqref{from-hoLie}
is determined by a Maurer-Cartan element 
\begin{equation}
\label{vf}
\vf \in \Conv(\La^2 \coCom_{\c}, \cO)\,. 
\end{equation}

The $n$-th space of $\La^2 \coCom_{\c}$ is
the trivial $S_{n}$-modules placed in degree $2-2n$: 
$$
\La^2 \coCom (n) = \bs^{2-2n} \bbK\,.
$$
So we have
$$
\Conv(\La^2 \coCom_{\c}, \cO) = 
 \prod_{n \ge 2}  \Hom_{S_n}( \bs^{2-2n}\bbK, \cO(n)) 
 =  \prod_{n \ge 2}   \bs^{2n-2} \big( \cO(n) \big)^{S_n}\,.
$$

For our purposes we will need to extend the dg Lie algebra 
$\Conv(\La^2 \coCom_{\c}, \cO)$ to 
\begin{equation}
\label{cL-cO}
\cL_{\cO} = \Conv(\La^2 \coCom, \cO) =
 \prod_{n \ge 1}  \Hom_{S_n}( \bs^{2-2n}\bbK, \cO(n))\,. 
\end{equation}
It is clear that 
$$
\cL_{\cO}  =   \prod_{n \ge 1}   \bs^{2n-2} \big( \cO(n) \big)^{S_n}\,.
$$

For $n, r \ge 1$ we realize the group $S_r$ 
as the following subgroup of $S_{r+n}$
\begin{equation}
\label{S-r-realize}
S_r \cong \big\{ \si\in S_{r+n} ~|~ \si(i) = i\,, \quad \forall\, i > r \big\}\,.
\end{equation}
In other words, the group $S_r$ may be 
viewed as subgroup of $S_{r+n}$ permuting only the 
first $r$ letters. We set $S_0$ to be the trivial group.
Using this embedding of $S_{r}$ into $S_{n+r}$
we introduce the following collection ($n \ge 0$)
\begin{equation}
\label{pre-Tw-cO}
\wt{\Tw} \cO(n) =  \prod_{r \ge 0}  \Hom_{S_r} ( \bs^{-2r} \bbK, \cO(r+n))\,.
\end{equation}

It is clear that 
\begin{equation}
\label{pre-Tw-cO2}
\wt{\Tw} \cO(n) = \prod_{r \ge 0}  \bs^{2r} \big(\cO(r+n)\big)^{S_r}\,. 
\end{equation}

To define an operad structure on \eqref{pre-Tw-cO} we denote by 
$1_{r}$ the generator 
$$  
1_r : = \bs^{-2r} \, 1\in  \bs^{-2r} \bbK\,. 
$$  
Then 
the identity element $\bu$ in $\wt{\Tw} \cO(1)$ is given by 
\begin{equation}
\label{bu}
\bu(1_r) = \begin{cases}
\bu_{\cO}
  \qquad {\rm if} ~~ r = 0\,, \\
 0 \qquad {\rm otherwise}\,,
\end{cases}
\end{equation}
where $\bu_{\cO}\in \cO(1)$ is the identity element 
for the operad $\cO$\,.
Next, for $f \in \wt{\Tw} \cO(n)$ and $g \in \wt{\Tw} \cO(m)$,
we define the $i$-th elementary insertion $\c_{i}$ 
$1 \le i \le n$ by the formula 
\begin{equation}
\label{circ-i-for-Tw-cO}
f \,\c_i\, g (1_r) = \sum_{p = 0}^{r} 
\sum_{\si \in \Sh_{p, r-p}} \mu_{\bt_{\si,i}} \big( f(1_p) \otimes g (1_{r-p})  \big)\,. 
\end{equation}
where the tree $\bt_{\si,i}$ is depicted on figure \ref{fig:bt-si-i}.
\begin{figure}[htp]
\centering
\begin{tikzpicture}[scale=0.6]
\tikzstyle{w}=[circle, draw, minimum size=3, inner sep=1]
\tikzstyle{b}=[circle, draw, fill, minimum size=3, inner sep=1]
\node[b] (si1) at (0, 2) {};
\draw (0,2.5) node[anchor=center] {{\small $\si(1)$}};
\draw (1,2) node[anchor=center] {{\small $\dots$}};
\node[b] (sip) at (2, 2) {};
\draw (2,2.5) node[anchor=center] {{\small $\si(p)$}};
\node[b] (r1) at (3.5, 2) {};
\draw (3.7,2.5) node[anchor=center] {{\small $r+1$}};
\draw (5,2) node[anchor=center] {{\small $\dots$}};
\node[b] (r1i) at (6.5, 2) {};
\draw (6.4,2.5) node[anchor=center] {{\small $r+i-1$}};
\node[w] (w2) at (8.5, 3) {};
\node[b] (sip1) at (6, 5) {};
\draw (5.8,5.5) node[anchor=center] {{\small $\si(p+1)$}};
\draw (7.3,5) node[anchor=center] {{\small $\dots$}};
\node[b] (sir) at (8, 5) {};
\draw (8,5.5) node[anchor=center] {{\small $\si(r)$}};
\node[b] (ri) at (9.5, 5) {};
\draw (9.5,5.5) node[anchor=center] {{\small $r+i$}};
\draw (10.5,5) node[anchor=center] {{\small $\dots$}};
\node[b] (ri1m) at (12, 5) {};
\draw (12.7,5.5) node[anchor=center] {{\small $r+i+m-1$}};
\node[b] (rim) at (10.5, 2) {};
\draw (10.5,2.5) node[anchor=center] {{\small $r+i+m$}};
\draw (11.7,2) node[anchor=center] {{\small $\dots$}};
\node[b] (rn1m) at (13, 2) {};
\draw (13.8,2.5) node[anchor=center] {{\small $r+n+m-1$}};
\node[w] (w1) at (7.5, 0) {};
\node[b] (r) at (7.5, -1) {};
\draw (w2) edge (sip1);
\draw (w2) edge (sir);
\draw (w2) edge (ri);
\draw (w2) edge (ri1m);
\draw (w1) edge (si1);
\draw (w1) edge (sip);
\draw (w1) edge (r1);
\draw (w1) edge (r1i);
\draw (w1) edge (w2);
\draw (w1) edge (rim);
\draw (w1) edge (rn1m);
\draw (r) edge (w1);
\end{tikzpicture}
\caption{\label{fig:bt-si-i} Here $\si$ is a $(p, r-p)$-shuffle}
\end{figure} 

\begin{remark}
The operad $\wt{\Tw} \cO$ is a completed version of the shifted operad 
$\cO[\bbK]$ introduced in \cite[section 1.2.2]{Fresse98}.
\end{remark}

Sometimes it is convenient to use 
a different but equivalent definition of
the $i$-th elementary insertion $\c_{i}$ for 
$\wt{\Tw} \cO$. This definition is given by
the formula\footnote{Here we use the notation for operadic 
multiplications \eqref{new-notation}.}
\begin{equation}
\label{circ-i-for-Tw-cO1}
f  \circ_i  g (1_r) = 
\end{equation}
$$
\sum_{p = 0}^{r} 
\sum_{\si \in \Sh_{p, r-p}} \si  
\big( f_p(1, \dots, p, r+1, \dots, r+i-1, g_{r-p}(p+1, \dots, r, r+i, \dots, r+i+m-1), 
r+i+m, \dots, r+n +m-1) \big)\,, 
$$
where $f\in \wt{\Tw} \cO(n)$,  $g\in \wt{\Tw} \cO(m)$,
$f_p  = f(1_p) \in \big(\cO(p+n) \big)^{S_p}$ and 
 $g_q  = g(1_q) \in \big(\cO(q+m) \big)^{S_q}$\,.

To see that the element $f \,\c_i\, g (1_r)\in \cO(r+n+m -1)$ is 
$S_r$-invariant one simply needs to use the fact that 
every element $\tau\in S_r$ can 
be uniquely presented as the composition 
$\tau_{sh} \c \tau_{p,r-p} $, where $\tau_{sh}$ is 
a $(p, r-p)$-shuffle and $\tau_{p,r-p} \in S_p \times S_{r-p}$\,.

Let  $f \in \wt{\Tw} \cO(n)$, $g \in \wt{\Tw} \cO(m)$, 
$h \in \wt{\Tw} \cO(k)$, $1\le i \le n$, and $1 \le j \le m$\,.  
To check the identity 
\begin{equation}
\label{fgh-assoc}
f\,\c_i \, (g \,\c_j\, h) =
(f\,\c_i \, g ) \,\c_{j+i-1}\, h 
\end{equation}
we observe that
\begin{align*}
f\,\c_i \, (g \,\c_j\, h) (1_r) &=   \sum_{p = 0}^{r} 
\sum_{\si \in \Sh_{p, r-p}} 
\mu_{\bt_{\si,i}} \big( f(1_p) \otimes   (g \,\c_j\, h) (1_{r-p})  \big) 
\\
&=\sum_{p_1+ p_2+ p_3 =r}\,
\sum_{\si \in \Sh_{p_1, p_2+ p_3} }\,
\sum_{\si' \in  \Sh_{p_2, p_3} }
\mu_{\bt_{\si,i}} \circ
(1 \otimes \mu_{\bt_{\si',j}} )
\big( f(1_{p_1}) \otimes   g(1_{p_2}) \otimes h(1_{p_3}) \big) 
\\
&= \sum_{p_1+ p_2+ p_3 =r}\,
\sum_{\tau \in \Sh_{p_1, p_2, p_3} }
\mu_{\bt_{\tau,i,j}} \big( f(1_{p_1}) \otimes   g(1_{p_2}) \otimes h(1_{p_3}) \big)\,, 
\end{align*}
where the tree $\bt_{\tau,i,j}$ is depicted on figure \ref{fig:bt-tau-ij}.
\begin{figure}[htp]
\centering
\makeatletter
\def\Ddots{\mathinner{\mkern1mu\raise\p@
\vbox{\kern7\p@\hbox{.}}\mkern2mu
\raise4\p@\hbox{.}\mkern2mu\raise7\p@\hbox{.}\mkern1mu}}
\makeatother
\begin{tikzpicture}[
xst/.style={draw, cross out,  minimum size=3, inner sep=1 }, 
int/.style={draw, circle, fill,  minimum size=3, inner sep=1},
ext/.style={draw, circle,  minimum size=3, inner sep=1},
arr/.style={-triangle 60},]
\node[int] (v1) at (4,0) {};
\node[ext] (v2) at (4,1) {};
\node[int, label=90:{\small $\tau(1)$}] (v3) at (0.5,3) {};
\node[int, label=90:{\small $\tau(p_1)$}] (v4) at (2,3) {};
\node[int, label=90:{\small $r+1$}] (v5) at (2.9,3) {};
\node[int, label=90:{\small $r+i-1$}] (v6) at (4.3,3) {};
\node[ext] (v7) at (6,4) {};
\node[int, label=20:{\small $r + i + m  + k -1$}] (v8) at (6.5,3) {};
\node[int, label=90:{\small $r + n + m + k -2$}] (v9) at (8,2) {};
\node[int, label=120:{\small $\tau(p_1+1)$}] (v10) at (4,5) {};
\node[int, label=120:{\small $\tau(p_1+p_2)$}] (v11) at (4.5,5.7) {};
\node[int, label=120:{\small $r+i$}] (v12) at (5.2,6.5) {};
\node[int, label=90:{\small $r + i +j -2$}] (v13) at (6.5,6.5) {};
\node[ext] (v16) at (8,7.5) {};
\node[int, label=30:{\small $r + i + j + k-1 $}] (v14) at (8,6) {};
\node[int, label=30:{\small $r + i + m + k - 2$}] (v15) at (9,5) {};
\node[int, label=120:{\small $\tau(p_1+p_2+1)$}] (v17) at (6,8) {};
\node[int, label=90:{\small $\tau(r)$}] (v18) at (7,9) {};
\node[int, label=90:{\small $r + i + j -1$}] (v19) at (9,9) {};
\node[int, label=30:{\small $r + i + j + k -2$}] (v20) at (10,8) {};

\node at (4.6451,5.014) {$\Ddots$};
\node at (6.9122,8.2302) {$\Ddots$};
\node at (7.914,5.3128) {$\ddots$};
\node at (9.1618,8.4059) {$\ddots$};
\node at (6.235,2.1974) {$\ddots$};
\node at (5.928,5.9279) {$\cdots$};
\node at (3.7,2.5) {$\cdots$};
\node at (2,2.5) {$\cdots$};

\draw (v1) edge (v2);
\draw (v2) edge (v3);
\draw (v2) edge (v4);
\draw (v2) edge (v5);
\draw (v2) edge (v6);
\draw (v2) edge (v7);
\draw (v2) edge (v8);
\draw (v2) edge (v9);
\draw (v7) edge (v10);
\draw (v7) edge (v11);
\draw (v7) edge (v12);
\draw (v7) edge (v13);
\draw (v7) edge (v14);
\draw (v7) edge (v15);
\draw (v7) edge (v16);
\draw (v16) edge (v17);
\draw (v16) edge (v18);
\draw (v16) edge (v19);
\draw (v16) edge (v20);
\end{tikzpicture}
\caption{\label{fig:bt-tau-ij} Here $\tau$ is a $(p_1, p_2, p_3)$-shuffle and 
$r= p_1 + p_2 + p_3$}
\end{figure} 
Similar calculations show that
$$
(f\,\c_i \, g ) \,\c_{j+i-1}\, h  =  \sum_{p_1+ p_2+ p_3 =r}\,
\sum_{\tau \in \Sh_{p_1, p_2, p_3} }
\mu_{\bt_{\tau,i,j}} \big( f(1_{p_1}) \otimes   g(1_{p_2}) \otimes h(1_{p_3}) \big)\,, 
$$
with $\bt_{\tau,i,j}$ being the tree depicted on figure \ref{fig:bt-tau-ij}. 

We leave the verification of the remaining axioms of the operad 
structure for the reader.

Our next goal is to define an auxiliary action of the dg Lie algebra 
$\cL_{\cO}$ on the operad $\wt{\Tw}\cO$\,. 
For a vector $f \in \wt{\Tw}\cO(n)$ the 
action of $v \in \cL_{\cO}$ (\ref{cL-cO}) on $f$ 
is defined by the formula
\begin{equation}
\label{eq:cL-action1}
v \cdot f(1_r) = -(-1)^{|v||f|} \sum_{p=1}^r \, 
\sum_{\si \in \Sh_{p,r-p}} 
\mu_{\bt_{\si, p, r-p}} \big( f(1_{r-p+1}) \otimes v(1^{\mc}_p) \big)\,,
\end{equation}
where $1^{\mc}_p$ is the generator $\bs^{2-2p}\, 1 \in  \La^2 \coCom(p) \cong \bs^{2-2p} \, \bbK$
and the tree $\bt_{\si, p, r-p}$ is depicted on figure 
\ref{fig:bt-si-p-r}. 
\begin{figure}[htp]
\centering
\begin{tikzpicture}[
xst/.style={draw, cross out,  minimum size=3, inner sep=1 }, 
int/.style={draw, circle, fill,  minimum size=3, inner sep=1},
arr/.style={-triangle 60},]

\node[int] (v1) at (3,0) {};
\node[ext] (v2) at (3,1) {};
\node[ext] (v3) at (1,2) {};
\node[int, label=90:{\small $\si(p+1)$}] (v4) at (2,2) {};
\node[int, label=90:{\small $\si(r)$}] (v5) at (3.5,2) {};
\node[int, label=90:{\small $r+1$}] (v8) at (4.5,2.5) {};
\node[int, label=90:{\small $r+n$}] (v9) at (6,2.5) {};
\node[int, label=90:{\small $\si(1)$}] (v6) at (0.5,3) {};
\node[int, label=90:{\small $\si(p)$}] (v7) at (1.5,3) {};

\draw (v1) edge (v2);
\draw (v2) edge (v3);
\draw (v2) edge (v4);
\draw (v2) edge (v5);
\draw (v3) edge (v6);
\draw (v3) edge (v7);
\draw (v2) edge (v8);
\draw (v2) edge (v9);
\node at (4.7,2.1) {$\cdots$};
\node at (2.9,1.7) {$\cdots$};
\node at (1,2.8) {$\cdots$};
\end{tikzpicture}
\caption{\label{fig:bt-si-p-r} Here $\si$ is a $(p, r-p)$-shuffle}
\end{figure} 

We claim that 
\begin{prop}
\label{prop:cL-action1}
Formula \eqref{eq:cL-action1} defines an action of the dg Lie algebra
$\cL_{\cO}$ (\ref{cL-cO}) on the operad $\wt{\Tw}\cO$\,. 
\end{prop}  
\begin{proof}
A simple degree bookkeeping shows that the degree 
of $v\cdot f$ is $|v|+ |f|$\,.

Then we need to check that for two homogeneous vectors  
$v,w\in \cL_{\cO}$ we have  
\begin{equation}
\label{need-action}
[v,w] \cdot f (1_r)  = (v\cdot  (w \cdot f)) (1_r) - 
(-1)^{|v||w|} (w \cdot (v \cdot f))  (1_r) 
\end{equation}

Using the definition of the operation $\cdot$ and 
the associativity axiom for the operad structure on $\cO$
we get
\begin{multline}
\label{rhs-equals}
 (v\cdot  (w \cdot f)) (1_r) - 
(-1)^{|v||w|} (w \cdot (v \cdot f))  (1_r) =  
\\
(-1)^{|f| (|v|+|w|) + |v||w|}
\sum_{p \ge 1 \, q \ge 0} ~ 
\sum_{\tau \in \Sh_{p, q, r- p- q}} ~
\mu_{\bt^{p, q}_{\tau}} (f(1_{r-p-q+1}) \otimes w(1^{\mc}_{q+1}) 
\otimes v(1^{\mc}_{p}))
\\
+(-1)^{|f| (|v|+|w|) + |v||w|}
\sum_{p, q \ge 1} ~ 
\sum_{\tau \in \Sh_{p, q, r- p- q}} ~
\mu_{\wt{\bt}^{p, q}_{\tau}} (f(1_{r-p-q+2}) \otimes w(1^{\mc}_{q}) 
\otimes v(1^{\mc}_{p})) 
\\
-(-1)^{|v| |w|} (v \leftrightarrow w)\,,
\end{multline}
where the trees $\bt^{p, q}_{\tau}$ and $\wt{\bt}^{p, q}_{\tau}$ 
are depicted on figures \ref{fig:bt-tau-p-q}  and \ref{fig:wtbt-tau-p-q} , respectively.
\begin{figure}[htp]
\centering
\begin{tikzpicture}[ 
int/.style={draw, circle, fill,  minimum size=3, inner sep=1},
arr/.style={-triangle 60},]

\node[int] (v1) at (4,0) {};
\node[ext] (v2) at (4,1) {};
\node[ext] (v3) at (2,3) {};
\node[int, label=90:{\small $\tau(p+q+1)$}] (v4) at (3.5,3) {};
\node[int, label=90:{\small $\tau(r)$}] (v5) at (5.5,3) {};
\node[int, label=90:{\small $r+1$}] (v6) at (7,4) {};
\node[int, label=90:{\small $r+n$}] (v7) at (9,4) {};

\node[ext] (v8) at (1,4.5) {};
\node[int, label=90:{\small $\tau(p+1)$}] (v11) at (2,4.5) {};
\node[int, label=90:{\small $\tau(p+q)$}] (v12) at (3.5,4.5) {};
\node[int, label=90:{\small $\tau(1)$}] (v9) at (0.3,5.5) {};
\node[int, label=90:{\small $\tau(p)$}] (v10) at (1.7,5.5) {};

\node at (7.5,3.6) {$\cdots$};
\node at (4.4,2.7) {$\cdots$};
\node at (2.6,4.2) {$\cdots$};
\node at (1.1,5.3) {$\cdots$};

\draw (v1) edge (v2);
\draw (v2) edge (v3);
\draw (v2) edge (v4);
\draw (v2) edge (v5);
\draw (v2) edge (v6);
\draw (v2) edge (v7);
\draw (v8) edge (v9);
\draw (v8) edge (v10);
\draw (v3) edge (v8);
\draw (v3) edge (v11);
\draw (v3) edge (v12);
\end{tikzpicture}
\caption{\label{fig:bt-tau-p-q} The tree  $\bt^{p, q}_{\tau}$}
\end{figure} 
\begin{figure}[htp]
\centering
\begin{tikzpicture}[
xst/.style={draw, cross out, minimum size=5, }, 
int/.style={draw, circle, fill,  minimum size=3, inner sep=1},
arr/.style={-triangle 60},]

\node[int] (v1) at (4,0) {};
\node[ext] (v2) at (4,1) {};
\node[ext] (v3) at (1.5,2) {};
\node[ext] (v6) at (3,3) {};
\node[int , label=90:{\small $\tau(p+q+1)$}] (v9) at (5,3) {};
\node[int , label=90:{\small $\tau(r)$}] (v10) at (6.5,3) {};
\node[int , label=90:{\small $r+1$}] (v11) at (9,4) {};
\node[int, label=90:{\small $r+n$}] (v12) at (11,4) {};
\node[int , label=90:{\small $\tau(1)$}] (v4) at (0.5,3) {};
\node[int , label=90:{\small $\tau(p)$}] (v5) at (2,3) {};
\node[int , label=90:{\small $\tau(p+1)$}] (v7) at (2.2,4) {};
\node[int , label=90:{\small $\tau(p+q)$}] (v8) at (3.8,4) {};
\node at (9.2,3.6) {$\cdots$};
\node at (5.4,2.6) {$\cdots$};
\node at (3.1,3.8) {$\cdots$};
\node at (1.4,2.8) {$\cdots$};
\draw (v1) edge (v2);
\draw (v2) edge (v3);
\draw (v3) edge (v4);
\draw (v3) edge (v5);
\draw (v2) edge (v6);
\draw (v6) edge (v7);
\draw (v6) edge (v8);
\draw (v2) edge (v9);
\draw (v2) edge (v10);
\draw (v2) edge (v11);
\draw (v2) edge (v12);
\end{tikzpicture}
\caption{\label{fig:wtbt-tau-p-q} The tree $\wt{\bt}^{p, q}_{\tau}$}
\end{figure} 

Since $f(1_{r-p-q+2})$ is invariant with respect to the action 
of $S_{r-p-q+2}$ the sums involving $\mu_{\wt{\bt}^{p, q}_{\tau}}$
cancel each other.  
Furthermore, it is not hard to see that the sums involving 
$\mu_{\bt^{p, q}_{\tau}}$ form the expression 
$$
[v,w] \cdot f (1_r)\,.
$$
Thus equation \eqref{need-action} follows. 
It remains to check that the operation 
$f \mapsto v \cdot f$ is an operadic derivation and we 
leave this step as an exercise for the reader. 

\end{proof}

\subsection{The action of \texorpdfstring{$\cL_{\cO}$ on $\wt{\Tw}\cO$}{L(O) on tilde Tw O}}   
   
Let us view   $\wt{\Tw}\cO(1)$ as the Lie algebra with 
the bracket being commutator. 
 
We have an obvious degree zero map 
$$
\ka : \cL_{\cO} \to \wt{\Tw}\cO(1)
$$
defined by the formula:
\begin{equation}
\label{kappa}
\ka(v)(1_{r}) = v(1^{\mc}_{r+1})\,.
\end{equation}
where, as above, $1_r$ is the generator $\bs^{-2r}\, 1 \in \bs^{-2r}\, \bbK$ 
and $1^{\mc}_r$ is the generator $\bs^{2-2r}\, 1 \in \La^2 \coCom(r) \cong \bs^{2-2r}\, \bbK$\,.

We have the following proposition. 
\begin{prop}
\label{prop:Theta}
Let us form the semi-direct product $\cL_{\cO} \ltimes \wt{\Tw}\cO(1)$
of the dg Lie algebras $\cL_{\cO}$ and $\wt{\Tw}\cO(1)$ 
using the action of $\cL_{\cO}$ on $\wt{\Tw}\cO$ defined 
in Proposition \ref{prop:cL-action1}. Then the formula 
\begin{equation}
\label{Theta}
\Te(v) = v + \ka(v)
\end{equation}
defines a Lie algebra homomorphism 
$$
\Te: \cL_{\cO} \to  \cL_{\cO} \ltimes \wt{\Tw}\cO(1)\,.
$$
\end{prop}
\begin{proof}
First, let us prove that 
for every pair of homogeneous vectors 
$v, w \in \cL_{\cO}$ we have 
\begin{equation}
\label{ka-property}
\ka([v,w])
= [\ka(v), \ka(w)] + v \cdot \ka(w) 
- (-1)^{|v||w|} w \cdot \ka(v)\,.
\end{equation}

Indeed, unfolding the definition of $\ka$ we 
get\footnote{Here we use the notation for operadic 
multiplications \eqref{new-notation}.}  
\begin{equation}
\label{ka-bracket} 
\ka([v,w])(1_r) =   \sum_{p = 1}^{r}
\sum_{\tau \in \Sh_{p,r-p}}  v_{r-p+2} \big( w_p(\tau(1), \dots,  \tau(p)), \tau(p+1), 
\dots, \tau(r), r+1 \big) 
\end{equation}
$$
 + \sum_{p = 0}^{r}
\sum_{\tau \in \Sh_{p,r-p}}  v_{r-p+1} \big( w_{p+1}(\tau(1), \dots,  \tau(p), r+1), \tau(p+1), 
\dots, \tau(r) \big) 
$$
$$
- (-1)^{|v||w|} (v \leftrightarrow w)\,,
$$
where $v_{t} = v(1_{t})$ and $w_t = w(1_t)$\,. 
The first sum in \eqref{ka-bracket} equals 
$$
- (-1)^{|v||w|}( w \cdot \ka(v)) (1_r)\,.
$$
Furthermore, since $v(1_{t})$ is invariant under 
the action of $S_t$, we see that the second sum 
in \eqref{ka-bracket} equals
$$
\big(\ka(v) \circ_1  \ka(w) \big)\, (1_r)\,.
$$
Thus equation \eqref{ka-property} holds.  
Now, using  \eqref{ka-property}, it is easy to see that 
$$
[ v + \ka(v),  w + \ka(w)] = [v,w] + v \cdot \ka(w) - 
(-1)^{|v||w|} w \cdot \ka(v) + [\ka(v), \ka(w)]= 
$$
$$
= [v,w] + \ka([v,w])
$$
and the statement of the proposition follows.
\end{proof}

The following corollaries are immediate consequences 
of Propositions \ref{prop:P-1-deriv} and \ref{prop:Theta}.
\begin{corollary}
\label{cor:the-action}
For  $v \in \cL_{\cO}$ and $f \in \wt{\Tw}\cO(n)$ 
the formula
\begin{equation}
\label{the-action}
f \to v \cdot f + \de_{\ka(v)} (f)   
\end{equation}
defines an action of  the Lie algebra $\cL_{\cO}$
on the operad $\wt{\Tw}\cO$\,.  $~~~\Box$
\end{corollary}
\begin{corollary}
\label{cor:MC-MC}
For every  Maurer-Cartan element $\vf \in \cL_{\cO}$, the sum
$$
\vf + \ka(\vf)
$$
is a  Maurer-Cartan element of the Lie algebra 
$ \cL_{\cO} \ltimes \wt{\Tw}\cO(1)$\,. $~~~\Box$
\end{corollary}

We finally give the definition of the operad $\Tw \cO$. 
\begin{defi}
\label{dfn:Tw-cO}
Let $\cO$ be an operad in $\Ch_{\bbK}$ and 
$\vf$ be a Maurer-Cartan element in $\cL_{\cO}$ \eqref{cL-cO} corresponding 
to an operad morphism  $\wh{\vf}$ \eqref{from-hoLie}. 
Let us also denote by $\pa^{\cO}$ the differential on 
$\wt{\Tw}\cO$ coming from the one on $\cO$\,. 
We define the operad $\Tw \cO$ in $\Ch_{\bbK}$ 
by declaring that $\Tw\cO = \wt{\Tw}\cO$ as operads 
in $\grVect_{\bbK}$ and letting 
\begin{equation}
\label{pa-Tw}
\pa^{\Tw} = \pa^{\cO} + \vf \cdot  \,+ \, \de_{\ka(\vf)}
\end{equation}
be the differential on $\Tw\cO$\,.
\end{defi}
Corollaries \ref{cor:the-action} and \ref{cor:MC-MC}
imply that  $\pa^{\Tw}$ is indeed a differential on $\Tw\cO$\,.

\begin{remark}
\label{rem:Tw-cO-0}
It is easy to see that, if $\cO(0)= \bfzero$ then 
the cochain complexes $\bs^{-2}\Tw\cO(0)$ and 
$\cL_{\cO}$ \eqref{cL-cO} are tautologically 
isomorphic. 
\end{remark}

\subsection{Algebras over \texorpdfstring{$\Tw \cO$}{Tw O}}
\label{sec:Tw-cO-alg}

Let us assume that  $V$ is an algebra over $\cO$ equipped with 
a complete decreasing filtration  (\ref{filtr-V}). We also
assume that the $\cO$-algebra structure on $V$ is 
compatible with this filtration. 

Given a Maurer-Cartan element $\al\in \cF_1 V$, the formula
\begin{equation}
\label{tw-diff}
\pa^{\al} (v) = \pa(v) + \sum_{r=1}^{\infty} \frac{1}{r!} \vf(1^{\mc}_{r+1})(\al, \dots, \al, v)
\end{equation}
defines a new (twisted) differential on $V$\,.
We will denote by $V^{\al}$ the cochain complex 
$V$ with this new differential. 
In this setting we have the following theorem:
\begin{thm}
\label{thm:twisting}
If  $V^{\al}$ is the cochain complex obtained 
from $V$ via twisting the differential by the Maurer-Cartan element $\al$
then the formula 
\begin{equation}
\label{twisting}
f(v_1, \dots, v_n) = 
\sum_{r=0}^{\infty} \frac{1}{r!} f(1_r) (\al, \dots, \al, v_1, \dots, v_n)
\end{equation}
$$
f\in \Tw\cO(n)\,,  \qquad v_i \in V
$$
defines a $\Tw\cO$-algebra structure on $V^{\al}$\,. 
\end{thm}     
\begin{proof}
Let $f\in \Tw\cO(n)$, $g \in \Tw\cO(k)$, 
$$
f_r := f(1_r) \in \big(\cO(r+n)\big)^{S_r}\,, 
\qquad \textrm{and} \qquad g_r = g(1_r) \in  \big(\cO(r+k)\big)^{S_r}\,. 
$$  

Our first goal is to verify that 
\begin{equation}
\label{operad-ok}
(-1)^{|g| (|v_{i}| + \dots + |v_{i-1}|) }
f(v_1, \dots, v_{i-1}, g(v_i, \dots, v_{i+k-1}), v_{i+k}, \dots, v_{n+k-1}) 
\end{equation}
$$
= f\, \c_i \, g (v_1, \dots, v_{n+k-1})\,. 
$$

The left hand side of \eqref{operad-ok} can be rewritten 
as 
$$
(-1)^{|g| (|v_{i}| + \dots + |v_{i-1}|) }
f(v_1, \dots, v_{i-1}, g(v_i, \dots, v_{i+k-1}), v_{i+k}, \dots, v_{n+k-1})=  
$$
$$
\sum_{p,q \ge 0} \frac{(-1)^{|g| (|v_{i}| + \dots + |v_{i-1}|) } }{p! q!}
f_p(\al, \dots, \al, v_1, \dots, v_{i-1}, g_{q}(\al, \dots, \al, v_i, \dots, v_{i+k-1}), 
 v_{i+k}, \dots, v_{n+k-1})\,.
$$

Using the obvious combinatorial identity
\begin{equation}
\label{number-shuffles}
|\Sh_{p,q}| = \frac{(p+q)!}{p! q!}
\end{equation}
we rewrite the  left hand side of \eqref{operad-ok}  further
$$
\textrm{L.H.S. of \eqref{operad-ok}} = 
$$
$$
\sum_{p, q \ge 0} \frac{ (-1)^{|g| (|v_{i}| + \dots + |v_{i-1}|) }  }{(p+q)!}
|\Sh_{p,q}| f_p(\al, \dots, \al, v_1, \dots, v_{i-1}, g_{q}(\al, \dots, \al, v_i, \dots, v_{i+k-1}), 
 v_{i+k}, \dots, v_{n+k-1}) = 
$$
$$
\sum_{r=0}^{\infty}
\frac{1}{r!} 
\sum_{p= 0}^r
\sum_{\si \in \Sh_{p, r-p}}
\si \circ \vr_{r,p,i} (f_p \circ_{p+i} g_{r-p})  
(\underbrace{\al, \dots, \al}_{r~\textrm{arguments}}, v_1, \dots, v_{n+k-1}),
$$
where $\vr_{r,p,i}$ is the following permutation in $S_{r + i -1}$ 
\begin{equation}
\label{vr-r-p-i}
 \vr_{r,p,i} = 
\left(
\begin{array}{cccccc}
p+1  & \dots  & p+i-1 & p+i & \dots & r+i-1    \\
r+1  &  \dots  & r+i-1 & p+1 & \dots & r  
\end{array}
\right)\,.
 \end{equation}

Thus, using \eqref{circ-i-for-Tw-cO1}, we get  
$$
\textrm{L.H.S. of \eqref{operad-ok}} = f\, \c_i \, g (v_1, \dots, v_{n+k-1})
$$
and equation  \eqref{operad-ok} holds.

Next, we need to show that 
\begin{multline}
\label{diff-ok}
\pa^{\Tw}(f)(v_1, \dots, v_n)  = \pa^{\al} f (v_1, \dots, v_n) 
\\
- (-1)^{|f|} \sum_{i=1}^n (-1)^{|v_{i}| + \dots + |v_{i-1}|}
 f (v_1, \dots, v_{i-1}, \pa^{\al}(v_i),  v_{i+1}, \dots, v_n)\,.
\end{multline}

The right hand side of  \eqref{diff-ok} can be rewritten as 
\begin{multline*}
\textrm{R.H.S. of \eqref{diff-ok}} =
\sum_{p\ge 0}\frac{1}{p!} \pa f_p(\al, \dots, \al, v_1, \dots, v_n) + 
\sum_{p\ge 0, q\ge 1}\frac{1}{p! q!} \vf_q(\al, \dots, \al, f_p(\al, \dots, \al, v_1, \dots, v_n))
\\
- (-1)^{|f|} \sum_{i=1}^n \sum_{p\ge 0}
 \frac{(-1)^{|v_{i}| + \dots + |v_{i-1}|}}{p!}
 f_p (\al, \dots, \al, v_1, \dots, v_{i-1}, \pa(v_i),  v_{i+1}, \dots, v_n)
\\
- (-1)^{|f|} \sum_{i=1}^n \sum_{p\ge 0, q\ge 1}
 \frac{(-1)^{|v_{i}| + \dots + |v_{i-1}|}}{p! q!}
 f_p (\al, \dots, \al, v_1, \dots, v_{i-1}, \vf_q(\al, \dots, \al, v_i),  v_{i+1}, \dots, v_n)\,,
\end{multline*}

where $f_p = f(1_p)$ and $\vf_q = \vf(1^{\mc}_q)$\,.

Let us now add to and subtract from the
right hand side of  \eqref{diff-ok} the sum 
$$
- (-1)^{|f|}\sum_{p\ge 0}\frac{1}{p!} f_{p+1}(\pa \al, \al, \dots, \al, v_1, \dots, v_n)\,.
$$

Using the symmetry of $f_p = f(1_p)$  with respect to 
the action of the subgroup $S_p \subset S_{p+n}$, we get
$$
\textrm{R.H.S. of \eqref{diff-ok}} =
$$
$$
\sum_{p\ge 0}\frac{1}{p!} \pa f_p(\al, \dots, \al, v_1, \dots, v_n) 
- (-1)^{|f|}\sum_{p\ge 0}\frac{1}{p!} f_{p+1}(\pa \al, \al, \dots, \al, v_1, \dots, v_n)
$$
$$
- (-1)^{|f|} \sum_{i=1}^n \sum_{p\ge 0}
 \frac{(-1)^{|v_{i}| + \dots + |v_{i-1}|}}{p!}
 f_p (\al, \dots, \al, v_1, \dots, v_{i-1}, \pa(v_i),  v_{i+1}, \dots, v_n)
$$
$$
+ (-1)^{|f|} \sum_{p\ge 0}\frac{1}{p!} f_{p+1}(\pa \al, \al, \dots, \al, v_1, \dots, v_n)
$$
$$
+\sum_{p\ge 0, q\ge 1}\frac{1}{p! q!} \vf_q(\al, \dots, \al, f_p(\al, \dots, \al, v_1, \dots, v_n))
$$
$$
- (-1)^{|f|} \sum_{i=1}^n \sum_{p\ge 0, q\ge 1}
 \frac{(-1)^{|v_{i}| + \dots + |v_{i-1}|}}{p! q!}
 f_p (\al, \dots, \al, v_1, \dots, v_{i-1}, \vf_q(\al, \dots, \al, v_i),  v_{i+1}, \dots, v_n) = 
$$
$$
(\pa^{\cO} f) (v_1, \dots, v_n) 
$$
$$
+ (-1)^{|f|} \sum_{p\ge 0}\frac{1}{p!} f_{p+1}(\pa \al, \al, \dots, \al, v_1, \dots, v_n)
$$
$$
+\sum_{p\ge 0, q\ge 1}\frac{1}{p! q!} \vf_q(\al, \dots, \al, f_p(\al, \dots, \al, v_1, \dots, v_n))
$$
$$
- (-1)^{|f|} \sum_{i=1}^n \sum_{p\ge 0, q\ge 1}
 \frac{(-1)^{|v_{i}| + \dots + |v_{i-1}|}}{p! q!}
 f_p (\al, \dots, \al, v_1, \dots, v_{i-1}, \vf_q(\al, \dots, \al, v_i),  v_{i+1}, \dots, v_n)\,. 
$$

Due to the Maurer-Cartan equation for $\al$
$$
\pa (\al) + \frac{1}{q!} \vf_q(\al, \al, \dots, \al) = 0
$$
we have 
\begin{multline*}
+ (-1)^{|f|} \sum_{p\ge 0}\frac{1}{p!} f_{p+1}(\pa \al, \al, \dots, \al, v_1, \dots, v_n)=
\\
-(-1)^{|f|} \sum_{p\ge 0, q \ge 2}\frac{1}{p! q!}  
f_{p+1}( \vf_q(\al, \dots, \al), \al, \dots, \al, v_1, \dots, v_n)\,.
\end{multline*}
Hence, using combinatorial formula \eqref{number-shuffles}, we get 
\begin{multline*}
\textrm{R.H.S. of \eqref{diff-ok}} =
(\pa^{\cO} f) (v_1, \dots, v_n) 
+
(\vf \cdot f )(v_1, \dots, v_n) 
\\
\ka(\vf) \,\c_1\, f  (v_1, \dots, v_n)
- (-1)^{|f|} f \, \c_1 \, \ka(\vf) (v_1, \dots, v_n)\,.
\end{multline*}

Theorem \ref{thm:twisting} is proven.  
\end{proof}

Let us now observe that the dg operad $\Tw\cO$ is equipped with 
a complete descending filtration. Namely,
\begin{equation}
\label{Tw-cO-filtr}
\cF_k \Tw\cO(n) = \{f \in  \Tw\cO(n) ~|~ f(1_r) = 0 \quad \forall ~ r < k \}\,.
\end{equation}
It is clear that the operad structure on $\Tw\cO$ is compatible with 
this filtration.
The endomorphism  operad $\End_V$ also carries 
a complete descending filtration since so does $V$\,.
This observation motivates the following definition:
\begin{defi}
\label{dfn:filtered-TwO-alg}
A \und{filtered  $\Tw\cO$-algebra} is a cochain complex 
$V$ equipped with a complete descending filtration for 
which the operad map 
$$
\Tw\cO \to \End_V
$$
is compatible with the filtrations. We denote by 
$\Alg^{\filtr}_{\Tw \cO}$ the category of filtered $\Tw\cO$-algebras.
\end{defi} 
It is easy to see that the $\Tw\cO$-algebra $V^{\al}$ 
from Theorem  \ref{thm:twisting} 
is a filtered $\Tw\cO$-algebra in the sense of this definition. 
The next theorem provides us with an equivalent description 
of the category of filtered $\Tw\cO$-algebras.
\begin{thm}
\label{thm:twisting-cont}
Let $\Alg^{\MC}_{\cO}$ be the category of pairs $(V, \alpha)$, where $V$ is an $\cO$-algebra equipped with a complete descending filtration as in \eqref{filtr-V} and $\alpha\in \cF_1 V$ is a Maurer-Cartan element.
Morphisms between pairs $(V, \alpha)$ and $(V', \alpha')$ are morphisms of filtered $\cO$-algebras $f:V\to V'$
which satisfy the condition $f(\al) = \al'$\,. 
The category $\Alg^{\filtr}_{\Tw \cO}$of filtered $\Tw\cO$-algebras 
is isomorphic to the category $\Alg^{\MC}_{\cO}$\,.
\end{thm}

\begin{proof}
Theorem \ref{thm:twisting} yields a functor
\begin{equation}
\label{Fun-AlgMC-AlgTw}
\mfF : \Alg^{\MC}_{\cO} \to \Alg^{\filtr}_{\Tw \cO}
\end{equation}
 from  the category $\Alg^{\MC}_{\cO}$
to the category $\Alg^{\filtr}_{\Tw \cO}$ of filtered $\Tw\cO$-algebras. 
The functor $\mfF$ assigns to a pair $(V,\al)$  the 
cochain complex $V^{\al}$ with the differential $\pa^{\al}$ \eqref{tw-diff} and the $\Tw\cO$-algebra 
structure defined by equation \eqref{twisting}.

To define a functor in the opposite direction we produce a degree $2$ element 
$u^{\c} \in \cF_1 \Tw\cO(0)$ using the identity element $u_{\cO} \in \cO(1)$:
\begin{equation}
\label{u-circ}
u^{\c}(1_r) : = 
\begin{cases}
 u_{\cO}   \qquad {\rm if} \quad r=1    \\
 0  \qquad \qquad {\rm otherwise}\,.
\end{cases}  
\end{equation}

Let, as above, $\vf$ be the  Maurer-Cartan element in $\Conv(\La^2\coCom, \cO)$ corresponding 
to the morphism $\hat{\vf}: \La\Lie_{\infty} \to \cO$\,. Also, let 
$\vf_r : = \vf(1^{\mc}_r)$\,. 

Since $S_r$ acts trivially on $\vf_r$, and $\pa^{\cO} (u_{\cO}) = 0$ we get 
$$
\pa^{\Tw} (u^{\c})(1_r) = 
\sum_{\si \in \Sh_{r-1,1} } \si\big( \vf_{r} \circ_r u_{\cO} \big) -
u_{\cO} \circ_1 \vf_{r} = r \vf_r - \vf_r\,.
$$
Thus
\begin{equation}
\label{diff-u-c}
\pa^{\Tw} (u^{\c})(1_r) = (r-1) \vf_{r}\,.
\end{equation}

Let us now consider a filtered $\Tw\cO$-algebra $W$ with the differential $\pa_{W}$\,.
The element $u^{\c} \in  \cF_1 \Tw\cO(0)$ gives us a degree $2$ vector 
in $W$\,. We denote this vector by $\al$ and observe that $\al\in \cF_1 W$
since the map $\Tw\cO \to \End_W$ is compatible with the filtrations. 

Next, we remark that the formula
\begin{equation}
\label{iota-dfn}
\iota(\ga)(1_r) : =
\begin{cases}
\ga  \qquad {\rm if} ~~ r = 0 \\
 0 \qquad {\rm otherwise}\,,
\end{cases}
\qquad \ga \in \cO(n)
\end{equation}
defines an embedding 
\begin{equation}
\label{iota}
\iota: \cO \hookrightarrow \Tw\cO
\end{equation}
of operads in the category $\grVect_{\bbK}$\,.

Thus, using $\iota$ and the $\Tw\cO$-algebra structure on $W$,
we get an $\cO$-algebra structure on $W$\,. Since, in general, 
$\iota$ is not compatible with the differentials, this   $\cO$-algebra structure
on $W$ is not compatible with the differential $\pa_W$ on $W$\,.

Let $r$ be a non-negative integer. Then,
for every vector  
$$
f_r \in \big( \cO(r+n) \big)^{S_r}
$$
the elements $f_r$, $\bs^{2}\, f_r$, $\bs^{4}\, f_r$, $\dots$, 
$\bs^{2r}\, f_r$ may be viewed as vectors in 
$\Tw\cO(r+n)$, $\Tw\cO(r-1+n)$,  $\Tw\cO(r-2+n)$, $\dots$, 
$\Tw\cO(n)$, respectively. Namely, 
\begin{equation}
\label{ladder}
\bs^{2 k}\, f_r (1_p) : = 
\begin{cases}
  f_r \qquad {\rm if} ~~ p = k\,,  \\
 0 \qquad {\rm otherwise}\,.
\end{cases}
\end{equation}

Let us consider the vector 
$ \bs^{2 k} \, f_r \circ_1 u^{\c} \in \Tw \cO(r-k-1+n)$ for 
$0 \le k \le r-1$\,. It is clear that for $p \neq k+1$
$$
\big( \bs^{2 k}\, f_r \circ_1 u^{\c} \big) (1_p) = 0\,.
$$

For $p= k+1$ we use the $S_r$-invariance of $f_r$ 
and get  
$$
\big( \bs^{2 k}\, f_r \circ_1 u^{\c} \big) (1_{k+1}) =
\sum_{\si \in \Sh_{k,1}} \si  (f_r \circ_{k+1} u_{\cO})=
(k+1) f_r \,.
$$
Thus
\begin{equation}
\label{going-down}
 \bs^{2 k} \, f_r \circ_1 u^{\c} = (k+1)\bs^{2k+2} f_r\,. 
\end{equation}

Applying identity \eqref{going-down} $r$ times we get 
\begin{equation}
\label{down-r-times}
( \dots ((f_r \underbrace{\circ_1 u^{\c}) \circ_1 u^{\c}) \dots  \circ_1 u^{\c})}_{r \textrm{ times}} =
r! \bs^{2 r} f_r\,.
\end{equation}

Combining this  observation with the fact that the filtration on $W$ is 
complete, we conclude that for every vector $f \in \Tw\cO(n)$ with 
$f_r = f(1_r) \in \big( \cO(r+n) \big)^{S_r}$ we have 
\begin{equation}
\label{action-on-W}
f(w_1, w_2, \dots, w_n) = 
\sum_{r=0}^{\infty} \frac{1}{r!} \iota(f_r)(\al, \dots, \al, w_1, w_2, \dots, w_n)\,,
\qquad  w_i \in W\,,
\end{equation}
where $\iota$ is the map of operads $\cO \to \Tw\cO$ in $\grVect_{\bbK}$
defined in \eqref{iota-dfn}.

Let us denote by $\pa$ the degree $1$ operation 
$$
\pa : W \to W
$$
defined by the equation 
\begin{equation}
\label{pa-dfn}
\pa(w) := \pa_{W}(w) - \sum_{r=1}^{\infty}\frac{1}{r!} \iota(\vf_{r+1}) (\al, \dots, \al, w)\,,  
\end{equation}
where, as above, $\vf_q = \vf(1^{\mc}_q)$\,.

To prove the identity $\pa^2 = 0$, we observe that,
due to equation \eqref{action-on-W},
\begin{equation}
\label{pa-dfn-equiv}
\pa(w) = \pa_{W}(w) - \ka(\vf)(w)\,.  
\end{equation}

Hence 
$$
\pa^2(w) = - \pa_W (\ka(\vf)(w)) - \ka(\vf)(\pa_W(w))  + 
\big( \ka(\vf) \circ_1 \ka(\vf) \big) (w) =
$$
$$
- \big( \pa^{\Tw} \ka(\vf) \big)(w)  +  \big( \ka(\vf) \circ_1 \ka(\vf) \big) (w)=
$$
$$
- \big( \pa^{\cO} \ka(\vf) \big)(w) - \big( \vf \cdot \ka(\vf) \big)(w)
- 2 \big( \ka(\vf) \circ_1 \ka(\vf) \big) (w)
+  \big( \ka(\vf) \circ_1 \ka(\vf) \big) (w) =
$$
$$
- \big( \pa^{\cO} \ka(\vf)  +  \vf \cdot \ka(\vf) +
 \ka(\vf) \circ_1 \ka(\vf) \big)(w)\,. 
$$

On the other hand, the vector
$ \pa^{\cO} \ka(\vf)  +  \vf \cdot \ka(\vf) +
\ka(\vf) \circ_1 \ka(\vf) \in \Tw\cO(1)$ is zero due to 
Corollary \ref{cor:MC-MC}\,.

Thus equation \eqref{pa-dfn} (or equation \eqref{pa-dfn-equiv}) defines 
another differential on $W$\,.

Let us show that the differential $\pa$ is compatible with 
the $\cO$-algebra structure given by $\iota$ \eqref{iota}. 

Equation  \eqref{pa-dfn-equiv} implies that
for every $\ga \in \cO(n)$ and $w_1,w_2, \dots, w_n \in W$, we have 
\begin{align*}
&\pa \big( \iota(\ga) (w_1, \dots, w_n) \big) - 
\sum_{i=1}^{n}
(-1)^{|\ga|+ |w_1| + \dots + |w_{i-1}|}
\iota(\ga)(w_1, \dots, w_{i-1}, \pa (w_i), w_{i+1}, \dots, w_n)
\\&=
\pa_{W} \big( \iota(\ga) (w_1, \dots, w_n) \big) 
-\sum_{i=1}^{n}
(-1)^{|\ga|+ |w_1| + \dots + |w_{i-1}|}
\iota(\ga)(w_1, \dots, w_{i-1}, \pa_W (w_i), w_{i+1}, \dots, w_n)
\\&\quad\quad
- \ka(\vf)  \big( \iota(\ga) (w_1, \dots, w_n) \big)  
\\&\quad\quad
+ \sum_{i=1}^{n}
(-1)^{|\ga|+ |w_1| + \dots + |w_{i-1}|}
\iota(\ga)(w_1, \dots, w_{i-1}, \ka(\vf)(w_i), w_{i+1}, \dots, w_n)
\\&=
\big( \pa^{\Tw} \iota(\ga) \big)  (w_1, \dots, w_n)
- \big( \ka(\vf) \circ_1  \iota(\ga) \big) (w_1, \dots, w_n) + 
(-1)^{|\ga|}
\sum_{i=1}^{n}
\big( \iota(\ga) \circ_1 \ka(\vf) \big)
(w_1, \dots, w_n) 
\\&=
\big(\iota (\pa^{\cO} \ga) \big)  (w_1, \dots, w_n)\,,
\end{align*}

where, in the last step, we used the fact that 
$\iota(\ga)(1_r) = 0$ for all $r \ge 1$\,.

Thus the differential $\pa$ \eqref{pa-dfn} is indeed compatible 
with the $\cO$-algebra structure on $W$. 

It remains to prove that the vector $\al \in W$ is a  Maurer-Cartan element 
(with respect to the differential $\pa$).

For this purpose we observe that, 
since the map $\Tw\cO \to \End_W$ is compatible with 
the differentials, the vector 
$\pa^{\Tw} u^{\c}$ maps to $\pa_W(\al)$ in $W$\,. 

Therefore, combining equation \eqref{diff-u-c} with equation  
\eqref{action-on-W} we get 
$$
\pa_{W}(\al) = \sum_{r=2}^{\infty}\frac{r-1}{r!} \iota(\vf_r) (\al, \dots, \al)
$$
or equivalently
$$
\pa(\al) +\sum_{r=1}^{\infty}\frac{1}{r!} \iota(\vf_{r+1}) (\al, \dots, \al) = 
 \sum_{r=2}^{\infty}\frac{r-1}{r!} \iota(\vf_r) (\al, \dots, \al)\,.
$$ 

Thus $\al$ indeed satisfies the  Maurer-Cartan equation
\begin{equation}
\label{MC-alpha}
\pa(\al) + \sum_{r=2}^{\infty} \frac{1}{r!} \iota(\vf_r)(\al, \dots, \al) = 0\,.
\end{equation}

Combining the performed work, we conclude that the cochain complex 
$W$ with the differential  \eqref{pa-dfn} and the  Maurer-Cartan element $\al$
is an object of the category $\Alg^{\MC}_{\cO}$\,.

Equations \eqref{action-on-W} and \eqref{pa-dfn} imply that the 
described construction, indeed, gives a (strict) inverse for the 
functor $\mfF$ \eqref{Fun-AlgMC-AlgTw}.
\end{proof}

\subsection{Example: The dg operad $\Tw \La\Lie_{\infty}$}
In this subsection, we apply the twisting procedure to the pair 
$(\La\Lie_{\infty}, \hat{\vf})$, where 
\begin{equation}
\label{id-LaLie}
\hat{\vf}  = \id : \La\Lie_{\infty} \to \La \Lie_{\infty}\,. 
\end{equation}

Let us recall that $1^{\mc}_r$ denotes the canonical generator
$\bs^{2 - 2r}\, 1 \in \La^2 \coCom_{\c}(r)$\,. Thus, 
the  Maurer-Cartan element $\vf$ corresponding to \eqref{id-LaLie} is 
\begin{equation}
\label{vf-La-Lie}
\vf (1^{\mc}_r) = \bs\, 1^{\mc}_r\,,  \qquad r \ge 2
\end{equation}
and the element $\ka(\vf) \in \Tw \La\Lie_{\infty}(1)$ takes the form
\begin{equation}
\label{kappa-vf-La-Lie}
\ka(\vf) (1_r) = \begin{cases}
\bs \, 1^{\mc}_{r+1}  \qquad {\rm if} ~~ r \ge 1 \,, \\
  0 \qquad {\rm if} ~~ r = 0\,.
\end{cases}
\end{equation}

According to general formula \eqref{pa-Tw}, 
the differential  $\pa^{\Tw}$ on $\Tw \La\Lie_{\infty}$
is given by the equation
\begin{equation}
\label{diff-Tw-LaLie-infty}
\pa^{\Tw}(f) (1_r) = \pa^{\Cobar} (f_r) - 
(-1)^{|f|} \sum_{p=2}^{r} 
\sum_{\si \in \Sh_{p, r-p}} \si
\big(
f_{r-p+1} \circ_1  (\bs\, 1^{\mc}_p)
\big)
\end{equation}
$$
 +
\sum_{p=1}^{r} \sum_{\si \in \Sh_{p, r-p}} \si
\big(
 (\bs \, 1^{\mc}_{p+1}) \circ_{p+1}  f_{r-p}
\big) -
(-1)^{|f|} \sum_{i=1}^n \sum_{p=1}^{r-1} 
\sum_{\si \in \Sh_{p, r-p}} \si \circ \vr_{r,p,i}
\big(
f_p \circ_{p+i}  (\bs\, 1^{\mc}_{r-p+1})
\big)\,.
$$
where $f \in \Tw \La\Lie_{\infty} (n)$, $f_r = f(1_r)$ and 
$\vr_{r,p,i}$ is the permutation of $S_{r+i-1}$ defined in \eqref{vr-r-p-i}.

We will need the following lemma:
\begin{lemma}
\label{lem:laliecoalg}
If $\Tw \La \Lie_{\infty}$ is the dg operad which is obtained 
from $\La\Lie_{\infty}$ via applying the twisting procedure 
to the identity map \eqref{id-LaLie} then the equation
\begin{equation}
\label{dfn-LaLieinfty-to-Tw}
\mT (\bs 1^{\mc}_n) (1_r)  = \bs 1^{\mc}_{r + n} 
\end{equation}
defines a map of dg operads
\begin{equation}
\label{LaLieinfty-to-Tw}
\mT : \La\Lie_{\infty} ~ \to ~ \Tw \La \Lie_{\infty}\,.
\end{equation}
\end{lemma}
%

\begin{proof}
According to  \eqref{LaLie-infty} 
$$
\La\Lie_{\infty} = \Cobar(\La^2 \coCom)\,. 
$$
Hence, due to Theorem \ref{thm:from-Cobar}, maps of 
dg operads from $\La\Lie_{\infty}$ to $\Tw\La\Lie_{\infty}$ are 
in bijection with  Maurer-Cartan elements of the convolution Lie algebra 
\begin{equation}
\label{Conv-coCom-TwLaLie-infty}
\Conv(\La^2 \coCom_{\c}, \Tw\La\Lie_{\infty}) =
\prod_{n \ge 2} \Hom_{S_n} \big( \bs^{2-2n} \bbK ,  \Tw\La\Lie_{\infty}(n) \big)\,.
\end{equation}

Let us denote by $\al_{\mT}$ a vector in \eqref{Conv-coCom-TwLaLie-infty}
defined by the equation:
\begin{equation}
\label{al-mT}
\al_{\mT} (1^{\mc}_n)(1_r) = \bs 1^{\mc}_{r+n}\,.
\end{equation}
Since the vectors $1^{\mc}_n$ and $1_n$ carry degrees
$$
|1^{\mc}_n | = 2 - 2n\,, \qquad  |1_n | = - 2n\,,
$$
the vector $\al_{\mT}$ has degree $1$\,.

Using formula \eqref{diff-Tw-LaLie-infty} for the differential 
$\pa^{\Tw}$ on $\Tw\La\Lie_{\infty}$ we get (here $n \ge 2$)
\begin{multline}
\label{diff-Tw-al-mT}
\big( \pa^{\Tw} \al_{\mT} (1^{\mc}_n) \big) (1_r) =
\pa^{\Cobar} (  \bs 1^{\mc}_{r+n} ) + 
\sum_{p=2}^{r} \sum_{\si \in \Sh_{p, r-p}} 
\si \big(  \bs 1^{\mc}_{r-p+1+n} \circ_1 \bs 1^{\mc}_p \big) +
\\
\sum_{p=1}^r \sum_{\si \in \Sh_{p, r-p}}
\si\big(  \bs 1^{\mc}_{p+1} \circ_{p+1}   \bs 1^{\mc}_{r-p+n} \big) +
\sum_{i=1}^n \sum_{p=1}^{r-1} 
\sum_{\si \in \Sh_{p, r-p}}
\si \circ \vr_{r,p,i} \big( \bs 1^{\mc}_{p+n} \circ_{p+i}  \bs 1^{\mc}_{r -p+1}  \big)\,,
\end{multline}
where $\vr_{r,p,i}$ is the permutation defined in \eqref{vr-r-p-i}.

Unfolding $\pa^{\Cobar} (  \bs 1^{\mc}_{r+n} ) $, we get 
\begin{multline}
\label{diff-Tw-al-mT-unfold}
\big( \pa^{\Tw} \al_{\mT} (1^{\mc}_n) \big) (1_r) =
- \sum_{q=2}^{r+n-1} \sum_{\la \in \Sh_{q, r+n-q} }
\la \big( \bs 1^{\mc}_{r+n-q+1} \circ_1 \bs 1^{\mc}_{q}  \big) + 
\sum_{p=2}^{r} \sum_{\si \in \Sh_{p, r-p}} 
\si \big(  \bs 1^{\mc}_{r-p+1+n} \circ_1 \bs 1^{\mc}_p \big) +
\\
\sum_{p=1}^r \sum_{\si \in \Sh_{p, r-p}}
\si\big(  \bs 1^{\mc}_{p+1} \circ_{p+1}   \bs 1^{\mc}_{r-p+n} \big) +
\sum_{i=1}^n \sum_{p=1}^{r-1} 
\sum_{\si \in \Sh_{p, r-p}}
\si \circ \vr_{r,p,i} \big( \bs 1^{\mc}_{p+n} \circ_{p+i}  \bs 1^{\mc}_{r -p+1}  \big)\,.
\end{multline}

Next, using the definition of the binary operation \eqref{Conv-bullet} we get 
\begin{equation}
\label{almT-almT}
\begin{aligned}
\big( \al_{\mT} \bullet  \al_{\mT} (1^{\mc}_n) \big) (1_r) 
&=
\sum_{q = 2}^{n-1} \sum_{\tau \in \Sh_{q, n-q}}
\Big( \tau \big( \al_{\mT} (1^{\mc}_{n-q+1}) \circ_1  \al_{\mT} (1^{\mc}_{q}) \big) \Big) (1_r) 
\\&=
\sum_{0 \le p \le r}^{2 \le q \le n-1}
\sum_{\si \in \Sh_{p, r-p, q, n-q}}
 \si \big(\, \bs 1^{\mc}_{n+p-q+1} \circ_{p+1}
  \bs 1^{\mc}_{r-p+q} \, \big) \,.
\end{aligned}
\end{equation}

We observe that the second sum in the right hand side of \eqref{diff-Tw-al-mT-unfold} 
is obtained from 
\begin{equation}
\label{master-sum}
 \sum_{q=2}^{r+n-1} \sum_{\la \in \Sh_{q, r+n-q} }
\la \big( \bs 1^{\mc}_{r+n-q+1} \circ_1 \bs 1^{\mc}_{q}  \big)
\end{equation}
by keeping only the terms for which 
$$
\la(r+1) = r+1, \quad \la(r+2) = r+2, \quad \dots, \quad \la(r+n) = r+n\,.
$$

Next, using the fact that $S_m$ acts trivially on 
the vector $1^{\mc}_m$,  we see that the third sum 
in the right hand side of \eqref{diff-Tw-al-mT-unfold} 
is obtained from \eqref{master-sum} by keeping only the terms 
for which $q \ge n$ and 
$$
\la(q-n+1) = r+1\,, \quad \la(q-n+2) = r+2\,, \quad \dots, \quad \la(q)= r+n\,.
$$

Similarly, the last sum  in the right hand side of \eqref{diff-Tw-al-mT-unfold} 
is obtained from the sum \eqref{master-sum} by keeping only the terms
for which 
$$
\la(q) \ge r+1 \qquad \textrm{and} \qquad
\la(i) \le r \quad \forall ~~ 1 \le i \le q-1\,. 
$$

Finally, using the invariance with respect to the action of the symmetric group
once again, we see that $\big( \al_{\mT} \bullet  \al_{\mT} (1^{\mc}_n) \big) (1_r)$ 
is obtained from the sum  \eqref{master-sum} by keeping only the terms
for which 
$$
\la(q-1) \ge r+1\,, \qquad \textrm{and} \qquad \la(r+n) \ge r+1\,.
$$

Thus, the vector $\al_{\mT}$ satisfies the Maurer-Cartan equation
\begin{equation}
\label{MC-al-mT}
 \pa^{\Tw} \al_{\mT} + \al_{\mT} \bullet \al_{\mT} = 0\,.
\end{equation}

Since the morphism of dg operads 
$$
\mT :  \La\Lie_{\infty} ~ \to ~ \Tw \La \Lie_{\infty}
$$
corresponding to the  Maurer-Cartan element  $\al_{\mT}$
is defined by equation \eqref{dfn-LaLieinfty-to-Tw}, 
the lemma is proved.  
\end{proof}
\begin{remark}
\label{rem:Tw-LaLie-infty}
Let $V$ be a  $\La\Lie_{\infty}$-algebra with a differential $\pa$
and the brackets 
$$
\{~,~, \dots, ~ \}_n : V^{\otimes \, n} \to V\,.
$$
Let us assume that $V$ is equipped with a complete 
descending filtration \eqref{filtr-V} and $\al \in \cF_1 V$ is a Maurer-Cartan element.
Theorem \ref{thm:twisting} and Lemma \ref{lem:laliecoalg} imply that the formula
\begin{equation}
\label{twisted-brackets}
\{v_1, \dots, v_n\}^{\al}_n = 
\sum_{r=0}^{\infty} \frac{1}{r!} 
\{\underbrace{\al, \dots, \al}_{r ~\textrm{times}}, v_1, \dots, v_n\}_{r+n}
\end{equation}
defines a $\La\Lie_{\infty}$-structure on $V$ with the differential 
given by equation \eqref{tw-diff}. We say that this latter 
$\La\Lie_{\infty}$-structure on $V$ is obtained from the former 
one via \und{twisting by the Maurer-Cartan element} $\al$\,. 
\end{remark}

\subsection{A useful modification $\Tw^{\oplus}\cO$}
\label{sec:Tw-oplus}
In practice the morphism \eqref{from-hoLie} often
comes from the map (of dg operads)
$$
\mi : \La\Lie \to \cO\,.
$$

In this case, the above construction of twisting 
is well defined for the suboperad $\Tw^{\oplus}(\cO)\subset \Tw\cO$ with 
\begin{equation}
\label{Tw-oplus}
\Tw^{\oplus}(\cO)(n) = \bigoplus_{r \ge 0} \bs^{2r} \big( \cO(r+n) \big)^{S_r}\,. 
\end{equation}

It is not hard to see that the Maurer-Cartan element 
$$
\vf \in \Conv(\La^2 \coCom, \cO)
$$
corresponding to the composition 
$$
\mi \circ U_{\La\Lie} : \Cobar(\La^2 \coCom) \to \cO
$$
is given by the formula:
\begin{equation}
\label{al-coCom-cO}
\vf(1^{\mc}_r) = \begin{cases}
  \mi (\{a_1, a_2\}) \qquad {\rm if} ~~ r =2 \\
  0  \qquad {\rm otherwise}\,.
\end{cases}
\end{equation}

Hence
\begin{equation}
\label{cL-oplus}
\cL^{\oplus}_{\cO} =  \bigoplus_{r \ge 0} \bs^{2r-2} \big( \cO(r) \big)^{S_r}
\end{equation}
is a sub- dg Lie algebra of $\cL_{\cO}$ \eqref{cL-cO}\,.

Specifying general formula \eqref{pa-Tw} to this particular 
case, we see that  the differential $\pa^{\Tw}$ on \eqref{Tw-oplus}
is given by the equation:
\begin{equation}
\label{diff-Tw-cO}
\pa^{\Tw} (v) = - (-1)^{|v|} \sum_{\si \in \Sh_{2, r-1}}
\si \big(v \circ_1 \mi (\{a_1,a_2\}) \big) + 
\sum_{\tau \in \Sh_{1, r}} \tau \big( \mi(\{a_1, a_2\}) \circ_2  v \big)
\end{equation}
$$
- (-1)^{|v|}  \sum_{\tau' \in \Sh_{r, 1}}  \sum_{i=1}^{n}  
\tau'  \circ \vs_{r+1,r+i} \big( v \circ_{r+i} \mi(\{a_1,a_2\}) \big)\,,
$$
where 
$$
v \in  \bs^{2r} \big( \cO(r+n) \big)^{S_r}\,,
$$
and $\vs_{r+1,r+i}$ is the cycle $(r+1, r+2, \dots, r+i)$\,.

\begin{remark}
\label{rem:diff-Tw}
We should remark that, when we apply elementary insertions
in the right hand side of \eqref{diff-Tw-cO}, we view
$v$ and $\mi(\{a_1, a_2\})$ as vectors in $\cO(r+n)$ and 
$\cO(2)$ respectively. The resulting sum in the  right hand 
side of \eqref{diff-Tw-cO} is viewed as a vector in $\Tw\cO(n)$\,.  
\end{remark}

\subsection{More general version of twisting}
\label{sec:Tw-general}
The reader may object that the version of twisting we have presented so far does not even cover the most classical case of the operad $\cO = \Lie$, since in \eqref{from-hoLie} we required a map $\La\Lie_\infty \to \cO$. However, this is easily repaired:

Let $k$ be an integer and let
\[
\La^k\Lie_\infty \to \cO
\]
be an operad map. Then, by functoriality of $\La$, there is an operad map 
\[
\La\Lie_\infty \to \La^{1-k}\cO.
\]
We then define: 
\[
\Tw \cO := \La^{k-1} \Tw (\La^{1-k}\cO).
\]
In this paper we will always assume $k=1$ since (i) this is the relevant case for the Deligne conjecture 
and (ii) the signs are significantly simpler for odd $k$. 
One should however be aware that all important statements may be transcribed to the arbitrary-$k$ case by application of the functor $\La$.

\section{Categorial properties of twisting}
\label{sec:categorial}

\subsection{\texorpdfstring{$\Tw$}{Tw} as a comonad}

The goal of this section is to show that the operation $\Tw$ defines a comonad 
 on the under-category $\La \Lie_{\infty}\downarrow \Operads$, where $\Operads$ is the category of dg operads. Recall that the under-category $\La \Lie_{\infty}\downarrow \Operads$ is the category of arrows $\La \Lie_{\infty}\to \cO$ in $\Operads$, with morphisms the commutative diagrams
 \[
 \begin{tikzpicture}
\matrix (m) [matrix of math nodes, row sep=2em]
{  & \La \Lie_{\infty} &  \\
\cO &  & \cO' \\ };
\path[->,font=\scriptsize]
(m-1-2) edge (m-2-1) edge (m-2-3)
(m-2-1) edge  (m-2-3);
\end{tikzpicture}\, .
 \]
 
\subsubsection{\texorpdfstring{$\Tw$}{Tw} is an endofunctor}
Consider arrows $\La \Lie_{\infty} \to \cO \stackrel{f}{\to} \cO'$ in the category $\Operads$. 
Twisting  $\cO$ and $\cO'$, we obtain new dg operads with
\begin{align*}
\Tw \cO(n) &= \prod_{r \geq 0}  \Hom_{S_r} ( \bs^{-2r} \bbK, \cO(r+n))\,, \\
\Tw \cO'(n) &= \prod_{r \geq 0}  \Hom_{S_r} ( \bs^{-2r} \bbK, \cO'(r+n))\,.
\end{align*}
By composing morphisms on the right with (components of) the map $f:\cO \to \cO'$ 
we obtain a map (of collections)
\[
\Tw f \colon \Tw \cO \to \Tw \cO'\,.
\]
\begin{lemma}
\label{lem:Twfdef}
The above map $\Tw f : \Tw \cO \to \Tw \cO'$ is a map of dg operads.
Furthermore, if 
\[
\La \Lie_{\infty} \to \cO \stackrel{f}{\to} \cO'\stackrel{g}{\to} \cO''
\]
are maps of dg operads, then $\Tw (f \circ g)=(\Tw f)\circ (\Tw g)$.
\end{lemma} 
\begin{proof}
It is clear since we only used natural operations (i.e., the operad structure) in defining $\Tw \cO$.
\end{proof}

Applying $\Tw$ to $\hat{\vf} : \La \Lie_{\infty} \to \cO$, we obtain a morphism
of dg operads
\begin{equation}
\label{TwLinf-TwcO}
\Tw(\hat{\vf}) : \Tw \La \Lie_{\infty} \to \Tw \cO\,.
\end{equation}

Composing \eqref{TwLinf-TwcO} with the canonical map 
of dg operads 
$$
\mT :  \La\Lie_{\infty} \to \Tw \La \Lie_{\infty}
$$
from  Lemma \ref{lem:laliecoalg}, we obtain a morphism 
of dg operads
\begin{equation}
\label{beta-cO}
\beta_{\cO} : = \Tw(\hat{\vf}) \circ  \mT  :  \La \Lie_{\infty} \to \Tw \cO\,.
\end{equation}

From Lemma \ref{lem:Twfdef} it is then evident that the following diagram commutes:
\[
\begin{tikzpicture}
\matrix (m) [matrix of math nodes, row sep=2em]
{  & \La \Lie_{\infty} &  \\
\Tw \cO &  & \Tw \cO' \\ };
\path[->,font=\scriptsize]
(m-1-2) edge node[above] {$ \beta_{\cO}~~~~~ $} (m-2-1) edge  node[auto] {$ \beta_{\cO'} $} (m-2-3)
(m-2-1) edge node[auto] {$ \Tw f $} (m-2-3);
\end{tikzpicture}
\]
Here $f, \cO, \cO'$ are as above. Summarizing, we obtain the following result:

\begin{corollary}
\label{cor:Tw-endofun}
The operation $\Tw$ defines an endofunctor on the under-category 
$\La \Lie_{\infty}\downarrow \Operads$. ~~$\Box$
\end{corollary}

\begin{remark}
\label{rem:TwO-augmented}
Let $\ast$ be the initial object \eqref{ast} in the category of dg operads. 
It is easy to see that $\Tw \ast$ carries the obvious augmentation 
\begin{equation}
\label{aug-Tw-ast}
\Tw \ast \to \ast\,.
\end{equation}
Thus, if the dg operad $\cO$ in an object of $\La \Lie_{\infty}\downarrow \Operads$
has an augmentation morphism $\ve : \cO \to \ast$ then 
$\Tw \cO$ is canonically augmented. 
The desired augmentation $\Tw\cO \to \ast$ is obtained by composing 
$$
\Tw(\ve) : \Tw \cO \to \Tw \ast
$$
with \eqref{aug-Tw-ast}.
\end{remark}

\subsubsection{The natural projection}
Let $\La \Lie_{\infty} \to \cO$ be an arrow in $\Operads$.
There is the natural map
\begin{equation}
\label{eta}
\eta_\cO : \Tw \cO \to \cO
\end{equation}
projecting the product 
\[
\Tw \cO(n) = \prod_{r \geq 0}  \Hom_{S_r} ( \bs^{-2r} \bbK, \cO(r+n))
\]
to its first ($r=0$) factor.
It is easy to see that \eqref{eta} is a map of dg operads.

We claim that
\begin{lemma}
The maps $\eta_\cO$ assemble to form a natural transformation $\eta: \Tw\Rightarrow \mathit{id}$, where $\mathit{id}$ is the identity functor on $\La \Lie_{\infty}\downarrow \Operads$.
\end{lemma}
\begin{proof}
We have to show that for all arrows $\La \Lie_{\infty} \to \cO \stackrel{f}{\longrightarrow} \cO'$ the following diagram commutes.
\[
\begin{tikzpicture}
\matrix (m) [matrix of math nodes, row sep=2em, column sep=2em]
{\Tw \cO  &\Tw \cO' \\
\cO & \cO' \\ };
\path[->,font=\scriptsize]
(m-1-1) edge node[auto] {$\Tw f$} (m-1-2)
        edge node[left] {$\eta_\cO$} (m-2-1) 
(m-2-1) edge node[auto] {$f$} (m-2-2)
(m-1-2) edge node[auto] {$\eta_{\cO'}$} (m-2-2);
\end{tikzpicture}
\]
This is obvious.
\end{proof}

\subsubsection{The comultiplication  $\mD : \Tw \to \Tw \circ \Tw$}

Let again $\La \Lie_{\infty} \to \cO$ be an arrow in $\Operads$.
Consider 
\[
\Tw \Tw \cO = \prod_{r,s \geq 0}  \Hom_{S_r}( \bs^{-2r} \bbK, \Hom_{S_s} ( \bs^{-2s} \bbK, \cO(s+r+n)))
\cong 
\prod_{r,s \geq 0}  \Hom_{S_s\times S_r}( \bs^{-2s-2r} \bbK, \cO(s+r+n))
.
\]
There are natural inclusions 
\[
\Hom_{S_{s+r} }( \bs^{-2s-2r} \bbK, \cO(s+r+n))
\to
\Hom_{S_s\times S_r}( \bs^{-2s-2r} \bbK, \cO(s+r+n))
\]
and they assemble to form a map
\begin{equation}
\label{mD-cO}
\mD_\cO \colon \Tw \cO \to \Tw \Tw \cO,
\end{equation}
which is alternatively defined by the equation: 
\begin{equation}
\label{mD-cO-dfn}
\mD_{\cO}(f)(1_r \otimes 1_s) = f(1_{s+r})\,, 
\qquad f \in \Tw \cO(n)\,.
\end{equation}

We claim that
\begin{lemma}
\label{lem:mDcO-oper-map}
The map $\mD_\cO$ is a map of operads and the diagram
\begin{equation}
\label{Linf-Tw-TwTw}
\begin{tikzpicture}
\matrix (m) [matrix of math nodes, row sep=2em]
{  & \La \Lie_{\infty} &  \\
\Tw \cO &  & \Tw \Tw \cO \\ };
\path[->,font=\scriptsize]
(m-1-2) edge  node[above] {$ \beta_{\cO}~~ $} (m-2-1) edge  node[auto] {$ \beta_{\Tw\cO} $} (m-2-3)
(m-2-1) edge node[auto] {$\mD_\cO$} (m-2-3);
\end{tikzpicture}
\end{equation}
commutes.
\end{lemma}
\begin{proof}
Let us consider commutativity of diagram \eqref{Linf-Tw-TwTw} first.

By definition of $\beta_{\cO}$ \eqref{beta-cO}, we have 
$$
\beta_{\cO}( \bs 1^{\mc}_n) (1_r) =  \vf_{r+n}\,,
$$
where $\vf_n : = \hat{\vf}(\bs\, 1^{\mc}_n)$\,.

On the other hand,
\begin{equation}
\label{beta-TwcO-dfn}
\beta_{\Tw \cO} = \Tw(\beta_{\cO}) \circ \mT\,.
\end{equation}
Hence
\begin{equation}
\label{beta-TwcO}
\beta_{\Tw \cO} (\bs\, 1^{\mc}_n) (1_r \otimes 1_s) =  \beta_{\cO} (\bs\, 1^{\mc}_{r+n})(1_s) 
= \vf_{s+r+n}\,.
\end{equation}

Composing $\beta_{\cO}$ with $\mD_{\cO}$, we get 
$$
\mD_{\cO} \circ \beta_{\cO}( \bs 1^{\mc}_n) (1_{r} \otimes 1_{s}) =
 \beta_{\cO}( \bs 1^{\mc}_n) (1_{r+s}) =  \vf_{s+r+n}\,.
$$
Comparing this result with \eqref{beta-TwcO} we conclude that 
diagram \eqref{Linf-Tw-TwTw} indeed commutes.  

Let us now show that $\mD_\cO$ is a map of operads.

For this purpose we consider a pair of vectors $f \in \Tw \cO(n)$ 
and $g \in \Tw\cO(m)$ and compute
\begin{multline}
\label{mD-f-i-g}
\mD_{\cO}(f \circ_i g)(1_r \otimes 1_s) = f \circ_i g(1_{r+s}) = 
\sum_{p=0}^{r+s-p} \sum_{\si \in \Sh_{p, r+s-p}}
\si \circ \vr_{r+s,p,i} \big(
f(1_p) \circ_{p+i} g(1_{r+s-p}) \big)
\end{multline}
where the family of permutations $\{ \vr_{r,p,i} \}_{p \le r}$ is defined
in \eqref{vr-r-p-i}\,.

On the other hand, 
\begin{equation}
\label{mD-f-i-mD-g}
\begin{split}
\Big( \mD_{\cO}(f) \circ_i \mD_{\cO}(g) \Big)(1_r \otimes 1_s) &=
\sum_{p_1=0}^{r} 
\sum_{\tau\in \Sh_{p_1, r-p_1}}
\tau \circ \vr_{r,p_1,i} \big(
\mD_{\cO}(f)(1_{p_1}) \circ_{p_1 + i}
\mD_{\cO}(g)(1_{r-p_1}) \big)\, (1_s)
\\ &=
\sum_{ \substack{ 0 \le p_1 \le r \\ 0 \le p_2 \le s} } \, 
\sum_{\si \in T_{r, s, p_1,p_2}}\,
\mu_{\bt_{\si} }
\Big(
\mD_{\cO}(f)(1_{p_1}\otimes 1_{p_2}) \otimes \mD_{\cO}(g)(1_{r-p_1} \otimes 1_{s-p_2})
\Big)
\\ &=
\sum_{ \substack{ 0 \le p_1 \le r \\ 0 \le p_2 \le s}} \, 
\sum_{\si \in T_{r, s, p_1,p_2}}\,
\mu_{\bt_{\si} }
\Big(
f(1_{p_1+p_2}) \otimes g(1_{r+s-p_1-p_2})
\Big)\,,
\end{split}
\end{equation}
where $\bt_{\si}$ is the labeled planar tree depicted on figure \ref{fig:bt-si} 
and the set $ T_{r, s, p_1,p_2}$ consists of shuffles $\si \in \Sh_{p_1+p_2, r+s-p_1-p_2}$ 
satisfying the conditions:
\begin{gather*}
\si(1), \, \si(2) ,\, \dots,\, \si(p_2)~ \in~  \{1,2,\dots, s\}\,, 
\\
\si(p_2+p_1+1) ,\, \si(p_2+p_1+2) ,\, \dots ,\, \si(p_1+s) ~\in~  \{1,2,\dots, s\}\,, 
\\
\si(p_2+1),\, \si(p_2+2),\, \dots ,\, \si(p_2+p_1) ~ \in~  \{s+1, s+2,\dots, s+r\}\,, 
\\
\si(p_1+s+1) ,\, \si(p_1+s+2) ,\, \dots ,\, \si(s+r) ~ \in ~  \{s+1, s+2,\dots, s+r\}\,.
\end{gather*}
\begin{figure}[htp] 
\centering 
\begin{tikzpicture}[scale=0.5]
\tikzstyle{w}=[circle, draw, minimum size=3, inner sep=1]
\tikzstyle{b}=[circle, draw, fill, minimum size=3, inner sep=1]
\node[b] (r) at (1, -2) {};
\node[w] (v1) at (1, 0) {};
\node[b] (l1) at (-12, 4) {};
\draw (-12,4.6) node[anchor=center] {{\small $\si(1)$}};
\node[b] (l2) at (-10, 4) {};
\draw (-10,4.6) node[anchor=center] {{\small $\si(2)$}};
\draw (-8,4) node[anchor=center] {{\small $\dots$}};
\node[b] (lp1p2) at (-6, 4) {};
\draw (-6,4.6) node[anchor=center] {{\small $\si(p_2+p_1)$}};
\node[b] (lrs1) at (-2.5, 4) {};
\draw (-2.5,4.6) node[anchor=center] {{\small $s+r+1$}};
\draw (-0.5,4) node[anchor=center] {{\small $\dots$}};
\node[b] (lrs1i) at (1, 4) {};
\draw (1.2,4.6) node[anchor=center] {{\small $s+r+i-1$}};
\node[w] (v2) at (4.5, 6) {};
\node[b] (lp1p21) at (-2, 10) {};
\draw (-2,10.6) node[anchor=center] {{\small $\si(p_2+p_1+1)$}};
\draw (0,10) node[anchor=center] {{\small $\dots$}};
\node[b] (lrs) at (2, 10) {};
\draw (2,10.6) node[anchor=center] {{\small $\si(s+r)$}};
\node[b] (lrsi) at (5, 10) {};
\draw (5,10.6) node[anchor=center] {{\small $s+r+i$}};
\draw (7,10) node[anchor=center] {{\small $\dots$}};
\node[b] (lrsi1m) at (10, 10) {};
\draw (10,10.6) node[anchor=center] {{\small $s+r+i+m-1$}};
\node[b] (lrsim) at (6, 4) {};
\draw (6,4.6) node[anchor=center] {{\small $s+r+i+m$}};
\draw (9,4) node[anchor=center] {{\small $\dots$}};
\node[b] (lrsn1m) at (12, 4) {};
\draw (12.3,4.6) node[anchor=center] {{\small $s+r+n+m-1$}};
\draw (r) edge (v1);
\draw (v1) edge (l1);
\draw (v1) edge (l2);
\draw (v1) edge (lp1p2);
\draw (v1) edge (lrs1);
\draw (v1) edge (lrs1i);
\draw (v1) edge (v2);
\draw (v1) edge (lrsim);
\draw (v1) edge (lrsn1m);
\draw (v2) edge (lp1p21);
\draw (v2) edge (lrs);
\draw (v2) edge (lrsi);
\draw (v2) edge (lrsi1m);
\end{tikzpicture}
\caption{The labeled planar tree $\bt_{\si}$} \label{fig:bt-si}
\end{figure}

It is clear that for every $0 \le p \le r+s$ the (disjoint) union
$$
\bigsqcup_{p_1+p_2 = p} T_{r,s,p_1, p_2} 
$$
coincides with the set $\Sh_{p,r+s-p}$. 

Hence,
$$
\sum_{\substack{ 0 \le p_1 \le r \\ 0 \le p_2 \le s}} \, 
\sum_{\si \in T_{r, s, p_1,p_2}}\,
\mu_{\bt_{\si} }
\Big(
f(1_{p_1+p_2}) \otimes g(1_{r+s-p_1-p_2})
\Big) = \sum_{p=0}^{r+s-p} \sum_{\si \in \Sh_{p, r+s-p}}
\si \circ \vr_{r+s,p,i} \big(
f(1_p) \circ_{p+i} g(1_{r+s-p}) \big)\,.
$$ 

Thus, 
$$
\mD_{\cO}(f \circ_i g) = \mD_{\cO}(f) \, \circ_i \, \mD_{\cO}(g)\,.
$$

In other words, $\mD_{\cO}$ is compatible with all elementary 
operadic insertions. 

It is easy to see that $\mD_{\cO}$ is compatible with the units.

Lemma \ref{lem:mDcO-oper-map} is proved.
\end{proof}

\begin{lemma}
The maps $\mD_\cO$ \eqref{mD-cO-dfn} assemble to form a natural transformation 
$\mD :\Tw \Rightarrow \Tw \circ \Tw $. 
\end{lemma}
\begin{proof}
We have to show that for all arrows $f:\cO\to \cO'$ (respecting the maps from $\La\Lie_\infty$) the following diagram commutes.
\[
\begin{tikzpicture}
\matrix (m) [matrix of math nodes, row sep=2em, column sep=2em]
{\Tw \cO  &\Tw \cO' \\
\Tw\Tw\cO & \Tw\Tw\cO' \\ };
\path[->,font=\scriptsize]
(m-1-1) edge node[auto] {$\Tw f$} (m-1-2)
        edge node[left] {$\mD_\cO$} (m-2-1) 
(m-2-1) edge node[auto] {$\Tw\Tw f$} (m-2-2)
(m-1-2) edge node[auto] {$\mD_{\cO'}$} (m-2-2);
\end{tikzpicture}
\]
Unravelling the definitions this is amounts to saying that the following diagrams commute
\[
\begin{tikzpicture}
\matrix (m) [matrix of math nodes, row sep=2em, column sep=2em]
{\Hom_{S_{r+s} }( \bs^{-2r-2s} \bbK, \cO(r+s+n))  &\Hom_{S_{r+s} }( \bs^{-2r-2s} \bbK, \cO'(r+s+n)) \\
\Hom_{S_r\times S_s}( \bs^{-2r-2s} \bbK, \cO(r+s+n)) & \Hom_{S_r\times S_s}( \bs^{-2r-2s} \bbK, \cO'(r+s+n)) \\ };
\path[->,font=\scriptsize]
(m-1-1) edge node[auto] {$f\circ$} (m-1-2)
        edge[right hook->] (m-2-1) 
(m-2-1) edge node[auto] {$f\circ$} (m-2-2)
(m-1-2) edge[right hook->] (m-2-2);
\end{tikzpicture}
\]
for all $r,s,n$. This is clear.
\end{proof}

\begin{thm}
\label{thm:Twcomonad}
The functor $\Tw$ together with the natural transformations $\eta, \mD$ \eqref{eta}, \eqref{mD-cO-dfn}, is 
a comonad on the under-category  
$\La \Lie_{\infty}\downarrow \Operads$.
\end{thm}
\begin{proof}
We have to verify the defining relations for a comonad. The two co-unit relations boil down to the statement that for any operad $\cO$ the compositions
\begin{gather*}
\Tw \cO \stackrel{\mD_\cO}{~\longrightarrow~} \Tw \Tw \cO \stackrel{\eta_{\Tw\cO}}{~\longrightarrow~} \Tw \cO \\
\Tw \cO \stackrel{\mD_\cO}{~\longrightarrow~} \Tw \Tw \cO \stackrel{\Tw \eta_\cO}{~\longrightarrow~} \Tw \cO
\end{gather*}
are the identity maps on $\Tw\cO$. This statement follows immediately from the definitions. Next consider the co-associativity axiom. In our case it boils down to the statement that for any operad $\cO$ the diagram
\[
\begin{tikzpicture}
\matrix (m) [matrix of math nodes, row sep=2em, column sep=3em]
{ \Tw \cO & \Tw\Tw\cO   \\
\Tw \Tw \cO & \Tw \Tw \Tw \cO \\ };
\path[->,font=\scriptsize]
(m-1-1) edge node[left] {$\mD_{\cO}$}(m-2-1) 
        edge node[auto] {$\mD_\cO$} (m-1-2)
(m-2-1) edge node[auto] {$\Tw(\mD_\cO)$} (m-2-2)
(m-1-2) edge node[auto] {$\mD_{\Tw \cO}$} (m-2-2);
\end{tikzpicture}
\]
commutes. Unravelling the definitions, we have to show that the following diagram commutes
\[
\begin{tikzpicture}
\matrix (m) [matrix of math nodes, row sep=2em, column sep=3em]
{ \Hom_{S_{r+s+t}} ( \bs^{-2r-2s-2t} \bbK, \cO(r+s+t+n)) 
& \Hom_{S_{r+s}\times S_t} ( \bs^{-2r-2s-2t} \bbK, \cO(r+s+t+n))   \\
\Hom_{S_r\times S_{s+t}} ( \bs^{-2r-2s-2t} \bbK, \cO(r+s+t+n)) 
& \Hom_{S_r\times S_s\times S_t} ( \bs^{-2r-2s-2t} \bbK, \cO(r+s+t+n)) \\ };
\path[right hook->,font=\scriptsize]
(m-1-1) edge (m-2-1) 
        edge (m-1-2)
(m-2-1) edge (m-2-2)
(m-1-2) edge (m-2-2);
\end{tikzpicture}
\]
for all $r,s,t,n$. This is again clear.
\end{proof}

\subsection{Coalgebras over the comonad \texorpdfstring{$\Tw$}{Tw}}
\label{sec:Tw-coalg}
Let us now consider coalgebras over the comonad $\Tw$. 
These are arrows of operads $\La \Lie_{\infty}\to \cO$ together with an operad map 
\[
c\colon \cO \to \Tw\cO
\]
such that the following axioms hold:
\begin{itemize}
\item The following diagram commutes:
\begin{equation}
\label{diag-LaLie-infty}
\begin{tikzpicture}
\matrix (m) [matrix of math nodes, row sep=2em]
{  & \La \Lie_{\infty} &  \\
\cO &  & \Tw \cO \\ };
\path[->,font=\scriptsize]
(m-1-2) edge (m-2-1) edge (m-2-3)
(m-2-1) edge node[auto] {$c$} (m-2-3);
\end{tikzpicture}
\end{equation}
\item The composition 
\begin{equation}
\label{comp-c-eta}
\cO \stackrel{c}{\to} \Tw\cO \stackrel{\eta_\cO}{\to} \cO
\end{equation}
is the identity.
\item The following diagram commutes:
\begin{equation}
\label{diag-c-Twc-mD}
\begin{tikzpicture}
\matrix (m) [matrix of math nodes, row sep=2em, column sep =3em]
{ \cO & \Tw\cO   \\
\Tw \cO & \Tw \Tw \cO \\ };
\path[->,font=\scriptsize]
(m-1-1) edge node[auto] {$c$}(m-2-1) 
        edge node[auto] {$c$} (m-1-2)
(m-2-1) edge node[auto] {$\Tw c$} (m-2-2)
(m-1-2) edge node[auto] {$\mD_{\cO}$} (m-2-2);
\end{tikzpicture}
\end{equation}
\end{itemize}

\begin{remark}
\label{rem:Tw-coalg-notation}
It often happens that a map from the dg operad $\La\Lie_{\infty}$ to $\cO$ is clear 
from the context. In this case, we abuse the notation and say that $\cO$ is a 
$\Tw$-coalgebra if the corresponding arrow  $\La\Lie_{\infty} \to \cO $ carries a coalgebra 
structure over the comonad $\Tw$. 
\end{remark}

\begin{example}
In Subsection \ref{sec:LaLieGerhomfixed} below, we will show that the operads
$\La\Lie$ and $\Ger$ carry canonical structures of a $\Tw$-coalgebra.
Furthermore, it is not hard to see that, the canonical morphism of 
dg operads
$$
\mT : \La\Lie_{\infty} \to \Tw  \La\Lie_{\infty}
$$ 
from Lemma \ref{lem:laliecoalg} equips $\La\Lie_{\infty}$ with a structure of a $\Tw$-coalgebra.
 \end{example}

\begin{example}
If $\La \Lie_\infty\to \cO$ is an arrow, then $\Tw \cO$ is a $\Tw$-coalgebra
(a cofree $\Tw$-coalgebra).
\end{example}

Another example of a  $\Tw$-coalgebra is given in the next section.

\subsection{$\Ger_{\infty}$ is canonically a $\Tw$-coalgebra}
\label{sec:Ginf-Twcoalg}

Let us recall from \cite{GJ}, \cite{GK}, \cite{Hinich} that 
$$
\Ger_{\infty} : = \Cobar(\Ger^{\vee})\,,
$$
where $\Ger^{\vee}$ is the cooperad Koszul dual to $\Ger$. Concretely, $\Ger^{\vee}$ is obtained from
$\La^{-2}\Ger$ by taking the linear dual.
The operad $\La\Lie_\infty$ is a sub-operad of $\Ger_{\infty}$, and we denote the inclusion by
\[
\iota: \La\Lie_\infty \to \Ger_{\infty}.
\]

Our goal, in this section, is to show 
that $\Ger_{\infty}$ is a $\Tw$-coalgebra. In order to do this, we have to complete two tasks:
\begin{enumerate}

\item Construct a map of dg operads $c: \Ger_{\infty}\to \Tw\Ger_{\infty}$. 

\item Verify that $c$ satisfies the axioms for a $\Tw$-coalgebra.
\end{enumerate}

To complete the first task 
we consider the free $\La^{-2}\Ger$-algebra $\La^{-2}\Ger_n$ in $n$ dummy 
variables $b_1, \dots, b_n$ of degree zero. The $n$-th space $\La^{-2}\Ger(n)$ of 
the operad $\La^{-2}\Ger$ is spanned by $\La^{-2}\Ger$-monomials in $b_1, \dots, b_n$
in which each dummy variable $b_i$ appears exactly once. It is clear that the operad 
$\La^{-2}\Ger$ is generated by the two vectors $b_1b_2, \{b_1, b_2\}\in \La^{-2}\Ger(2)$
of degrees $2$ and $1$, respectively.  

Next, we consider the ordered partitions of the set $\{1, 2, \dots, n\}$ 
\begin{equation}
\label{sp-partition}
\{i_{11}, i_{12}, \dots, i_{1 p_1}\} \sqcup 
\{i_{21}, i_{22}, \dots, i_{2 p_2}\} \sqcup \dots \sqcup \{i_{t1}, i_{t 2}, \dots, i_{t p_t}\}
\end{equation}
satisfying the following properties: 
\begin{itemize}

\item for each $1 \le \beta \le t$ the index $i_{\beta p_{\beta}}$ is 
the biggest among $i_{\beta 1}, \dots, i_{\beta p_{\beta}}$

\item $i_{1 p_1}  <  i_{2 p_2} < \dots <  i_{t p_t}$ (in particular, $i_{t p_t} = n$).

\end{itemize}

It is not hard to see that for each $n \ge 1$ the monomials 
\begin{equation}
\label{La-2Ger-n-basis}
\{ b_{i_{11}},  \dots, \{ b_{i_{1 (p_1-1)}}, b_{i_{1 p_1}} \br \dots
\{ b_{i_{t1}},  \dots, \{ b_{i_{t (p_t-1)}}, b_{i_{t p_t}} \br
\end{equation}
corresponding to all ordered partitions \eqref{sp-partition} satisfying the above properties
form a basis\footnote{Using this fact, it is easy to see that $\dim \La^{-2}\Ger(n) =n!$}
of $\La^{-2}\Ger(n)$\,. We denote by $I_n$ the set of the ordered 
partitions \eqref{sp-partition} and reserve the notation
\begin{equation}
\label{the-basis-n}
\big\{ w_{n,i} \big\}_{i \in I_n}
\end{equation}
and
\begin{equation}
\label{the-dual-basis}
\{w^*_{n, i}\}_{i \in I_n}
\end{equation}
for the basis of $\La^{-2}\Ger(n)$ formed by monomials \eqref{La-2Ger-n-basis}
and the dual basis of $\Ger^{\vee}(n) = \big( \La^{-2}\Ger(n) \big)^*$\,, respectively.

We observe that, for every $r \ge 0$ and for every basis vector
$w_{n,i}$, the monomial
$$
b_1 \dots b_r \, w_{n,i}(b_{r+1}, \dots, b_{r+n})
$$ 
belongs to the basis of  $\La^{-2} \Ger(r+n)$\,. 
In particular, we denote by 
$$
\big( b_1 \dots b_r \, w_{n,i}(b_{r+1}, \dots, b_{r+n} ) \big)^*
$$
the basis vector of $\Ger^{\vee}(r+n)$
which is dual to $b_1 \dots b_r \, w_{n, i}(b_{r+1}, \dots, b_{r+n}) $
in $\La^{-2}\Ger(r+n)$\,.

Since for every $r\ge 0$ and $n \ge 2$ the element
$\bs^{2r+1} \big( b_1 \dots b_r \, w_{n,i}(b_{r+1}, \dots, b_{r+n} ) \big)^*$
can be viewed as a vector in 
$$
\bs^{2r+1}\big( \Ger^{\vee}(r+n) \big)^{S_r} \subset  \Tw\Ger_{\infty}(n)\,,
$$
the formula
\begin{equation}
\label{al-mG}
\al_{\mG} = 
\sum_{n \ge 2,\, r \ge 0} ~ \sum_{i \in I_n}~ \bs^{2r+1}\,
\big( b_1 \dots b_r \, w_{n,i}(b_{r+1}, \dots, b_{r+n} ) \big)^*
\otimes w_{n,i}
\end{equation}
defines a degree $1$ element in the dg Lie algebra 
\begin{equation}
\label{Conv-Ginf-TwGinf}
\Conv(\Ger^{\vee}_{\c}, \Tw \Ger_{\infty}) \cong
\prod_{n \ge 2} \Big( \Tw \Ger_{\infty}(n) \otimes 
\La^{-2} \Ger(n) \Big)^{S_n}\,.
\end{equation}

We claim that 
\begin{prop}
\label{prop:al-mG}
Equation \eqref{al-mG} defines a Maurer-Cartan element of the 
dg Lie algebra \eqref{Conv-Ginf-TwGinf}.
\end{prop}
The proof of this proposition is quite technical so we postpone it to
Subsection \ref{sec:al-mG} given below.

Due to Theorem \ref{thm:from-Cobar}, maps from $\Ger_{\infty}$
to $\Tw \Ger_{\infty}$ are in bijection with Maurer-Cartan elements of the dg Lie algebra 
\eqref{Conv-Ginf-TwGinf}. Hence the Maurer-Cartan element $\al_{\mG}$ \eqref{al-mG}
defines a map of dg operads 
\begin{equation}
\label{Ginf-to-TwGinf}
c: \Ger_{\infty} \to \Tw \Ger_{\infty}\,.
\end{equation}

To prove that $c$ equips $\Ger_{\infty}$ with the structure of a $\Tw$-coalgebra, 
we have to verify three conditions.
The first condition states that the following diagram shall commute:
\[
\begin{tikzpicture}
\matrix (m) [matrix of math nodes, row sep=2em]
{  & \La \Lie_{\infty} &  \\
\Ger_\infty &  & \Tw \Ger_\infty \\ };
\path[->,font=\scriptsize]
(m-1-2) edge (m-2-1) node[auto] {$\iota$} edge (m-2-3)
(m-2-1) edge node[auto] {$c$} (m-2-3);
\end{tikzpicture}.
\]
Here the right hand arrow is defined as the composition 
\[
\La \Lie_{\infty} \to \Tw\La \Lie_{\infty} \stackrel{\Tw \iota}{\longrightarrow}  \Tw \Ger_\infty.
\]
It sends the generator $\bs\, 1^{\mc}_n$ of $\La\Lie_{\infty}$ to (cf. \eqref{dfn-LaLieinfty-to-Tw})
\[
\sum_{r\geq 0} \bs^{2r+1} (b_1 b_2 \cdots b_{r+n})^*.
\]
On the other hand, the left and bottom morphisms in the diagram under consideration send  $\bs \, 1^{\mc}_n$ to
\[
c\circ  \iota(\bs\, 1^{\mc}_n) = c \big(\bs (b_1\cdots b_n)^* \big) = \sum_{r\geq 0} \bs^{2r+1}  (b_1 b_2 \cdots b_{r+n})^*\,.
\]
Hence this condition is satisfied. The second condition for $\Tw$-coalgebras to be checked says that the composition
\[
\Ger_{\infty}\to \Tw\Ger_\infty \to \Ger_{\infty}
\]
shall be the identity. It is obviously satisfied.
The third condition states that the two compositions below shall be the same.
\begin{gather*}
\Ger_{\infty}\stackrel{c}{\longrightarrow} 
\Tw\Ger_\infty \stackrel{\mD_{\Ger_\infty} }{\longrightarrow} \Tw\Tw\Ger_\infty \\
\Ger_{\infty}\stackrel{c}{\longrightarrow} 
\Tw\Ger_\infty \stackrel{\Tw(c)}{\longrightarrow} \Tw\Tw\Ger_\infty 
\end{gather*}
In fact, unfolding the definitions one finds that both compositions operate 
on generators $\bs w^*_{n,i}$ of $\Ger_{\infty}$ as follows:
\[
\bs \, w^*_{n,i} ~\mapsto~ \sum_{r\geq 0} \sum_{s\geq 0}
\bs^{2r+2s+1} b_1 \dots b_{r+s} w^*_{n,i}(b_{r+s+1}, \dots, b_{r+s+n})
\]
Thus the third condition is also satisfied.

Summarizing, we obtain the following theorem:
\begin{thm}
\label{thm:Ginf-Twcoalg}
The morphism of dg operads 
\begin{equation}
\label{mG}
c : \Ger_{\infty} \to \Tw\Ger_{\infty}
\end{equation}
corresponding to the Maurer-Cartan element $\al_{\mG}$ \eqref{al-mG}
equips $\Ger_{\infty}$ with a $\Tw$-coalgebra structure. $\Box$
\end{thm}

\subsubsection{The proof of Proposition \ref{prop:al-mG}}
\label{sec:al-mG}

The proof of Theorem \ref{thm:Ginf-Twcoalg} is based on  Proposition \ref{prop:al-mG}. 
Here we prove this proposition.

Let $n,r$ be a pair of integers with $r \ge 0$ and $n \ge 2$\,.
In the groupoid $\Tree_2(r+n)$ we consider the objects  $\bt^{below}_{\si}$, 
$\bt^{above}_{\si}$,  $\bt^{k}_{\si}$,  $\bt^{p,q}_{\si, \tau}$ shown 
on figures \ref{fig:lower}, \ref{fig:higher},  \ref{fig:one-above}, and \ref{fig:ge-two}, 
respectively. 
\begin{figure}[htp]
\centering
\begin{minipage}[t]{0.48\linewidth}
\centering
\begin{tikzpicture}[scale=0.5]
\tikzstyle{w}=[circle, draw, minimum size=3, inner sep=1]
\tikzstyle{b}=[circle, draw, fill, minimum size=3, inner sep=1]
\node[b] (l1) at (-1, 4) {};
\draw (-1,4.6) node[anchor=center] {{\small $\si(1)$}};
\draw (0,3.8) node[anchor=center] {{\small $\dots$}};
\node[b] (lp) at (1, 4) {};
\draw (1,4.6) node[anchor=center] {{\small $\si(p)$}};
\node[w] (vv) at (0, 2.5) {};
\node[b] (lp1) at (3, 2.5) {};
\draw (2.8,3.1) node[anchor=center] {{\small $\si(p+1)$}};
\draw (4,2.4) node[anchor=center] {{\small $\dots$}};
\node[b] (lr) at (5, 2.5) {};
\draw (5,3.1) node[anchor=center] {{\small $\si(r)$}};
\node[b] (lr1) at (7, 2.5) {};
\draw (7,3.1) node[anchor=center] {{\small $r+1$}};
\draw (8.5,2.4) node[anchor=center] {{\small $\dots$}};
\node[b] (lrn) at (10, 2.5) {};
\draw (10,3.1) node[anchor=center] {{\small $r+n$}};
\node[w] (v) at (4, 0) {};
\node[b] (r) at (4, -1) {};
\draw (vv) edge (l1);
\draw (vv) edge (lp);
\draw (v) edge (vv);
\draw (v) edge (lp1);
\draw (v) edge (lr);
\draw (v) edge (lr1);
\draw (v) edge (lrn);
\draw (r) edge (v);
\end{tikzpicture}
\caption{\label{fig:lower} The tree $\bt^{below}_{\si}$\,. Here $2 \le p \le r$ and
$\si \in \Sh_{p, r-p}$}
\end{minipage}
\begin{minipage}[t]{0.48\linewidth}
\centering
\begin{tikzpicture}[scale=0.5]
\tikzstyle{w}=[circle, draw, minimum size=3, inner sep=1]
\tikzstyle{b}=[circle, draw, fill, minimum size=3, inner sep=1]
\node[b] (l1) at (0, 5) {};
\draw (0,5.6) node[anchor=center] {{\small $\si(1)$}};
\draw (1,5) node[anchor=center] {{\small $\dots$}};
\node[b] (lp) at (2, 5) {};
\draw (2,5.6) node[anchor=center] {{\small $\si(p)$}};
\node[b] (lr1) at (4, 5) {};
\draw (4,5.6) node[anchor=center] {{\small $r+1$}};
\draw (5.5,5) node[anchor=center] {{\small $\dots$}};
\node[b] (lrn) at (7, 5) {};
\draw (7,5.6) node[anchor=center] {{\small $r+n$}};
\node[w] (vv) at (4, 3) {};
\node[b] (lp1) at (7, 3) {};
\draw (7,3.6) node[anchor=center] {{\small $\si(p+1)$}};
\draw (8.5,3) node[anchor=center] {{\small $\dots$}};
\node[b] (lr) at (10, 3) {};
\draw (10,3.6) node[anchor=center] {{\small $\si(r)$}};
\node[w] (v) at (7, 1) {};
\node[b] (r) at (7, 0) {};
\draw (vv) edge (l1);
\draw (vv) edge (lp);
\draw (vv) edge (lr1);
\draw (vv) edge (lrn);
\draw (v) edge (vv);
\draw (v) edge (lp1);
\draw (v) edge (lr);
\draw (r) edge (v);
\end{tikzpicture}
\caption{\label{fig:higher} The tree $\bt^{above}_{\si}$\,.
Here $0 \le p \le r-1$ and $\si \in \Sh_{p, r-p}$}
\end{minipage}
\end{figure}
\begin{figure}[htp]
\centering
\begin{tikzpicture}[scale=0.5]
\tikzstyle{w}=[circle, draw, minimum size=3, inner sep=1]
\tikzstyle{b}=[circle, draw, fill, minimum size=3, inner sep=1]

\node[b] (lp1) at (0, 5) {};
\draw (-0.2,5.6) node[anchor=center] {{\small $\si(p+1)$}};
\draw (1.5,5) node[anchor=center] {{\small $\dots$}};
\node[b] (lr) at (3, 5) {};
\draw (3,5.6) node[anchor=center] {{\small $\si(r)$}};
\node[b] (lrk) at (5, 5) {};
\draw (5.1,5.6) node[anchor=center] {{\small $r+k$}};
\node[w] (vv) at (2, 2) {};

\node[b] (l1) at (-7, 2) {};
\draw (-7,2.6) node[anchor=center] {{\small $\si(1)$}};
\draw (-6,2) node[anchor=center] {{\small $\dots$}};
\node[b] (lp) at (-5, 2) {};
\draw (-5,2.6) node[anchor=center] {{\small $\si(p)$}};
\node[b] (lr1) at (-3, 2) {};
\draw (-3,2.6) node[anchor=center] {{\small $r+1$}};
\draw (-1.5,2) node[anchor=center] {{\small $\dots$}};
\node[b] (lr1k) at (0, 2) {};
\draw (0,2.6) node[anchor=center] {{\small $r+k-1$}};

\node[b] (lrk1) at (5, 2) {};
\draw (5,2.6) node[anchor=center] {{\small $r+k+1$}};
\draw (6.5,2) node[anchor=center] {{\small $\dots$}};
\node[b] (lrn) at (8, 2) {};
\draw (8,2.6) node[anchor=center] {{\small $r+n$}};
\node[w] (v) at (0, -1) {};
\node[b] (r) at (0, -2) {};
\draw (vv) edge (lp1);
\draw (vv) edge (lr);
\draw (vv) edge (lrk);
\draw (v) edge (vv);
\draw (v) edge (l1);
\draw (v) edge (lp);
\draw (v) edge (lr1);
\draw (v) edge (lr1k);
\draw (v) edge (lrk1);
\draw (v) edge (lrn);
\draw (r) edge (v);
\end{tikzpicture}
\caption{\label{fig:one-above} The tree $\bt^{k}_{\si}$\,.
Here $ 2-n \le p \le r-1$, $1 \le k \le n$, and $\si \in \Sh_{p, r-p}$ }
\end{figure}
\begin{figure}[htp]
\centering
\begin{tikzpicture}[scale=0.5]
\tikzstyle{w}=[circle, draw, minimum size=3, inner sep=1]
\tikzstyle{b}=[circle, draw, fill, minimum size=3, inner sep=1]
\node[b] (lp1) at (0, 5) {};
\draw (0,5.6) node[anchor=center] {{\small $\si(p+1)$}};
\draw (1.5,5) node[anchor=center] {{\small $\dots$}};
\node[b] (lr) at (3, 5) {};
\draw (3,5.6) node[anchor=center] {{\small $\si(r)$}};
\node[b] (lr1) at (5, 5) {};
\draw (5.2,5.6) node[anchor=center] {{\small $r+\tau(1)$}};
\draw (6.5,5) node[anchor=center] {{\small $\dots$}};
\node[b] (lrq) at (8, 5) {};
\draw (8.2,5.6) node[anchor=center] {{\small $r+\tau(q)$}};
\node[w] (vv) at (4, 2) {};
\node[b] (l1) at (0, 2) {};
\draw (0,2.6) node[anchor=center] {{\small $\si(1)$}};
\draw (1,2) node[anchor=center] {{\small $\dots$}};
\node[b] (lp) at (2, 2) {};
\draw (2,2.6) node[anchor=center] {{\small $\si(p)$}};
\node[b] (lrq1) at (7, 2) {};
\draw (7.2,2.6) node[anchor=center] {{\small $r+\tau(q+1)$}};
\draw (9,2) node[anchor=center] {{\small $\dots$}};
\node[b] (lrn) at (11, 2) {};
\draw (11,2.6) node[anchor=center] {{\small $r+\tau(n)$}};
\node[w] (v) at (4, -1) {};
\node[b] (r) at (4, -2) {};
\draw (vv) edge (lp1);
\draw (vv) edge (lr);
\draw (vv) edge (lr1);
\draw (vv) edge (lrq);
\draw (v) edge (l1);
\draw (v) edge (lp);
\draw (v) edge (vv);
\draw (v) edge (lrq1);
\draw (v) edge (lrn);
\draw (r) edge (v);
\end{tikzpicture}
\caption{\label{fig:ge-two} The tree $\bt^{p,q}_{\si, \tau}$\,. Here 
$0 \le p \le r$, $2 \le q \le n-1$, $\si \in \Sh_{p, r-p}$ and $\tau \in  \Sh_{q, n-q} $}
\end{figure}

It is clear that the trees   $\bt^{below}_{\si}$, 
$\bt^{above}_{\si}$,  $\bt^{k}_{\si}$,  $\bt^{p,q}_{\si, \tau}$ are all mutually
non-isomorphic. Furthermore, if both nodal vertices of a tree $\bt \in \Tree_2(n)$ 
have valencies $\ge 3$, then $\bt$ is isomorphic to one of the trees in the list:   $\bt^{below}_{\si}$, 
$\bt^{above}_{\si}$,  $\bt^{k}_{\si}$,  $\bt^{p,q}_{\si, \tau}$\,.

We will need the following technical statement:
\begin{lemma}
\label{lem:comult-coGer}
Let $n,r$ be a pair of integers with $r \ge 0$, $n \ge 2$\,, 
$w_{n,i}$ be a vector in the basis \eqref{the-basis-n}, and 
$$
\big( b_1 \dots b_r \, w_{n,i}(b_{r+1}, \dots, b_{r+n} ) \big)^*
$$
be the basis vector of $\Ger^{\vee}(r+n)$ dual to 
$b_1 \dots b_r \, w_{n,i}(b_{r+1}, \dots, b_{r+n} ) \in \La^{-2}\Ger(r+n)$\,.
Then
\begin{equation}
\label{cobar-bbbw}
\pa^{\Cobar} \bs\, (b_1 \dots b_r w_{n,i}(b_{r+1}, \dots, b_{r+n}))^* =
\end{equation}
$$
- \sum_{\substack{ 2 \le p \le r \\[0.1cm]  \si \in \Sh_{p, r-p} }} (-1)^{|w_{n,i}|} \big( \bt^{below}_{\si};  
 \bs\, (b_1 \dots b_{r-p+1} w_{n,i}(b_{r-p+2}, \dots, b_{r-p+1+n}))^* \otimes
 \bs\, (b_1 \dots b_p)^* \big) 
$$
$$
- \sum_{\substack{ 0 \le p \le r-1 \\[0.1cm]  \si \in \Sh_{p, r-p} }}  \big( \bt^{above}_{\si};  
 \bs\,( b_1 \dots b_{r-p+1})^* \otimes
 \bs\, (b_1 \dots b_{p} w_{n,i}(b_{p+1}, \dots, b_{p+n}))^* 
\big) 
$$
$$
- \sum_{\substack{ 2-n \le p \le r-1 \\[0.1cm]  \si \in \Sh_{p, r-p} }}  ~ \sum_{k=1}^n
\, (-1)^{|w_{n,i}|} \,
\big( \bt^{k}_{\si};   
  \bs\, (b_1 \dots b_p w_{n,i}(b_{p+1}, \dots, b_{p+n}))^* \otimes 
  \bs\, (b_1 \dots b_{r-p+1})^*
 \big)
$$
$$
- \sum_{\substack{ 0 \le p \le r \\[0.1cm]  2 \le q \le n-1 }}
\sum_{\substack{ \si \in \Sh_{p, r-p} \\[0.1cm]  \tau \in \Sh_{q, n-q} }} \sum_{i_1, i_2}  
(-1)^{ | w_{n-q+1 , i_1}| | w_{q , i_2}| + |w_{n-q+1 , i_1}|} f^i_{i_1 i_2}
\big( \bt^{p,q}_{\si, \tau} ~ ;~\bs\, (b_1 \dots b_p w_{n-q+1 , i_1}(b_{p+1}, \dots, b_{p+n-q+1}))^*
$$
$$
\otimes ~ 
\bs\, (b_1 \dots b_{r-p} w_{q , i_2}(b_{r-p+1}, \dots, b_{r-p+q}))^*
\big)\,, 
$$
where the coefficients $f_{i_1 i_2}^i \in \bbK$ are defined by the equation
\begin{equation}
\label{f-ii1i2-dfn}
\tau (w_{n-q+1 , i_1} \circ_1  w_{q , i_2}) = \sum_{i \in I_n} f^{i}_{i_1 i_2} w_{n, i}\,.
\end{equation}
\end{lemma}
\begin{proof}
The statement of the lemma follows from these equations:
\begin{equation}
\label{D-bt-below-si}
\D_{\bt^{below}_{\si}} \big( b_1 \dots b_r w_{n,i}(b_{r+1}, \dots, b_{r+n}) \big)^* = 
\end{equation}
$$
(b_1 \dots b_{r-p+1} w_{n,i}(b_{r-p+2}, \dots, b_{r-p+1+n}))^* \otimes
 (b_1 \dots b_p)^*\,, 
$$
\begin{equation}
\label{D-bt-above-si}
\D_{\bt^{above}_{\si}} \big( b_1 \dots b_r w_{n,i}(b_{r+1}, \dots, b_{r+n}) \big)^* = 
\end{equation}
$$
 ( b_1 \dots b_{r-p+1})^* \otimes
  (b_1 \dots b_{p} w_{n,i}(b_{p+1}, \dots, b_{p+n}))^* \,,
$$
\begin{equation}
\label{D-bt-bullet-si}
\D_{\bt^{k}_{\si}} \big( b_1 \dots b_r w_{n,i}(b_{r+1}, \dots, b_{r+n}) \big)^* = 
\end{equation}
$$
 (b_1 \dots b_p w_{n,i}(b_{p+1}, \dots, b_{p+n}))^* \otimes 
  (b_1 \dots b_{r-p+1})^*\,,
$$
and 
\begin{equation}
\label{D-bt-pq-si-tau}
\D_{\bt^{p,q}_{\si, \tau}} \big(b_1 \dots b_r w_{n,i}(b_{r+1}, \dots, b_{r+n})\big)^* =
\end{equation}
$$
\sum_{\substack{i_1 \in I_{n-q+1},\\ i_2 \in I_{q}} } 
(-1)^{ | w_{n-q+1 , i_1}| \cdot | w_{q , i_2}|} 
f^i_{i_1 i_2}  (b_1 \dots b_p w_{n-q+1 , i_1}(b_{p+1}, \dots, b_{p+n-q+1}))^*
$$
$$
\otimes ~  (b_1 \dots b_{r-p} w_{q , i_2}(b_{r-p+1}, \dots, b_{r-p+q}))^*\,,
$$
where the coefficients $f^i_{i_1 i_2}$ are defined by equation \eqref{f-ii1i2-dfn}.

Let us prove that equation \eqref{D-bt-pq-si-tau} holds. 

In general, we have
$$
\D_{\bt^{p,q}_{\si, \tau}} \big(b_1 \dots b_r w_{n,i}(b_{r+1}, \dots, b_{r+n})\big)^* =
$$
\begin{equation}
\label{Dbt-pq-general}
\sum_{\substack{j_1 \in I_{n+p-q+1}, \\  j_2 \in I_{r-p+q}} } 
 g^{i}_ {j_1, j_2}  w^*_{n+p-q+1, j_1}  \otimes w^*_{r-p+q, j_2}\,, 
\qquad g^i_{j_1 j_2} \in \bbK\,.
\end{equation}

It is not hard to see that, if the basis vector $w_{n+p-q+1, j_1}$
is not of the form 
$$
b_1 \dots b_p w_{n-q+1 , i_1}(b_{p+1}, \dots, b_{p+n-q+1})
$$
or the basis vector $w_{r-p+q, j_2}$ is not of the form
$$
b_1 \dots b_{r-p} w_{q , i_2}(b_{r-p+1}, \dots, b_{r-p+q})\,,
$$
then 
$$
\big(b_1 \dots b_r w_{n,i}(b_{r+1}, \dots, b_{r+n})\big)^* 
\Big(\mu_{\bt^{p,q}_{\si, \tau}} ( w_{n+p-q+1, j_1} 
\otimes w_{r-p+q, j_2}
\Big) = 0\,.
$$

In other words, $g^i_{j_1 j_2} = 0$ unless the basis vector 
$w_{n+p-q+1, j_1}$ is of the form 
$$
b_1 \dots b_p w_{n-q+1 , i_1}(b_{p+1}, \dots, b_{p+n-q+1})
$$
and the basis vector $w_{r-p+q, j_2}$ is of the form
$$
b_1 \dots b_{r-p} w_{q , i_2}(b_{r-p+1}, \dots, b_{r-p+q})\,.
$$

Hence, 
\begin{equation}
\label{Dbt-pq-precise}
\D_{\bt^{p,q}_{\si, \tau}} \big(b_1 \dots b_r w_{n,i}(b_{r+1}, \dots, b_{r+n})\big)^* =
\end{equation}
$$
\sum_{\substack{i_1 \in I_{n-q+1},\\ i_2 \in I_{q}} } 
\wt{f}^i_{i_1 i_2}  (b_1 \dots b_p w_{n-q+1 , i_1}(b_{p+1}, \dots, b_{p+n-q+1}))^*
\otimes  (b_1 \dots b_{r-p} w_{q , i_2}(b_{r-p+1}, \dots, b_{r-p+q}))^*
$$
for some coefficients $\wt{f}^i_{i_1 i_2} \in \bbK$\,.

On the other hand, we have
$$
\D_{\bt^{p,q}_{\si, \tau}} \big(b_1 \dots b_r w_{n,i}(b_{r+1}, \dots, b_{r+n})\big)^* 
\Big(  b_1 \dots b_p w_{n-q+1 , i_1}(b_{p+1}, \dots, b_{p+n-q+1})
$$
$$  
\otimes ~ b_1 \dots b_{r-p} w_{q , i_2}(b_{r-p+1}, \dots, b_{r-p+q}) \Big) =
$$
\begin{equation}
\label{comp-bt-pq-si-tau}
\big(b_1 \dots b_r w_{n,i}(b_{r+1}, \dots, b_{r+n})\big)^* 
\Big(\mu_{\bt^{p,q}_{\si, \tau}} \big(
 b_1 \dots b_p w_{n-q+1 , i_1}(b_{p+1}, \dots, b_{p+n-q+1}) 
\end{equation}
$$ 
\otimes ~ b_1 \dots b_{r-p} w_{q , i_2}(b_{r-p+1}, \dots, b_{r-p+q})
\big) \Big) =
$$
$$
(-1)^{\ve_{i_1 i_2}}
\big(b_1 \dots b_r w_{n,i}(b_{r+1}, \dots, b_{r+n})\big)^* \Big(
b_{\si(1)} \dots b_{\si(p)}
$$
$$ 
w_{n-q+1 , i_1}\big( b_{\si(p+1)} \dots b_{\si(r)} w_{q , i_2}(b_{r+ \tau(1)}, \dots, b_{r+ \tau(q)}), 
b_{r+ \tau(q+1)}, \dots, b_{r+\tau(n)} \big) \Big)\,,
$$
where the sign factor $(-1)^{ \ve_{i_1 i_2} }$ comes from permuting the  $\La^{-2}\Ger$-monomial 
$w_{q, i_2}$ with brackets  $\{~,~\}$\,.

Since in the variables $b_1, \dots, b_r$ in the $\La^{-2}\Ger$-monomial 
$b_1 \dots b_r w_{n,i}(b_{r+1}, \dots, b_{r+n})$ only enter $\La^{-1}\Lie$ words 
of length $1$\,, we have 
$$
\D_{\bt^{p,q}_{\si, \tau}} \big(b_1 \dots b_r w_{n,i}(b_{r+1}, \dots, b_{r+n})\big)^* 
\Big(  b_1 \dots b_p w_{n-q+1 , i_1}(b_{p+1}, \dots, b_{p+n-q+1})
$$
$$  
\otimes ~ b_1 \dots b_{r-p} w_{q , i_2}(b_{r-p+1}, \dots, b_{r-p+q}) \Big) =
$$
\begin{equation}
\label{upshot-bt-pq-si-tau}
(-1)^{ \ve_{i_1 i_2} }
\big(b_1 \dots b_r w_{n,i}(b_{r+1}, \dots, b_{r+n})\big)^*  \Big(
b_{\si(1)} \dots b_{\si(r)}
\end{equation}
$$
w_{n-q+1 , i_1}\big( w_{q , i_2}(b_{r+ \tau(1)}, \dots, b_{r+ \tau(q)}), 
b_{r+ \tau(q+1)}, \dots, b_{r+\tau(n)} \big)
\Big) =
$$
$$
(-1)^{ \ve_{i_1 i_2} } \big(b_1 \dots b_r w_{n,i}(b_{r+1}, \dots, b_{r+n})\big)^*  \Big(
b_{1} \dots b_{r}
$$
$$
w_{n-q+1 , i_1}\big( w_{q , i_2}(b_{r+ \tau(1)}, \dots, b_{r+ \tau(q)}), 
b_{r+ \tau(q+1)}, \dots, b_{r+\tau(n)} \big)
\Big) = f^i_{i_1 i_2}\,.
$$
Thus $\wt{f}^i_{i_1 i_2} = (-1)^{ | w_{n-q+1 , i_1}| | w_{q , i_2}|}  f^i_{i_1 i_2}$ 
and equation \eqref{D-bt-pq-si-tau} holds. 

Using the identification $\Ger^{\vee} = \La^{-2}\Ger^*$ in the similar way, 
it is easy to prove that equations  \eqref{D-bt-below-si},  
\eqref{D-bt-above-si},  \eqref{D-bt-bullet-si} also hold.  

Lemma \ref{lem:comult-coGer} follows. 
\end{proof}

To prove  Proposition \ref{prop:al-mG} we need to show that the element 
$\al_{\mG}$ \eqref{al-mG} satisfies the Maurer-Cartan equation 
\begin{equation}
\label{MC-al-mG}
\pa^{\Cobar} \al_{\mG} +
\pa^{\Tw} \al_{\mG}  + \al_{\mG} \bul
\al_{\mG} = 0
\end{equation}
in the dg Lie algebra \eqref{Conv-Ginf-TwGinf}. 

Equation \eqref{MC-al-mG} unfolds as follows:
\begin{equation}
\label{MC-al-mG-unfold}
\sum_{n \ge 2,\, r \ge 0} ~ \sum_{i \in I_n}~ \bs^{2r}\, \pa^{\Cobar}
\Big( \bs \, \big( b_1 \dots b_r \, w_{n,i}(b_{r+1}, \dots, b_{r+n} ) \big)^* \Big)
\otimes w_{n,i} +
\end{equation}
$$
\sum_{n \ge 2,\, r \ge 0} ~ \sum_{i \in I_n}~ \pa^{\Tw}\, \Big(  \bs^{2r+1}\, 
\big( b_1 \dots b_r \, w_{n,i}(b_{r+1}, \dots, b_{r+n} ) \big)^* \Big)
\otimes w_{n,i}
$$
$$
+  \sum_{\substack{ 0 \le p \le r \\[0.1cm]  2 \le q \le n-1 }}
\sum_{\substack{ \si \in \Sh_{p, r-p} \\[0.1cm]  \tau \in \Sh_{q, n-q} }} \sum_{i_1, i_2}
(-1)^{ | w_{n-q+1 , i_1}| (| w_{q , i_2}| + 1)} \,
 \big( \bt^{p,q}_{\si, \tau};  \bs^{2p+1}\,  (b_1 \dots b_p w_{n-q+1 , i_1}(b_{p+1}, \dots, b_{p+n-q+1}))^*
$$
$$
\otimes~  \bs^{2(r-p)+1}\, (b_1 \dots b_{r-p} w_{q , i_2}(b_{r-p+1}, \dots, b_{r-p+q}))^*  \big)
~\otimes ~ f^i_{i_1 i_2} w_{n,i} = 0\,,
$$
where the coefficients $f^i_{i_1 i_2}$ are defined by equations \eqref{f-ii1i2-dfn}.

Let us now use Lemma \ref{lem:comult-coGer}.

The contribution to 
$$
\sum_{n \ge 2,\, r \ge 0} ~ \sum_{i \in I_n}~ \bs^{2r}\, \pa^{\Cobar}
\Big( \bs \, \big( b_1 \dots b_r \, w_{n,i}(b_{r+1}, \dots, b_{r+n} ) \big)^* \Big)
\otimes w_{n,i}
$$  
coming from the last sum in the right hand side of  \eqref{cobar-bbbw}
cancels with the third sum in equation \eqref{MC-al-mG-unfold}.
 
The contributions to
$$
\sum_{n \ge 2,\, r \ge 0} ~ \sum_{i \in I_n}~ \bs^{2r}\, \pa^{\Cobar}
\Big( \bs \, \big( b_1 \dots b_r \, w_{n,i}(b_{r+1}, \dots, b_{r+n} ) \big)^* \Big)
\otimes w_{n,i}
$$  
coming from the remaining sums in the right hand side of  \eqref{cobar-bbbw}
cancel the sum
$$
\sum_{n \ge 2,\, r \ge 0} ~ \sum_{i \in I_n}~ \pa^{\Tw}\, \Big(  \bs^{2r}\, 
\big( b_1 \dots b_r \, w_{n,i}(b_{r+1}, \dots, b_{r+n} ) \big)^* \Big)
\otimes w_{n,i}\,.
$$

Thus $\al_{\mG}$ satisfies \eqref{MC-al-mG} and Proposition \ref{prop:al-mG} follows.  

\hfill \qed

\section{Homotopy theoretic properties of the twisting procedure}
\label{sec:twisting-homotopy}

Let us prove that the functor $\Tw$ preserves quasi-isomorphisms.
\begin{thm}
\label{thm:Twpreservesqiso}
Let $\La \Lie_\infty\to \cO \stackrel{F}{\to} \cO'$ be a morphism of 
the under-category  $\La \Lie_{\infty}\downarrow \Operads$
with $F$ being a quasi-isomorphism. 
Then, the map $\Tw(F): \Tw\cO \to \Tw\cO'$ is also a quasi-isomorphism.
\end{thm}
\begin{proof}
We want to show that $\Tw\cO(n) \to \Tw\cO'(n)$ is a quasi-isomorphism for every $n=0,1,2, \dots$. 
This is equivalent to saying that the mapping cone 
\[
C := \Tw\cO(n)\oplus \bs\, \Tw\cO'(n)
\]
is acyclic for every $n$. There is a natural complete filtration $C=\mF_0 \supset \mF_1 \supset \cdots$ such that 
\[
\mF_p = \prod_{r \geq p}  \Hom_{S_r} ( \bs^{-2r} \bbK, \cO(r+n)) \oplus \prod_{r \geq p} 
\bs\, \Hom_{S_r} ( \bs^{-2r} \bbK, \cO'(r+n))\, .
\]

Note that the associated graded complex for this filtration is
\[
\bigoplus_{p}
\mF_p / \mF_{p+1} \cong \bigoplus_{p} \, \Hom_{S_p} ( \bs^{-2p} \bbK, \cO(p+n) \oplus \bs \, \cO'(p+n)).
\]
The complex $\cO(p+n) \oplus \bs \, \cO'(p+n)$ is the mapping cone of $\cO(p+n) \to \cO'(p+n)$ and hence acyclic by assumption. Since taking invariants with respect to a finite group action commutes with taking cohomology, we conclude that the associated graded is acyclic as well. Thus, by Lemma \ref{lem:filtered} from Appendix \ref{app:filtered-lem}, 
the statement of the proposition follows.
\end{proof}

\begin{example}
\label{exam:Tw-q-iso}
Since the canonical maps 
$$
U_{\La\Lie} : \La\Lie_{\infty}  \to  \La\Lie \,, \qquad 
U_{\Ger} : \Ger_{\infty} \to \Ger
$$
are quasi-isomorphisms of dg operads, the morphisms
$$
\Tw (U_{\La\Lie}) : \Tw(\La\Lie_{\infty})  \to  \Tw(\La\Lie)\,,
$$
and
$$
\Tw(U_{\Ger}) :  \Tw \Ger_{\infty} \to \Tw \Ger
$$
are also  quasi-isomorphisms of dg operads.
\end{example}

As we will see below in Section \ref{sec:distr-law}, many objects of the 
under-category    $\La \Lie_{\infty}\downarrow \Operads$ satisfy the 
following remarkable property:
\begin{defi}
\label{dfn:homot-fixed}
An arrow $\La \Lie_\infty\to \cO$ is called a \emph{homotopy fixed point} for $\Tw$ 
if the counit map $\eta_\cO: \Tw\cO \to \cO$ is a quasi-isomorphism.
\end{defi}

If the map from $\La \Lie_\infty$ is clear from the context, we will say, by abusing the notation, that $\cO$ is 
a homotopy fixed point of $\Tw$.

Theorem \ref{thm:Twpreservesqiso} implies that, if a dg operad $\cO$ is a homotopy fixed point 
for $\Tw$ and a dg operad $\cO'$ is quasi-isomorphic to $\cO$, then $\cO'$ is also 
a homotopy fixed point for $\Tw$\,.

\begin{example}
\label{ex:laliegerhomfixed}
In Subsection \ref{sec:LaLieGerhomfixed} below, we will show that the operads
$\La\Lie$ and $\Ger$  are homotopy fixed points for $\Tw$\,. Hence, by Theorem \ref{thm:Twpreservesqiso}, 
the dg operads $\La\Lie_\infty$ and $\Ger_\infty$ are also  homotopy fixed points for $\Tw$\,.
\end{example}

\begin{example}
In Section \ref{sec:Br-homotopyfixed} below, we will show that the dg operad 
$\Br$ governing braces algebras (see Section \ref{sec:Br}) is  a homotopy fixed point for $\Tw$.
\end{example}

\subsection{The distributive law and the functor \texorpdfstring{$\Tw$}{Tw} }
\label{sec:distr-law}
In this section, we describe a large class of dg operads 
which are simultaneously $\Tw$-coalgebras and homotopy fixed 
points for $\Tw$\,.

Let us recall \cite{GJ} that, for collections $P$ and $Q$ in $\mC$,
the formula
\begin{equation}
\label{plethysm}
P \odot Q (n)  =  \bigoplus_{r \ge 0} 
P(r)  \otimes_{S_r}  \Big(\, \bigoplus_{k_1 + \dots + k_r = n} 
\Ind^{S_n}_{S_{k_1} \times \dots \times S_{k_r} }  Q (k_1) \otimes \dots 
\otimes  Q (k_r) \,\Big)
\end{equation}
defines a monoidal structure on the category of collections. 
The product \eqref{plethysm} is known as plethysm.

Let $P$ be an arbitrary dg operad. Furthermore, let 
$\cO$ be the collection of cochain complexes 
\begin{equation}
\label{O-P-odot-LaLie}
\cO = P \odot \La \Lie 
\end{equation}
which is obtained by computing the plethysm \eqref{plethysm}
of $P$ with $\La \Lie$\,.

Let $ \vs_{1,i}$ be the cycle $(1,2, \dots, i)$ in $S_{n+1}$\,.
It is not hard to see that imposing the relation 
\begin{equation}
\label{distrib-law}
\{a_1, a_2\} \circ_2 \ga = (-1)^{|\ga|} \sum_{i=1}^n \vs_{1,i} (\ga \circ_i \{a_1, a_2\} ) \qquad \forall ~~ \ga \in P(n)     
\end{equation}
and using the operad structures on $P$ and on $\La\Lie$ we get
a natural dg operad structure on the collection \eqref{O-P-odot-LaLie}. 
Following \cite{Markl-distr} and \cite[Section 8.6]{Loday-Vallette},
we say that the dg operad $\cO$ is 
obtained from $P$ and $\La\Lie$ using the {\it distributive law}
\eqref{distrib-law}.

\begin{example}
\label{ex:Ger-Com-circ-LaLie}
The operad $\Ger$ is obtained from the operad $\Com$ via the above construction.
$$
\Ger = \Com \odot \La \Lie\,. 
$$
\end{example}

For dg operads obtained in this way, we have the 
following straightforward proposition:  
\begin{prop}
\label{prop:cO-distrib}
Let $\cO$ be the dg operad which is obtained via taking 
the plethysm \eqref{O-P-odot-LaLie} and imposing relation
\eqref{distrib-law}. Then for every $\ga \in \cO(n)$ we have
\begin{equation}
\label{beta-ga}
\{a_1, a_2\} \circ_2 \ga = (-1)^{|\ga|} \sum_{i=1}^n \vs_{1,i} (\ga \circ_i \{a_1, a_2\})\,.     
\end{equation}
Furthermore, for every $\cO$-algebra $V$, the adjoint action $\{v, ~\}$ is a derivation 
of the $\cO$-algebra structure on $V$ for every $v \in V$\,. ~~$\Box$
\end{prop}

\begin{remark}
\label{rem:BV-non}
Proposition \ref{prop:cO-distrib} implies that the operad $\BV$
governing the Batalin-Vilkovisky (BV) algebra \cite{BV} is not an operad 
which is obtained via the above construction. Indeed, the unary operation 
$\de$ on a BV algebra $V$ satisfies the relation 
$$
\de \big( \{v_1, v_2\} \big) + \{\de (v_1), v_2\} + (-1)^{|v_1|} \{v_1, \de (v_2)\} = 0\,.
$$
Hence, in general, the adjoint action $\{v,~\}$ is not a derivation of 
the BV algebra structure on $V$\,.
\end{remark}

Let $\cO$ be the dg operad which is obtained via taking 
the plethysm \eqref{O-P-odot-LaLie} and imposing relation
\eqref{distrib-law}. It is obvious that the dg operad $\cO$ receives a natural 
embedding
\begin{equation}
\label{LaLie-to-cO}
\mi : \La \Lie \hookrightarrow \cO\,.
\end{equation}

Composing $\mi$ with the canonical quasi-isomorphism 
\eqref{LaLie-infty-LaLie} $U_{\La\Lie} :  \La \Lie_{\infty} \to  \La \Lie$ we 
get a map of dg operads
\begin{equation}
\label{LaLie-infty-to-cO}
\hat{\vf} : =  \mi \circ U_{\La\Lie} ~ : ~ \La \Lie_{\infty} ~\to~ \cO\,.
\end{equation}
Hence, we may apply the twisting procedure to the pair $(\cO, \hat{\vf})$
and obtain a dg operad $\Tw\cO$\,.

According to Section \ref{sec:Tw-oplus}, the spaces 
\begin{equation}
\label{Tw-oplus-cO}
\Tw^{\oplus} \cO(n) = \bigoplus \bs^{2r} \big( \cO(r+n) \big)^{S_r}
\end{equation}
form a sub- dg operad of $\Tw\cO$\,. 

It turns our that the dg operad $\Tw^{\oplus} \cO$ coincides 
with $\Tw\cO$, provided the dg operad $P$ satisfies 
a minor technical condition. Namely,
\begin{prop}
\label{prop:Tw-oplus-Tw}
Let $P$ be a dg operad for which there exists a integer
$N$ such that for each $n\ge 0$ and for every $v \in P(n)$
\begin{equation}
\label{bounded-below}
| v | \ge N\,.
\end{equation}
If the dg operad $\cO$ is obtained via taking 
the plethysm \eqref{O-P-odot-LaLie} and imposing relation
\eqref{distrib-law} then  
\begin{equation}
\label{Tw-oplus-cO-TwcO}
\Tw^{\oplus} \cO = \Tw \cO\,. 
\end{equation}
\end{prop}
\begin{proof}
Let $m$ be an integer and $n$ be a non-negative integer. 
Our goal is to prove every sum
\begin{equation}
\label{sum}
\sum_{r=0}^{\infty} w_r\,, \qquad w_r \in  \bs^{2r} \big( \cO(r+n) \big)^{S_r}
\end{equation}
of a fixed degree $m$ has only finitely many terms.

For every $r\ge 0$, the graded vector space  $\bs^{2r} \big( \cO(r+n) \big)^{S_r}$
is spanned by vectors of the form
\begin{equation}
\label{this-form}
\sum_{\si \in S_r} 
\si  (v ; X_1, X_2, \dots, X_t )\,, 
\end{equation}
where $v \in P(t)$ for some $t$ and $X_i$ are 
vectors in $\La\Lie(k_i)$ such that 
\begin{equation}
\label{sum-k-i}
\sum_{i=1}^t k_i = n + r\,.
\end{equation}

Since a vector \eqref{this-form} carries degree $m$ in 
$$
\bs^{2r} \big( \cO(r+n) \big)^{S_r}\,,
$$
we obtain the following equation 
\begin{equation}
\label{m-r-v-ki}
m = 2r + |v| + \sum_{i=1}^t (1-k_i)\,.
\end{equation}

Combining \eqref{sum-k-i} with \eqref{m-r-v-ki} we get 
\begin{equation}
\label{m-r-v-t-n}
m = r + |v| + t - n 
\end{equation}
and hence  
\begin{equation}
\label{r-equals}
r = m + n - t -|v|\,. 
\end{equation}
 
Since $t \ge 0$ and $|v| \ge N $, we deduce that 
\begin{equation}
\label{r-less-or-eq}
r \le m + n - N\,. 
\end{equation}

Thus the sum in \eqref{sum} has indeed only finitely many terms.  
\end{proof}

The following theorem is the central result of this 
section\footnote{An idea of this proof is borrowed from \cite{grt}.}.
\begin{thm}
\label{thm:distrib-homot-fixed}
If an object  $(\cO, \hat{\vf})$ of the under-category 
 $\La \Lie_{\infty}\downarrow \Operads$ is obtained via taking 
the plethysm \eqref{O-P-odot-LaLie} of a dg operad $P$ with 
$\La \Lie$ and imposing relation
\eqref{distrib-law}, then  $(\cO, \hat{\vf})$ is canonically a $\Tw$-coalgebra. 
If, in addition, the dg operad $P$ satisfies the condition of Proposition \ref{prop:Tw-oplus-Tw},
then $\cO$ is a homotopy fixed point for $\Tw$\,.
\end{thm}
\begin{proof} Identity \eqref{beta-ga} implies that the canonical embedding 
$$
\emb_{\cO} : \cO \to \Tw\cO
$$
\begin{equation}
\label{cO-TwcO}
\big(\emb_{\cO}(v)\big)(1_r) = 
\begin{cases}
 v \qquad {\rm if} ~~ r = 0 \,, \\
 0  \qquad {\rm otherwise}
\end{cases}
\end{equation}
is compatible with the differentials $\pa^{\Tw}$ and $\pa^{\cO}$\,.
Furthermore, it is easy to see that the diagrams 
$$
\begin{minipage}[t]{0.4\linewidth} 
\centering 
\begin{tikzpicture}
\matrix (m) [matrix of math nodes, row sep=2.6em, column sep=2.6em]
{\cO & \Tw \cO \\
       & \cO \\ };
\path[->,font=\scriptsize]
(m-1-1) edge node[auto] {$\emb_{\cO}$} (m-1-2) 
edge node[auto] {$\id$} (m-2-2)
(m-1-2) edge  node[auto] {$\eta_{\cO}$} (m-2-2); 
\end{tikzpicture}
\end{minipage}
~~
\begin{minipage}[t]{0.4\linewidth} 
\centering 
\begin{tikzpicture}
\matrix (m) [matrix of math nodes, row sep=2.6em, column sep=2.6em]
{\cO & \Tw \cO \\
  \Tw \cO    & \Tw \Tw\cO \\ };
\path[->,font=\scriptsize]
(m-1-1) edge node[auto] {$\emb_{\cO}$} (m-1-2) 
edge node[left] {$\emb_{\cO}$} (m-2-1) 
(m-1-2) edge  node[auto] {$\Tw(\eta_{\cO})$} (m-2-2)
(m-2-1) edge  node[auto] {$\mD_{\cO}$} (m-2-2); 
\end{tikzpicture}
\end{minipage}
$$
commute. Thus $(\cO, \hat{\vf})$ is indeed a coalgebra over the comonad $\Tw$\,.

Let us now observe that the map $\emb_{\cO}$ \eqref{cO-TwcO} lands 
into the sub- dg operad $\Tw^{\oplus}\cO$\,. Furthermore,
if the dg operad $P$ satisfies the condition of Proposition \ref{prop:Tw-oplus-Tw},
then 
$$
\Tw^{\oplus}\cO  =  \Tw \cO\,.
$$

Thus, we need to prove that the map
\begin{equation}
\label{emb-cO-n}
\emb_{\cO} \Big|_{\cO(n)}  : \cO(n) \hookrightarrow \Tw^{\oplus}\cO (n) 
\end{equation}
is a quasi-isomorphism of cochain complexes for every $n$\,.
For this purpose we consider the free $\cO$-algebra  
\begin{equation}
\label{cO-a-ai}
\cO(a, a_1, a_2, \dots, a_n)
\end{equation}
generated by $n$ dummy  variables $a_1,a_2, \dots, a_n$ of degree zero and one dummy variable $a$
of degree $2$\,. 
Next, we denote by $\de$ the degree $1$ derivation of the 
$\cO$-algebra  $\cO(a, a_1, a_2, \dots, a_n)$
defined by the formulas:
\begin{equation}
\label{delta-cO-a-ai}
\de (a) = \frac{1}{2} \{a,a\}\,, \qquad 
\de (a_i) = 0 \qquad \forall~~ 1 \le i \le n\,. 
\end{equation}
Due to the Jacobi identity, we have $\de^2 = 0$\,. 
Hence, $\de$ is a differential on $\cO$-algebra  
\eqref{cO-a-ai}. We will combine $\de$
with the differential $\pa^{\cO}$ coming from $\cO$ and consider 
the $\cO$-algebra \eqref{cO-a-ai}  with the
differential 
\begin{equation}
\label{diff-cO-a-ai}
\pa^{\cO} + \de\,.
\end{equation}

Let us denote by 
\begin{equation}
\label{cOprime-a-ai}
\cO'(a, a_1, a_2, \dots, a_n) 
\end{equation}
the subcomplex of \eqref{cO-a-ai} which is spanned by 
$\cO$-monomials in which each variable from the set $\{a_1,a_2, \dots, a_n\}$
appears exactly once. 

It is not hard to see that the formula 
\begin{equation}
\label{TwcO-cO-a-ai}
\psi (\bs^{2r}\, v) = \frac{1}{r!}\, v \otimes a^{\otimes \, r} \otimes a_1 \otimes a_2 
\otimes \dots \otimes a_n\,, \qquad 
v \in \big( \cO(r+n) \big)^{S_r}
\end{equation}
gives us an obvious isomorphism of graded vector spaces
\begin{equation}
\label{psi-map}
\psi :  \Tw^{\oplus}\cO (n) = 
\bigoplus_{r=0}^{\infty}  \big( \cO(r+n) \big)^{S_r} ~ \to ~ 
\cO'(a, a_1, a_2, \dots, a_n)\,.
\end{equation}
Using identity \eqref{beta-ga}, one can show that $\psi$ intertwines 
the differentials $\pa^{\Tw}$ on $\Tw^{\oplus}\cO (n)$ and 
\eqref{diff-cO-a-ai} on \eqref{cOprime-a-ai}. Hence, $\psi$ \eqref{psi-map} is 
an isomorphism of cochain complexes. 

Let us denote by
\begin{equation}
\label{cO-pr-ai}
\cO''(a_1,a_2, \dots, a_n) 
\end{equation}
the subcomplex of the free $\cO$-algebra $\cO(a_1,a_2, \dots, a_n)$
which is spanned by monomials in which each 
variable from the set $\{a_1,a_2, \dots, a_n\}$
appears exactly once. 

We observe that the assignment
$$
v \mapsto v \otimes a_1\otimes a_2 \otimes \dots \otimes a_n\,, 
\qquad v \in \cO(n)
$$
gives us an obvious identification between the cochain complex 
$\cO(n)$ and \eqref{cO-pr-ai}.

We also observe that the composition $\psi \circ \emb_{\cO}$ 
operates by the formula
\begin{equation}
\label{psi-emb-cO}
\psi \circ \emb_{\cO} (v) = v\otimes a_1\otimes a_2 \otimes \dots \otimes a_n\,.
\end{equation}

Thus, our goal is to show that the embedding 
\begin{equation}
\label{cO-TwcO-ai}
\cO''(a_1, a_2, \dots, a_n)  \hookrightarrow \cO'(a, a_1,a_2, \dots, a_n) 
\end{equation}
is a quasi-isomorphism of cochain complexes. 

For this purpose we introduce the free $\La\Lie$-algebra 
\begin{equation}
\label{La-Lie-a-a12-n}
\La\Lie(a, a_1, a_2, \dots, a_n)
\end{equation}
generated by $n$ dummy 
variables $a_1,a_2, \dots, a_n$ of degree zero and one dummy variable $a$
of degree $2$\,. 

We consider the $\La\Lie$-algebra  $\La\Lie(a, a_1, a_2, \dots, a_n)$
with the differential $\de$ defined in \eqref{delta-cO-a-ai}.
 
In addition, we introduce two subspaces
\begin{equation}
\label{La-Lie-pr}
\La\Lie''(a_1, a_2, \dots, a_n) \subset \La\Lie(a, a_1, a_2, \dots, a_n)
\end{equation}
and
\begin{equation}
\label{La-Lie-a-pr}
\La\Lie'(a, a_1, a_2, \dots, a_n) \subset \La\Lie(a, a_1, a_2, \dots, a_n)\,.
\end{equation}
Here  $\La\Lie'(a, a_1, a_2, \dots, a_n)$ is spanned by $\La\Lie$-monomials 
in   $\La\Lie(a, a_1, a_2, \dots, a_n)$ which involve each variable from 
the set $\{a_1, a_2, \dots, a_n\}$ at most once, and 
$\La\Lie''(a_1, a_2, \dots, a_n)$ is spanned by $\La\Lie$-monomials 
in   $\La\Lie'(a, a_1, a_2, \dots, a_n)$ which do not involve the variable $a$
at all. 

It is clear that both subspaces \eqref{La-Lie-pr} and \eqref{La-Lie-a-pr}
are subcomplexes of \eqref{La-Lie-a-a12-n}. Moreover, the restriction 
of the differential $\de$ to \eqref{La-Lie-pr} is zero.

The proof of Theorem \ref{thm:distrib-homot-fixed} is based on the following statement:
\begin{lemma}
\label{lem:La-Lie-a-pr}
The embedding 
\begin{equation}
\label{emb-LaLie-pr}
\emb : \La\Lie''(a_1, a_2, \dots, a_n)
\hookrightarrow
\La\Lie'(a, a_1, a_2, \dots, a_n)
\end{equation}
is a quasi-isomorphism of cochain complexes.  
In other words, for every cocycle $c \in \La\Lie'(a, a_1, a_2, \dots, a_n)$, 
there exists a vector $c_1 \in \La\Lie'(a, a_1, a_2, \dots, a_n)$ such that 
$$
c - \de(c_1) \in \La\Lie''(a_1, a_2, \dots, a_n)\,.
$$
\end{lemma}
The proof of this lemma is somewhat technical so 
we give it in Subsection \ref{sec:lemma-proof} below.
To deduce the desired statement from Lemma \eqref{lem:La-Lie-a-pr},
we observe that the cone of the embedding 
\eqref{cO-TwcO-ai} is a direct summand in the cone of the embedding 
\begin{equation}
\label{PLie-PLie}
P \big( \La\Lie''(a_1, a_2, \dots, a_n) \big)
  \hookrightarrow 
P\big( \La\Lie'(a, a_1, a_2, \dots, a_n) \big) \,.
\end{equation}
Hence the map \eqref{cO-TwcO-ai} is a quasi-isomorphism if so is
the map \eqref{PLie-PLie}. 

To prove that the embedding  \eqref{PLie-PLie} is a quasi-isomorphism 
we consider the following increasing filtration on the cochain 
complex $P\big( \La\Lie'(a, a_1, a_2, \dots, a_n) \big)$: 
\begin{equation}
\label{filtr-PLie-a}
\dots \subset
\cF^m P\big( \La\Lie'(a, a_1, a_2, \dots, a_n) \big) \subset 
\cF^{m+1} P\big( \La\Lie'(a, a_1, a_2, \dots, a_n) \big)\subset \dots\,,
\end{equation}
where 
$$
\cF^m  P\big( \La\Lie'(a, a_1, a_2, \dots, a_n)
$$ 
is spanned by $\cO$-monomials
$w$ for which 
$$
\deg_a(w) - |w| \le m\,,
$$
with $\deg_a(w)$ being the degree of $w$ in $a$\,.
Since the differential $\pa^{\cO}$ does 
not change the degree in $a$ and the differential $\de$ raises 
it by $1$ the total differential $\pa^{\cO} + \de$ is compatible with 
the filtration \eqref{filtr-PLie-a}. 
Restricting \eqref{filtr-PLie-a} to the subcomplex 
$P \big( \La\Lie''(a_1, a_2, \dots, a_n) \big)$ we get the ``silly'' filtration  
\begin{equation}
\label{silly-filtr}
\cF^m P \big( \La\Lie''(a_1, a_2, \dots, a_n) \big)^k = 
\begin{cases}
 P \big( \La\Lie''(a_1, a_2, \dots, a_n) \big)^k  \qquad {\rm if} ~~  k \ge  -m\,, \\
   \bfzero \qquad {\rm otherwise}
\end{cases}
\end{equation}
with the zero differential on the associated graded complex. 

To describe the associated graded complex for \eqref{filtr-PLie-a}
we denote by $\wt{P}$ the operad in $\grVect_{\bbK}$ which is obtained 
from $P$ by forgetting the differential. It is clear from the construction that 
the associated graded complex 
$$
\Gr P\big( \La\Lie'(a, a_1, a_2, \dots, a_n) \big)
$$
is isomorphic to the cochain complex
\begin{equation}
\label{isom-Gyryam}
\wt{P}\big( \La\Lie'(a, a_1, a_2, \dots, a_n) \big)
\end{equation}
with the differential $\de$\,.
Using the K\"unneth theorem and the fact that cohomology 
commutes with taking coinvariants, we deduce from Lemma   
\eqref{lem:La-Lie-a-pr} that the map 
$$
\Big( \wt{P}(n) \otimes \big( \La\Lie''(a_1, a_2, \dots, a_n) \big)^{\otimes\, n} \Big)_{S_n} 
\to 
\Big( \wt{P}(n) \otimes \big( \La\Lie'(a, a_1, a_2, \dots, a_n) \big)^{\otimes\, n} \Big)_{S_n} 
$$
is a quasi-isomorphism of cochain complexes for every $n$\,. 
Thus the embedding \eqref{PLie-PLie} induces a quasi-isomorphism on 
the level of associated graded complexes. Combining this observation 
with the fact that the filtrations \eqref{filtr-PLie-a}
and \eqref{silly-filtr} are locally bounded and cocomplete, we deduce, 
from Lemma A.3 in \cite[Appendix A]{notes}, that  the embedding \eqref{PLie-PLie}
is also a quasi-isomorphism. 
Hence the map \eqref{cO-TwcO-ai} is a  quasi-isomorphism of 
cochain complexes  and the theorem is proved. 
 \end{proof}

\subsubsection{The operads $\La\Lie$, $\Ger$,  $\La\Lie_{\infty}$, and $\Ger_{\infty}$ 
are homotopy fixed points for $\Tw$}
\label{sec:LaLieGerhomfixed}

Theorem \ref{thm:distrib-homot-fixed} has the following useful Corollary: 
\begin{corollary}
\label{cor:LaLie-Ger}
The operads  $\La\Lie$ and $\Ger$ carry a canonical $\Tw$-coalgebra structure. 
Moreover, these operads are homotopy fixed points for $\Tw$.  $\Box$
\end{corollary}

Combining Corollary \ref{cor:LaLie-Ger} with Theorem \ref{thm:Twpreservesqiso}, we immediately 
conclude that\footnote{The same result for the operad $\Lie_{\infty}$ 
was obtained in \cite[Section 7]{C-Lazarev} by J. Chuang and A. Lazarev.}  
\begin{corollary}
\label{cor:LaLie-Ger-inf}
The dg operads  $\La\Lie_{\infty}$ and $\Ger_{\infty}$ are homotopy fixed points for $\Tw$. $\Box$  
\end{corollary}
\begin{remark}
\label{rem:Linf-Ginf}
Let us recall that, due to Lemma \ref{lem:laliecoalg} and Theorem
\ref{thm:Ginf-Twcoalg},
both dg operads $\La\Lie_{\infty}$ and $\Ger_{\infty}$ carry 
canonical $\Tw$-coalgebra structures. In fact, the canonical map 
of operads 
$$
\mT :  \La\Lie_{\infty} \to \Tw \La\Lie_{\infty}
$$
was used to introduce the notion of $\Tw$-coalgebra.  
\end{remark}

\subsubsection{Proof of Lemma \ref{lem:La-Lie-a-pr}}
\label{sec:lemma-proof}
The proof of Theorem \ref{thm:distrib-homot-fixed} was based on 
Lemma  \ref{lem:La-Lie-a-pr}. Here we give a proof of this lemma.
We consider a non-empty ordered subset  $\{i_1< i_2 < \dots < i_k\}$ of $\{1,2,\dots, n\}$
and denote by 
\begin{equation}
\label{La-Lie-prpr}
\La\Lie'(a,a_{i_1}, \dots, a_{i_k})
\end{equation}
the subcomplex of $ \La\Lie'(a,a_1, \dots, a_n)$ which is spanned by $\La\Lie$-monomials
in $\La\Lie(a,a_{i_1}, \dots, a_{i_k})$ involving each variable in the set 
$\{a_{i_1}, \dots, a_{i_k} \}$ exactly once. 

It is clear that $ \La\Lie'(a,a_1, \dots, a_n)$ splits into the direct
sum of subcomplexes: 
\begin{equation}
\label{LaLiepr-decomp}
\La\Lie'(a,a_1, \dots, a_n) = \bbK\L a, \{a,a\}\R ~\oplus~ \bigoplus_{\{i_1< i_2 < \dots < i_k\}}
\La\Lie'(a,a_{i_1}, \dots, a_{i_k})\,,
\end{equation}
where the summation runs over all non-empty 
ordered subsets $\{i_1< i_2 < \dots < i_k\}$ of $\{1,2,\dots, n\}$\,.

It is not hard to see that the subcomplex $ \bbK\L a, \{a,a\}\R$ is 
acyclic. Thus our goal is to show that every cocycle in $\La\Lie'(a,a_{i_1}, \dots, a_{i_k})$
is cohomologous to cocycle in the intersection 
$$
\La\Lie'(a,a_{i_1}, \dots, a_{i_k}) \cap  \La\Lie''(a_1, a_2, \dots, a_n)\,.
$$

To prove this fact we consider the tensor algebra
\begin{equation}
\label{T-1-2-k}
T\big( \bbK \L \bsi a, \bsi a_{i_1}, \bsi a_{i_2}, \dots, \bsi a_{i_{k-1}} \R \big) 
\end{equation}
in the variables $\bsi a, \bsi a_{i_1}, \bsi a_{i_2}, \dots, \bsi a_{i_{k-1}}$ and denote by 
\begin{equation}
\label{Tpr-1-2-k}
T'(\bsi a, \bsi a_{i_1}, \bsi a_{i_2}, \dots, \bsi  a_{i_{k-1}}) 
\end{equation}
the subspace of \eqref{T-1-2-k} which is spanned by monomials 
involving each variable from the set 
$$
\{\bsi a_{i_1}, \bsi a_{i_2}, \dots, \bsi a_{i_{k-1}}\}
$$
exactly once. 

It is not hard to see that the formula 
\begin{equation}
\label{nu}
\nu (x_{j_1} \otimes x_{j_2} \otimes \dots \otimes x_{j_N}) = 
\{\bs\,x_{j_1}, \{\bs\,x_{j_2} ,\{ \dots \{\bs\, x_{j_{N}}, a_{i_k} \}..\}
\end{equation}
defines an isomorphism of the graded vector spaces
$$
\nu : T'(\bsi a, \bsi a_{i_1}, \bsi a_{i_2}, \dots, \bsi  a_{i_{k-1}})  \stackrel{\cong}{\longrightarrow} 
\La\Lie''(a,a_{i_1}, \dots, a_{i_k})\,.
$$ 

Let us denote by $\de_T$ a degree $1$ derivation of 
the tensor algebra \eqref{T-1-2-k} defined by the equations 
\begin{equation}
\label{de-T}
\de_T (\bsi a_{i_t}) = 0\,, \qquad
\de_T (\bsi a) = \bsi a \otimes \bsi a\,. 
\end{equation}
It is not hard to see that $(\de_T)^2=0$\,. Thus, $\de_T$ is 
a differential on the tensor algebra \eqref{T-1-2-k}\,.

The subspace \eqref{Tpr-1-2-k} is obviously a subcomplex of \eqref{T-1-2-k}. 
Furthermore, using the following consequence of Jacobi identity 
$$
\{a,\{a, X\}\} = - \frac{1}{2} \{\{a,a\},X\}\,, \qquad \forall ~~X \in \La\Lie(a,a_1, \dots, a_n),
$$
it is easy to show that 
$$
\de \circ \nu = \nu \circ \de_T\,.
$$

Thus $\nu$ is an isomorphism from the cochain complex 
$$
\big( T'(\bsi a, \bsi a_{i_1}, \bsi a_{i_2}, \dots, \bsi  a_{i_{k-1}}), \de_T \big)
$$
to the cochain complex 
$$
\big( \La\Lie'(a,a_{i_1}, \dots, a_{i_k}), \de \big)\,.
$$

To compute cohomology of the cochain complex 
\begin{equation}
\label{T-de-T}
\Big( T\big( \bbK \L \bsi a, \bsi a_{i_1}, \bsi a_{i_2}, \dots, \bsi a_{i_{k-1}} \R \big), 
 \de_T
\Big)
\end{equation}
we observe that the truncated tensor algebra 
\begin{equation}
\label{undT-bsi-a}
\und{T}_{\,\bsi a}: = \und{T}\big(\bbK\L \bsi a \R\big) 
\end{equation}
forms an acyclic subcomplex of \eqref{T-de-T}.

We also observe that the cochain complex 
 \eqref{T-de-T}
splits into the direct sum of subcomplexes
\begin{equation}
\label{T-de-T-sum}
T\big( \bbK \L \bsi a, \bsi a_{i_1}, \bsi a_{i_2}, \dots, \bsi a_{i_{k-1}} \R \big)  = 
T \big( \bbK \L \bsi a_{i_1}, \bsi a_{i_2}, \dots, \bsi a_{i_{k-1}} \R \big)  ~ \oplus 
\end{equation}
$$
\bigoplus_{m \ge 2,\, p_1, \dots, p_m}
V^{\otimes\, p_1}_{a_{\bul}} \otimes \und{T}_{\,\bsi a}
\otimes V^{\otimes\, p_2}_{a_{\bul}}\otimes 
\und{T}_{\,\bsi a} \otimes \dots  \otimes 
V^{\otimes\, p_{m-1}}_{a_{\bul}}\otimes \und{T}_{\,\bsi a} \otimes  V^{\otimes\, p_m}_{a_{\bul}}\,,
$$
where $V_{a_{\bul}}$ is the cochain complex 
$$ 
V_{a_{\bul}}: = \bbK \L \bsi a_{i_1}, \bsi a_{i_2}, \dots, \bsi a_{i_{k-1}}  \R
$$
with the zero differential and the summation runs over all 
combinations $(p_1, \dots, p_m)$ of integers satisfying the conditions
$$
p_1, p_m \ge 0, \qquad p_2, \dots, p_{m-1} \ge 1\,. 
$$

By K\"unneth's theorem all the subcomplexes 
$$
V^{\otimes\, p_1}_{a_{\bul}} \otimes \und{T}_{\,\bsi a}
\otimes V^{\otimes\, p_2}_{a_{\bul}}\otimes 
\und{T}_{\,\bsi a} \otimes \dots  \otimes 
V^{\otimes\, p_{m-1}}_{a_{\bul}}\otimes \und{T}_{\,\bsi a} \otimes  V^{\otimes\, p_m}_{a_{\bul}}
$$
are acyclic. Hence for every cocycle $c$  in \eqref{T-de-T}
there exists a vector $c_1$ in  \eqref{T-de-T}
such that 
$$
c - \de_T (c_1) \in T \big( \bbK \L \bsi a_{i_1}, \bsi a_{i_2}, \dots, \bsi a_{i_{k-1}} \R \big)\,.
$$

Combining this observation with the fact that the 
subcomplex  \eqref{Tpr-1-2-k} is a direct summand in  \eqref{T-de-T}, 
we conclude that, for every cocycle $c$ in   \eqref{Tpr-1-2-k} there 
exists a vector $c_1$  in   \eqref{Tpr-1-2-k}  such that 
$$
c- \de_T(c_1) \in  T'(\bsi a, \bsi a_{i_1}, \bsi a_{i_2}, \dots, \bsi  a_{i_{k-1}})
~\cap~  T \big( \bbK \L \bsi a_{i_1}, \bsi a_{i_2}, \dots, \bsi a_{i_{k-1}} \R \big)\,.
$$

Since the map $\nu$ \eqref{nu} is an isomorphism from the 
cochain complex   \eqref{Tpr-1-2-k} with the differential $\de_T$ to 
the cochain complex \eqref{La-Lie-prpr} with the differential $\de$, 
we deduce that every cocycle in \eqref{La-Lie-prpr} is cohomologous to 
a unique cocycle in the intersection
$$
\La\Lie'(a,a_{i_1}, \dots, a_{i_k}) \cap \La\Lie''(a_1, \dots, a_n)\,.
$$

Therefore  every cocycle in $\La\Lie'(a,a_1, \dots, a_n)$
is cohomologous to a unique cocycle in the subcomplex
$$ 
\La\Lie''(a_1, \dots, a_n)\,.
$$  

Lemma \ref{lem:La-Lie-a-pr} is proved.  ~~~$\Box$

\subsubsection{The operads $\As$ and $\As_{\infty}$ are $\Tw$-coalgebras 
and are homotopy fixed points for $\Tw$}

Recall that $\As$ is the operad which governs associative 
algebras without unit. 

We have the obvious map of operads
$$
\Lie \to \As
$$
and hence the maps of operads
\begin{equation}
\label{LaLie-LaAs}
\mi : \La \Lie \to \La \As\,,
\end{equation}
and
\begin{equation}
\label{LaLie-infty-LaAs}
\La \Lie_{\infty} \to \La \As\,.
\end{equation}
In other words, the operad $\La\As$ is naturally an object of 
the under-category  $\La \Lie_{\infty}\downarrow \Operads$. 

Let us show that 
\begin{prop}
\label{prop:LaAs}
The operad $\La\As$ is naturally a $\Tw$-coalgebra. 
Furthermore, $\La\As$ is a homotopy fixed point for the functor $\Tw$\,.
\end{prop}
\begin{remark}
\label{rem:La-As}
Before we proceed to the proof, we should remark that 
the associativity law for $\La \As$-algebras has a ``funny''
sign factor. Namely, a $\La\As$-algebra structure on a graded vector space $V$
is a degree $-1$ binary operation $\cdot$ on $V$ satisfying the associativity condition
$$
(v_1 \cdot v_2) \cdot v_3 = - (-1)^{|v_1|} v_1 \cdot (v_2 \cdot v_3)\,.
$$
\end{remark}
\begin{proof}
Although the operad $\La\As$ is not obtained\footnote{Instead, $\La\As$ carries 
a natural filtration, such that $\Gr\, \La\As \cong \La \Com \odot \La \Lie$\,.} 
via taking a plethysm of  $\La\Lie$ with another operad, 
it is not hard to see that for every vector $v \in \La \As(n)$
\begin{equation}
\label{beta-ga-As}
\mi(\{a_1, a_2\}) \circ_2 v = (-1)^{|v|} \sum_{i=1}^n \vs_{1,i} \big( v \circ_i \mi(\{a_1, a_2\})  \big)\,,     
\end{equation}
where $\vs_{1,i}$ is the cycle $(1,2, \dots, i)$\,.

Using this identity, it is easy to see that the canonical embedding
$$
\emb : \La \As \to \Tw\, \La \As
$$
\begin{equation}
\label{emb-LaAs}
\big(\emb(v)\big)(1_r) = 
\begin{cases}
 v \qquad {\rm if} ~~ r = 0 \,, \\
 0  \qquad {\rm otherwise}
\end{cases}
\end{equation}
is compatible with the differentials, i.e.
$$
\pa^{\Tw}\big( \emb(v) \big) = 0\,, \qquad \forall~~ v \in \La\As(n)\,. 
$$

A straightforward verification shows that the map $\emb$ 
\eqref{emb-LaAs} satisfies the axioms of the $\Tw$-coalgebra.  

To show that $\La\As$ is a homotopy fixed point for $\Tw$, we consider the 
free $\La\As$-algebra $\La\As(a, a_1, \dots, a_n)$ in $n$ dummy variables $a_1, \dots, a_n$ of 
degree $0$ and one dummy variable $a$ of degree $2$. 

The algebra  $\La\As(a, a_1, \dots, a_n)$ carries the differential $\de$ defined 
by the formulas:
\begin{equation}
\label{de-TwLaAs}
\de (a) = a \cdot a\,, \qquad \de (a_i) = 0\,, \qquad \forall~~ 1\le i \le n\,.
\end{equation}

Next, we denote by  $\La\As'(a, a_1, \dots, a_n)$ the subcomplex of 
$\La\As(a, a_1, \dots, a_n)$ which is spanned by $\La\As$-monomials in which 
each variable from the set $\{ a_1, \dots, a_n\}$ appears exactly once. 
It is not hard to see that the cochain complexes 
$$
\Tw\, \La\As (n) \qquad \textrm{and} \qquad \La\As'(a, a_1, \dots, a_n)
$$ 
are isomorphic and, moreover, the natural embedding
$$
\La\As(n) \hookrightarrow  \La\As'(a, a_1, \dots, a_n)
$$
induces an isomorphism on the level of cohomology.

This observation implies that the embedding $\emb$ \eqref{emb-LaAs}
is a quasi-isomorphism of dg operads. Hence $\La\As$ is indeed a 
homotopy fixed point for the functor $\Tw$\,.
\end{proof}

According to Section \ref{sec:Tw-general}, we may ask the 
same questions about the operad $\As$ keeping in mind 
the canonical map $\Lie \to \As$.  
For this case, Proposition \ref{prop:LaAs} 
gives us the following obvious corollary:
\begin{corollary}
\label{cor:As}
The operad $\As$ is naturally a $\Tw$-coalgebra. 
Furthermore, $\As$ is a homotopy fixed point for the functor $\Tw$\,. ~~~$\Box$
\end{corollary}

Let us denote by $\As_{\infty}$ the dg operad which governs 
$A_{\infty}$-algebras, i.e. 
\begin{equation}
\label{A-infty}
\As_{\infty} = \Cobar(\La \coAs)\,. 
\end{equation}
 
We claim that
\begin{corollary}
\label{cor:As-infty}
The dg operad $\As_{\infty}$ is naturally a $\Tw$-coalgebra. 
Furthermore, $\As_{\infty}$ is a homotopy fixed point for the functor $\Tw$\,. 
\end{corollary}
\begin{proof}
The second claim is an obvious consequence of Theorem \ref{thm:Twpreservesqiso} 
and Corollary \ref{cor:As}. So it remains to prove the first claim.

Because of sign factor, it is more convenient to prove that 
the dg operad 
\begin{equation}
\label{La-A-infty}
\La^{-1} \As_{\infty} = \Cobar(\coAs) 
\end{equation}
is a $\Tw$-coalgebra. Then the desired statement will follow
from the arguments of Section \ref{sec:Tw-general}. 

Our goal is to produce a map of dg operads 
\begin{equation}
\label{c-As-infty}
c :   \Cobar(\coAs) \to \Tw \Cobar(\coAs)
\end{equation}
and verify axioms of the coalgebra over $\Tw$\,.

For this purpose we recall that $\As(n)$ has the 
canonical basis 
\begin{equation}
\label{As-n-basis}
\Big\{ a_{\si(1)} a_{\si(2)} \dots a_{\si(n)} \Big\}_{\si \in S_n}\,, 
\end{equation}
where $a_1, a_2, \dots, a_n$ are dummy variables of degree zero. 

Next, we denote by $(a_{\si(1)} a_{\si(2)} \dots a_{\si(n)})^*$ 
vectors of the dual basis in $\coAs(n) = \big( \As(n) \big)^*$ and 
observe that the formula $(n \ge 2, r \ge 0)$
\begin{equation}
\label{al-A-infty}
\al \big( (a_{\si(1)} a_{\si(2)} \dots a_{\si(n)})^* \big) (1_r) = 
\sum_{\la \in \Sh_{r,n}}
\sum_{\tau \in S_r} \bs \, 
 \la \, ( a_{\tau(1)} \dots a_{\tau(r)} a_{r + \si(1)} a_{r + \si(2)} \dots a_{r +\si(n)})^*
\end{equation}
defines a degree $1$ vector in the dg Lie algebra $\Conv( \coAs_{\c} ,  \Tw \Cobar(\coAs) )$\,.
Here 
$$
 \bs \, 
 \la \, ( a_{\tau(1)} \dots a_{\tau(r)}  a_{r + \si(1)} a_{r + \si(2)} \dots a_{r +\si(n)})^*
$$
is viewed as a vector in $\bs \, \coAs(r+n) \subset  \Cobar(\coAs)$\,.

A direct computation shows that $\al$ is a Maurer-Cartan element of the 
dg Lie algebra
$$
\Conv \big( \coAs_{\c} ,  \Tw \Cobar(\coAs) \big)
$$

Hence, by Theorem \ref{thm:from-Cobar}, $\al$ gives 
us a map of dg operads \eqref{c-As-infty}. 

It is easy to check that this map satisfies all the axioms of the 
coalgebra over the comonad $\Tw$. So we leave the verification of 
these axioms to the reader.  
\end{proof}

\begin{remark}
\label{rem:Lazarev}
In Section 7 of \cite{C-Lazarev}, J. Chuang and A. Lazarev also 
proved that $\As$ and $\As_{\infty}$ are homotopy fixed points for the functor $\Tw$\,. 
\end{remark}

\subsubsection{Example: a $\Tw$-coalgebra which is not a homotopy fixed 
point for $\Tw$}
\label{sec:example-non-fixed}

We constructed a large class of $\Tw$-coalgebras each of which 
is a  homotopy fixed point for $\Tw$\,. Let us now give an example of 
a  $\Tw$-coalgebra which is not a homotopy fixed 
point for $\Tw$\,.
Let $P$ be an operad in the category $\grVect_{\bbK}$ such that $P(n) \neq 0$
for all $n \ge 1$\,. (For example, the operad $\As$ would work.) 
Let 
$$
\hat{\vf} : \La \Lie_{\infty} \to P
$$
be the zero map. 
Then the dg operads  $\Tw P$ and $\Tw (\Tw P)$ with the spaces 
\begin{equation}
\label{Tw-P-n}
\Tw P (n) = \prod_{r \ge 0} \bs^{2r} \big( P(r+n) \big)^{S_r} 
\end{equation}
and
\begin{equation}
\label{TwTw-P-n}
\Tw (\Tw P) (n) = \prod_{r,s \ge 0} \bs^{2r+2s} \big( P(r+s+n) \big)^{S_r \times S_s} 
\end{equation}
carry the zero differentials.
The operad $\Tw P$ is a $\Tw$-coalgebra by construction, however, 
\begin{prop}
\label{prop:TwP-non-fixed}
The operad $\Tw P$ is not a homotopy fixed point for $\Tw$\,.
\end{prop}
\begin{proof}
Let $r, n, s$ be a positive integers and let $v$ be a non-zero vector 
in $P(r+s+n)$. Such a vector exists, since $P(n)\neq 0$ for all $n \ge 1$\,. 
The sum 
\begin{equation}
\label{Av-Av-v}
\sum_{\si \in S_r \times S_s} \bs^{2r+2s} \si(v)
\end{equation}
can be viewed as a vector in $\Tw(\Tw P)(n)$\,.
Since the differential on $\Tw(\Tw P)$ is non-zero, then vector 
\eqref{Av-Av-v} is a non-trivial cocycle. 
On the other hand, since $r \ge 1$, the counit map 
$$
\eta_{\Tw P}  : \Tw(\Tw P) \to \Tw P
$$ 
sends vector \eqref{Av-Av-v} to zero. 
Thus $\eta_{\Tw P}$ is not a quasi-isomorphism.
\end{proof}

\section{Twisting $\cO$-algebra structures by  Maurer-Cartan elements}
\label{sec:twisting-algebras} 
Let $\cO$ be an object of the under-category  $\La \Lie_{\infty}\downarrow \Operads$\,. 
Following Section \ref{sec:Tw-cO-alg} we consider the category $\Alg^{\MC}_{\cO}$
whose objects are pairs 
$$
(V, \al)\,,
$$
where $V$ is an $\cO$-algebra equipped with a complete descending
filtration\footnote{As above, we assume that the  $\cO$-algebra structure on $V$ is 
compatible with the filtration \eqref{filtr-V}.}  and  $\al \in \cF_1 V^2$ is 
a Maurer-Cartan element. Morphisms in $\Alg^{\MC}_{\cO}$ are morphisms 
of filtered $\cO$-algebras $f : V \to V'$ which send $\al$ to $\al'$.  

Recall that  Theorem \ref{thm:twisting} yields a functor
\begin{equation}
\label{mfF}
\mfF : \Alg^{\MC}_{\cO} \to \Alg^{\filtr}_{\Tw \cO}
\end{equation}
 from  the category $\Alg^{\MC}_{\cO}$
to the category $\Alg^{\filtr}_{\Tw \cO}$ of filtered $\Tw\cO$-algebras. 
More precisely, the functor $\mfF$ assigns to a pair $(V,\al)$  the 
cochain complex $V^{\al}$ with the differential $\pa^{\al}$ \eqref{tw-diff} and the $\Tw\cO$-algebra 
structure defined by equation \eqref{twisting}.

If the arrow $\hat{\vf} : \La\Lie_{\infty}  \to \cO$ is a $\Tw$-coalgebra then 
we have a morphism of dg operads
\begin{equation}
\label{cO-TwcO-here}
c: \cO \to \Tw \cO
\end{equation}
which induces a functor 
\begin{equation}
\label{c-star}
c^* : \Alg^{\filtr}_{\Tw \cO} \to \Alg_{\cO}
\end{equation}
from the category $ \Alg^{\filtr}_{\Tw \cO}$ to the category $\Alg_{\cO}$ 
of $\cO$-algebras.  

Composing \eqref{mfF} and \eqref{c-star} we get the functor 
\begin{equation}
\label{ATw}
\ATw : = c^* \circ \mfF : \Alg^{\MC}_{\cO} \to \Alg_{\cO}\,.
\end{equation}
As a cochain complex $\ATw(V, \al)$ is  $V^{\al}$ with the differential $\pa^{\al}$ 
\eqref{tw-diff} and the $\cO$-algebra structure is induced by the map $c: \cO \to \Tw \cO$\,.
 
For a pair $(V,\al)\in  \Alg^{\MC}_{\cO}$, we say that the $\cO$-algebra  
$\ATw(V,\al)$ is obtained from $V$ via {\it twisting by a Maurer-Cartan element} $\al$\,.
We refer to the construction of $\ATw(V,\al)$ as {\it the twisting procedure}\footnote{Note that this
twisting procedure for $\cO$-algebras is defined only if the dg operad $\cO$ is a coalgebra 
over the comonad $\Tw$.} for $\cO$-algebras.  

Axioms of a coalgebra over the comonad $\Tw$ listed in Section \ref{sec:Tw-coalg}
imply that the twisting procedure satisfies the following properties:
  
\begin{itemize}
\item[{\bf P1}] Considered as the $\La \Lie_\infty$-algebra, $\ATw(V,\al)$  is obtained 
via twisting the $\La \Lie_\infty$ algebra $V$ by the Maurer-Cartan element $\al$. 

\item[{\bf P2}] Twisting by the zero Maurer-Cartan element $\al = 0$ does not change the $\cO$-algebra, 
i.e., $\ATw(V,0) = V$\,. 

\item[{\bf P3}] If $\al$ is a  Maurer-Cartan element of $V$ and $\al'$ is a  Maurer-Cartan element 
of $\ATw(V,\al)$  then the twisted $\cO$-algebras 
$$
\ATw\big( \ATw(V,\al), \al' \big) 
\qquad \textrm{and} \qquad 
\ATw(V, \al + \al')
$$
coincide.
\end{itemize}

Indeed, Property {\bf P1} follows easily from commutativity of diagram \eqref{diag-LaLie-infty}.

Equation \eqref{twisting} implies that, if $\al = 0$, then the $\Tw\cO$-algebra structure 
on $V^{\al} = V$ factors through the counit map $\eta_{\cO} : \Tw\cO \to \cO$. 
Thus  Property {\bf P2} holds because the composition \eqref{comp-c-eta} is 
the identity map on $\cO$\,.

Finally, Property {\bf P3} follows from commutativity of diagram \eqref{diag-c-Twc-mD}.

Let us now consider functorial properties of $\ATw$ with respect to morphisms 
in the under-category  $\La \Lie_{\infty}\downarrow \Operads$\,. 

For any arrow $\Psi : \cO \to \cO'$ which fits into the commutative diagram  
\begin{equation}
\label{Psi-diag}
\begin{tikzpicture}
\matrix (m) [matrix of math nodes, row sep=2em]
{  & \La \Lie_{\infty} &  \\
\cO &  & \cO' \\ };
\path[->,font=\scriptsize]
(m-1-2) edge node[above] {$\hat{\vf}~~$} (m-2-1) edge node[auto] {$\hat{\vf}'$} (m-2-3)
(m-2-1) edge node[auto] {$\Psi$} (m-2-3);
\end{tikzpicture}
\end{equation}
we can obviously extend the functor 
$$
\Psi^* : \Alg_{\cO'} \to \Alg_{\cO}
$$
to
\begin{equation}
\label{Psi-star}
\Psi^* :  \Alg^{\MC}_{\cO'} \to \Alg^{\MC}_{\cO}\,.
\end{equation}

Furthermore, we observe that the diagram
\begin{equation}
\label{Psi-star-TwPsi-star}
\begin{tikzpicture}
\matrix (m) [matrix of math nodes, row sep=3em, column sep=4em]
{\Alg^{\MC}_{\cO'}  & \Alg^{\MC}_{\cO}  \\
\Alg^{\filtr}_{\Tw\cO'} & \Alg^{\filtr}_{\Tw\cO} \\ };
\path[->,font=\scriptsize]
(m-1-1) edge node[auto] {$\Psi^*$} (m-1-2) edge node[auto] {$\mfF$} (m-2-1)
(m-2-1) edge node[auto] {$\big(\Tw\Psi\big)^*$} (m-2-2)
 (m-1-2) edge node[auto] {$\mfF$} (m-2-2);
\end{tikzpicture}
\end{equation}
commutes because the construction of the functor $\mfF$ \eqref{Fun-AlgMC-AlgTw}
is functorial in $\cO$\,.

If the dg operads $\cO$ and $\cO'$ are $\Tw$-coalgebras and the map $\Psi$ is 
compatible with the $\Tw$-coalgebra structures then the diagram   
\begin{equation}
\label{TwPsi-star-c-c-pr}
\begin{tikzpicture}
\matrix (m) [matrix of math nodes, row sep=3em, column sep=4em]
{\Alg^{\filtr}_{\Tw\cO'} & \Alg^{\filtr}_{\Tw\cO} \\ 
\Alg_{\cO'}  & \Alg_{\cO}  \\};
\path[->,font=\scriptsize]
(m-1-1) edge node[auto] {$\big( \Tw\Psi \big)^*$} (m-1-2) edge node[auto] {$(c')^*$} (m-2-1)
(m-2-1) edge node[auto] {$\Psi^*$} (m-2-2)
 (m-1-2) edge node[auto] {$c^*$} (m-2-2);
\end{tikzpicture}
\end{equation}
commutes. 

Thus, combining commutative diagrams \eqref{Psi-star-TwPsi-star}
and \eqref{TwPsi-star-c-c-pr} we arrive at the following functorial 
property of the twisting procedure. 
\begin{thm}
\label{thm:atwfunctor}
If $\Psi$ is a map of $\Tw$-coalgebras $\cO$ and $\cO'$ then the diagram 
\begin{equation}
\label{funct-ATw}
\begin{tikzpicture}
\matrix (m) [matrix of math nodes, row sep=3em, column sep=4em]
{\Alg^{\MC}_{\cO'} & \Alg^{\MC}_{\cO} \\ 
\Alg_{\cO'}  & \Alg_{\cO}  \\};
\path[->,font=\scriptsize]
(m-1-1) edge node[auto] {$\Psi^*$} (m-1-2) edge node[auto] {$\ATw_{\cO'}$} (m-2-1)
(m-2-1) edge node[auto] {$\Psi^*$} (m-2-2)
 (m-1-2) edge node[auto] {$\ATw_{\cO}$} (m-2-2);
\end{tikzpicture}
\end{equation}
commutes. $~~~\Box$
\end{thm}

At the end of this section, we observe that, under a mild technical condition,  
the functor $\ATw$ preserves quasi-isomorphisms. More precisely,
\begin{prop}
\label{prop:ATw-preserves}
Let $\cO$ be a dg operad equipped with a $\Tw$-coalgebra structure and let 
$f : (V, \al) \to (V', \al')$ be a morphism in $\Alg^{\MC}_{\cO}$\,. 
Let us assume that the filtrations on $V$ and $V'$ are bounded above. If $f$ is a quasi-isomorphism of the underlying cochain complexes
$(V, \pa)$ and $(V', \pa')$ then $f$ is also a quasi-isomorphism 
from $\ATw(V, \al)$ to $\ATw(V', \al')$\,.
\end{prop}
\begin{proof}
The cone 
$$
C = V^{\al} \oplus \bs (V')^{\al'}
$$
of the morphism $f : V^{\al} \to  (V')^{\al'}$ carries the obvious descending 
filtration 
$$
\cF_m C: =
\cF_m V^{\al}\, \oplus \, \bs \cF_m (V')^{\al'}\,.
$$
This filtration is complete and bounded above. 

Furthermore, the associated graded complex 
$$
\Gr C 
$$
is isomorphic to the cone 
$$
V \oplus \bs V'
$$
of the morphism $f: V \to V'$\,.

Hence $\Gr C$ is acyclic and, by Lemma \ref{lem:filtered}, the cone 
of the morphism $f : V^{\al} \to  (V')^{\al'}$ is also acyclic. 
 
Proposition \ref{prop:ATw-preserves} follows. 
\end{proof}

\section{The operad of brace trees \texorpdfstring{$\BT$}{BT}}
\label{sec:BT}
The remaining four sections of the paper are devoted to an application 
of the developed machinery to Deligne's conjecture. We will introduce an auxiliary 
operad $\BT$, define the braces operad $\Br$ as a suboperad of 
$\Tw \BT$, and finally, prove Theorem \ref{thm:main} stated in the 
Introduction. 

Let us start by defining the operad of brace trees $\BT$. 

To define the space $\BT(n)$ we introduce an auxiliary set $\cT(n)$\,. 
An element of this set is a planar tree $T$ equipped with 
a bijection between $\{1,2, \dots, n\}$ and the set 
$$
V(T) \setminus \{\textrm{root vertex}\}
$$
of non-root vertices. For $n=0$ the set $\cT(n)$ is empty.

We call elements of $\cT(n)$ {\it brace trees}. 
Examples of brace trees are shown on figures \ref{fig:T6}
and \ref{fig:id-BT}. 
\begin{figure}[htp] 
\begin{minipage}[t]{0.45\linewidth} 
\centering 
\begin{tikzpicture}[scale=0.5]
\tikzstyle{lab}=[circle, draw, minimum size=5, inner sep=1]
\tikzstyle{n}=[circle, draw, fill, minimum size=5]
\tikzstyle{root}=[circle, draw, fill, minimum size=0, inner sep=1]
\node[root] (r) at (0, 0) {};
\node [lab] (v1) at (0,1) {$2$};
\node [lab] (v2) at (-1,2) {$3$};
\node [lab] (v3) at (-2.2,3.3) {$1$};
\node [lab] (v4) at (-1,3.3) {$5$};
\node [lab] (v5) at (0.2,3.3) {$4$};
\node [lab] (v6) at (1,2) {$6$};
\draw (r) edge (v1);
\draw (v1) edge (v2);
\draw (v1) edge (v6);
\draw (v2) edge (v3);
\draw (v2) edge (v4);
\draw (v2) edge (v5);
\end{tikzpicture}
\caption{A brace tree $T' \in \cT(6)$} \label{fig:T6}
\end{minipage}
~
\begin{minipage}[t]{0.45\linewidth} 
\centering 
\begin{tikzpicture}[scale=0.5]
\tikzstyle{lab}=[circle, draw, minimum size=5, inner sep=1]
\tikzstyle{n}=[circle, draw, fill, minimum size=5]
\tikzstyle{root}=[circle, draw, fill, minimum size=0, inner sep=1]
\node[root] (r) at (0, 0) {};
\node [lab] (v1) at (0,1.5) {$1$};
\draw (r) edge (v1);
\end{tikzpicture}
\caption{The brace tree $T_{\id}\in \cT(1)$} \label{fig:id-BT}
\end{minipage}
\end{figure} 
On figures, non-root vertices are depicted by white circles  
with the corresponding numbers inscribed.

The $n$-th space $\BT(n)$ of $\BT$ consists of  linear combinations 
of elements in $\cT(n)$\,. 
 The structure of a graded vector 
space on $\BT(n)$ is obtained by  declaring that each non-root
edge carries degree $-1$\,.
In other words, for every $T \in \cT(n)$ we have  
\begin{equation}
\label{deg-T}
|T| = 1 - |E(T)|\,, 
\end{equation}
where $E(T)$ is the set of all edges of $T$\,.
A simple combinatorics shows that for every brace tree 
$T \in \cT(n)$ we have 
\begin{equation}
\label{deg-T-n}
|T| = 1 - n\,. 
\end{equation}
Hence, the graded vector space $\BT(n)$ is concentrated in the single degree $1-n$\,.
(In particular, the operad $\BT$ may only carry the zero differential.)

Since $\cT(0)$ is empty, we have 
\begin{equation}
\label{BT-0}
\BT(0) =\bfzero\,.
\end{equation}
Furthermore, since
in $\cT(1)$ we have only element $T_{\id}$ (see figure \ref{fig:id-BT}), 
\begin{equation}
\label{BT-1}
\BT(1) = \bbK\,.
\end{equation}

\subsection{The operad structure on \texorpdfstring{$\BT$}{BT}}
Let $T \in \cT(n) $, $T' \in \cT(k)$ and $1\le i \le k$. Our goal 
is to define the output of the elementary insertion  $T' \circ_i  T  \in \BT(n+k-1)$\,.

Let $v_i$ be the non-root vertex of $T'$ with label $i$\,.
If $v_i$ is a leaf (i.e. $v_i$ does not have 
incoming edges) then the vector $T' \circ_i T \in \BT(n+k-1)$ is, up to sign, represented by a brace 
tree $T''$ which is obtained from $T'$ by erasing the vertex $v_i$ and gluing the 
brace tree $T$ via identifying the root edge of $T$ with the edge originating 
at $v_i$\,. After this operation we relabel elements of the set 
$$
V(T'') \setminus  \{\textrm{root vertex}\} = \big( V(T)  \setminus \{\textrm{root vertex}\} \big)
 \sqcup (V(T') \setminus \{v_i\})
$$
in the obvious way. The sign factor in front of $T''$ is obtained by keeping 
track of the reordering of non-root edges of $T'$ and $T$. 

Let us now consider the case when $v_i$ has $q \ge 1$ incoming 
edges. Since $T'$ is a planar tree, these $q$ incoming edges 
are totally ordered. So we denote them by $e_1, e_2, \dots, e_q$
keeping in mind that 
\begin{equation}
\label{edges-are-ordered}
e_1 < e_2 < \dots < e_q\,.
\end{equation}

The desired vector $T' \circ_i T \in \BT(n+k-1)$ is represented 
by the sum 
\begin{equation}
\label{circ-i-BT}
T' \circ_i T = \sum_{\al} (-1)^{f(\al)} T_{\al} 
\end{equation}
where $ T_{\al}$ is obtained from $T'$ and $T$ following these steps: 
\begin{itemize}

\item first, we erase the vertex $v_i$ and 
glue the brace tree $T$  via identifying the root edge $T$ with 
the edge originating from $v_i$\,,

\item second, we attach the edges $e_1, e_2, \dots, e_q$ to 
vertices in the set 
\begin{equation}
\label{non-root-V-T}
V(T)  \setminus  \{\textrm{root vertex}\}\,.
\end{equation}
 
\item finally, we relabel elements of the set 
$$
V(T'') \setminus  \{\textrm{root vertex}\} = \big( V(T)  \setminus \{\textrm{root vertex}\} \big)
 \sqcup (V(T') \setminus \{v_i, \textrm{root vertex} \})
$$
in the  obvious way. 
\end{itemize}
Ways of connecting the edges $e_1, e_2, \dots, e_q$ to 
vertices in the set \eqref{non-root-V-T} should 
satisfy the following condition
\begin{cond} 
\label{cond:connect-BT}
The restriction of the total order on the 
set $E(T_{\al})$ of $T_{\al}$ to the subset 
$\{e_1, e_2, \dots, e_q\}$ should coincide 
with the order \eqref{edges-are-ordered}\,.
\end{cond}   
This condition can be reformulated in geometric 
terms as follows.  If we choose a small tubular neighborhood of the tree
$T$ (drawn on the plane) and walk along its boundary starting from
the root vertex in the clockwise direction then we will cross the edges 
in this order: first, we will cross $e_1$, second, we will cross $e_2$, 
third, we will cross $e_3$,  and so on.  

The summation in \eqref{circ-i-BT} goes over all ways
$\al$ of connecting the edges $e_1, e_2, \dots, e_q$ to 
vertices in the set \eqref{non-root-V-T} satisfying 
Condition \ref{cond:connect-BT}.

To define the sign factors $(-1)^{f(\al)}$ in  \eqref{circ-i-BT} 
we extend the total orders on the sets $E(T') $ and 
$(E(T)\setminus \{\textrm{root edge}\})$ to the disjoint 
union  
\begin{equation}
\label{ET-ETpr}
E(T') \sqcup (E(T)\setminus \{\textrm{root edge}\})
\end{equation}
by declaring that all elements of  $E(T')$ are smaller 
than elements of  $E(T)\setminus \{\textrm{root edge}\}$.

Next we observe that the set \eqref{ET-ETpr} is 
naturally isomorphic to the set $E(T_{\al})$ of edges of 
$T_{\al}$\,. On the other hand, the set  $E(T_{\al})$ carries 
a possibly different total order coming from planar structure 
on $T_{\al}$\,.

So the factor  $(-1)^{f(\al)}$ is the sign of the permutation 
which connects these total orders on the set 
$$
E(T_{\al}) \cong E(T') \sqcup (E(T)\setminus \{\textrm{root edge}\})\,.
$$

\begin{example}
\label{ex:BT-ins}
Let $T'$ (resp. $T_{\cc}$) be the brace tree depicted on figure 
\ref{fig:T6} (resp. figure \ref{fig:T12}). The result of the insertion 
$T' \c_2 T$ is the sum of brace trees shown on figure 
\ref{fig:BT-ins}.
\begin{figure}[htp] 
\centering 
\begin{tikzpicture}[scale=0.5]
\tikzstyle{lab}=[circle, draw, minimum size=5, inner sep=1]
\tikzstyle{n}=[circle, draw, fill, minimum size=5]
\tikzstyle{root}=[circle, draw, fill, minimum size=0, inner sep=1]
\node[root] (r) at (0, 0) {};
\node [lab] (v1) at (0,1) {$1$};
\node [lab] (v2) at (0,2.2) {$2$};
\draw (r) edge (v1);
\draw (v1) edge (v2);
\end{tikzpicture}
\caption{The brace tree $T_{\cc} \in \cT(2)$} \label{fig:T12}
\end{figure}
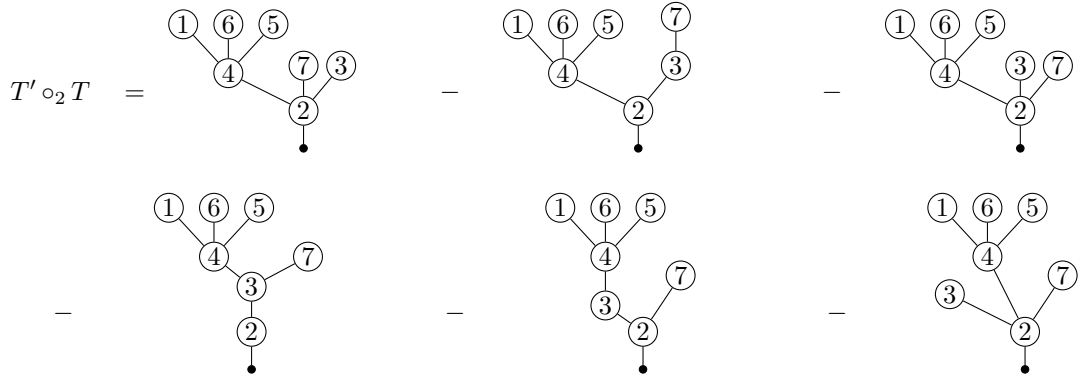
\begin{figure}[htp] 
\begin{minipage}[t]{0.3\linewidth} 
\centering 
\begin{tikzpicture}[scale=0.5]
\tikzstyle{lab}=[circle, draw, minimum size=5, inner sep=1]
\tikzstyle{n}=[circle, draw, fill, minimum size=5]
\tikzstyle{root}=[circle, draw, fill, minimum size=0, inner sep=1]
\draw (-6,1.5) node[anchor=center] {{$T' \,\c_2\, T \quad =$}};
\node[root] (r) at (0, 0) {};
\node [lab] (v1) at (0,1) {$2$};
\node [lab] (v2) at (-2,2) {$4$};
\node [lab] (v3) at (-3.2,3.3) {$1$};
\node [lab] (v4) at (-2,3.3) {$6$};
\node [lab] (v5) at (-0.8,3.3) {$5$};
\node [lab] (v6) at (0,2.2) {$7$};
\node [lab] (vv) at (1,2.2) {$3$};
\draw (r) edge (v1);
\draw (v1) edge (v2);
\draw (v1) edge (v6);
\draw (v1) edge (vv);
\draw (v2) edge (v3);
\draw (v2) edge (v4);
\draw (v2) edge (v5);
\end{tikzpicture}
\end{minipage}
\begin{minipage}[t]{0.3\linewidth} 
\centering 
\begin{tikzpicture}[scale=0.5]
\tikzstyle{lab}=[circle, draw, minimum size=5, inner sep=1]
\tikzstyle{n}=[circle, draw, fill, minimum size=5]
\tikzstyle{root}=[circle, draw, fill, minimum size=0, inner sep=1]
\draw (-5,1.5) node[anchor=center] {{$-$}};
\node[root] (r) at (0, 0) {};
\node [lab] (v1) at (0,1) {$2$};
\node [lab] (v2) at (-2,2) {$4$};
\node [lab] (v3) at (-3.2,3.3) {$1$};
\node [lab] (v4) at (-2,3.3) {$6$};
\node [lab] (v5) at (-0.8,3.3) {$5$};
\node [lab] (v6) at (1,3.5) {$7$};
\node [lab] (vv) at (1,2.2) {$3$};
\draw (r) edge (v1);
\draw (v1) edge (v2);
\draw (vv) edge (v6);
\draw (v1) edge (vv);
\draw (v2) edge (v3);
\draw (v2) edge (v4);
\draw (v2) edge (v5);
\end{tikzpicture}
\end{minipage}
\begin{minipage}[t]{0.3\linewidth} 
\centering 
\begin{tikzpicture}[scale=0.5]
\tikzstyle{lab}=[circle, draw, minimum size=5, inner sep=1]
\tikzstyle{n}=[circle, draw, fill, minimum size=5]
\tikzstyle{root}=[circle, draw, fill, minimum size=0, inner sep=1]
\draw (-5,1.5) node[anchor=center] {{$-$}};
\node[root] (r) at (0, 0) {};
\node [lab] (v1) at (0,1) {$2$};
\node [lab] (v2) at (-2,2) {$4$};
\node [lab] (v3) at (-3.2,3.3) {$1$};
\node [lab] (v4) at (-2,3.3) {$6$};
\node [lab] (v5) at (-0.8,3.3) {$5$};
\node [lab] (v6) at (0,2.2) {$3$};
\node [lab] (vv) at (1,2.2) {$7$};
\draw (r) edge (v1);
\draw (v1) edge (v2);
\draw (v1) edge (v6);
\draw (v1) edge (vv);
\draw (v2) edge (v3);
\draw (v2) edge (v4);
\draw (v2) edge (v5);
\end{tikzpicture}
\end{minipage}
\begin{minipage}[t]{\linewidth} 
\vspace{0.5cm}
\end{minipage}
\begin{minipage}[t]{0.3\linewidth} 
\centering 
\begin{tikzpicture}[scale=0.5]
\tikzstyle{lab}=[circle, draw, minimum size=5, inner sep=1]
\tikzstyle{n}=[circle, draw, fill, minimum size=5]
\tikzstyle{root}=[circle, draw, fill, minimum size=0, inner sep=1]
\draw (-5,1.5) node[anchor=center] {{$-$}};
\node[root] (r) at (0, 0) {};
\node [lab] (v1) at (0,1) {$2$};
\node [lab] (v2) at (-1,3) {$4$};
\node [lab] (v3) at (-2.2,4.3) {$1$};
\node [lab] (v4) at (-1,4.3) {$6$};
\node [lab] (v5) at (0.2,4.3) {$5$};
\node [lab] (v6) at (1.5,3) {$7$};
\node [lab] (vv) at (0,2.2) {$3$};
\draw (r) edge (v1);
\draw (vv) edge (v2);
\draw (vv) edge (v6);
\draw (v1) edge (vv);
\draw (v2) edge (v3);
\draw (v2) edge (v4);
\draw (v2) edge (v5);
\end{tikzpicture}
\end{minipage}
\begin{minipage}[t]{0.3\linewidth} 
\centering 
\begin{tikzpicture}[scale=0.5]
\tikzstyle{lab}=[circle, draw, minimum size=5, inner sep=1]
\tikzstyle{n}=[circle, draw, fill, minimum size=5]
\tikzstyle{root}=[circle, draw, fill, minimum size=0, inner sep=1]
\draw (-5,1.5) node[anchor=center] {{$-$}};
\node[root] (r) at (0, 0) {};
\node [lab] (v1) at (0,1) {$2$};
\node [lab] (v2) at (-1,3) {$4$};
\node [lab] (v3) at (-2.2,4.3) {$1$};
\node [lab] (v4) at (-1,4.3) {$6$};
\node [lab] (v5) at (0.2,4.3) {$5$};
\node [lab] (v6) at (1,2.5) {$7$};
\node [lab] (vv) at (-1, 1.7) {$3$};
\draw (r) edge (v1);
\draw (vv) edge (v2);
\draw (v1) edge (v6);
\draw (v1) edge (vv);
\draw (v2) edge (v3);
\draw (v2) edge (v4);
\draw (v2) edge (v5);
\end{tikzpicture}
\end{minipage}
\begin{minipage}[t]{0.3\linewidth} 
\centering 
\begin{tikzpicture}[scale=0.5]
\tikzstyle{lab}=[circle, draw, minimum size=5, inner sep=1]
\tikzstyle{n}=[circle, draw, fill, minimum size=5]
\tikzstyle{root}=[circle, draw, fill, minimum size=0, inner sep=1]
\draw (-5,1.5) node[anchor=center] {{$-$}};
\node[root] (r) at (0, 0) {};
\node [lab] (v1) at (0,1) {$2$};
\node [lab] (v2) at (-1,3) {$4$};
\node [lab] (v3) at (-2.2,4.3) {$1$};
\node [lab] (v4) at (-1,4.3) {$6$};
\node [lab] (v5) at (0.2,4.3) {$5$};
\node [lab] (v6) at (1,2.5) {$7$};
\node [lab] (vv) at (-2, 2) {$3$};
\draw (r) edge (v1);
\draw (v1) edge (v2);
\draw (v1) edge (v6);
\draw (v1) edge (vv);
\draw (v2) edge (v3);
\draw (v2) edge (v4);
\draw (v2) edge (v5);
\end{tikzpicture}
\end{minipage}
\caption{The vector $T' \,\c_2\, T \in \BT(7)$} \label{fig:BT-ins}
\end{figure} 
\end{example}

The symmetric group $S_n$ acts on $\BT(n)$ in 
the obvious way by rearranging the labels. It is not hard to 
see that operations \eqref{circ-i-BT} together with this action 
give us an operad structure on $\BT$ with the identity 
element represented by the brace tree $T_{\id}$ depicted 
on figure \ref{fig:id-BT}.

\section{The operad \texorpdfstring{$\Tw \BT$}{Tw BT}}
\label{sec:TwBT}
The operad $\BT$ receives a natural map from $\La\Lie$
\begin{equation}
\label{LaLie-BT}
\vf: \La\Lie \to \BT \,.
\end{equation}
Since the operad  $\La\Lie$ is generated by the binary bracket operation 
$\{\cdot,\cdot\} \in \La\Lie(2)$, the map \eqref{LaLie-BT} is 
uniquely determined by its value  $\vf(\{\cdot,\cdot\})$, 
which equals   
\begin{equation}
\label{LaLie-BT-form}
\vf(\{\cdot,\cdot\}) =  T_{\cc} + \si_{12} T_{\cc}\,,
\end{equation}
where $T_{\cc}$ is the brace tree depicted on figure \ref{fig:T12}
and $\si_{12}$ is the transposition in $S_2$\,.
The desired compatibility with the Jacobi relation 
\begin{equation}
\label{vf-Jacobi}
\vf(\{\cdot,\cdot\}) \, \c_1 \, \vf(\{\cdot,\cdot\})  +  \textrm{cyclic permutations}\,(1,2,3) = 0
\end{equation}
can be checked by a straightforward computation.

In this section we give an explicit description for 
the twisted version $\Tw\BT$ of $\BT$ corresponding to the map 
\eqref{LaLie-BT-form}.
As a graded vector space,
\begin{equation}
\label{TwBT-n}
\Tw\BT(n) = \prod_{r = 0}^{\infty} \bs^{2r} 
\big( \BT(r+n) \big)^{S_r}\,.
\end{equation}
Using the observation that for every $m$ the space $\BT(m)$ is concentrated in the 
single degree $1-m$ we conclude that the subspace 
$\Tw\BT^p(n)$ of degree $p$ vectors in $\Tw\BT(n)$ 
is spanned by vectors of the form
\begin{equation}
\label{T-averaged}
\sum_{\si \in S_r} \si (T)
\end{equation}
where $T$ is an arbitrary brace tree in $\cT(r+n)$
and $r = p+n-1$\,. In particular, 
\begin{equation}
\label{TwBT-TwBT-oplus}
\Tw \BT = \Tw^{\oplus} \BT 
\end{equation}
(cf. Section \ref{sec:Tw-oplus}).

To represent vectors \eqref{T-averaged}, it is convenient 
to extend the set $\cT(n)$ to another auxiliary set  $\cT^{\tw}(n)$\,. 
An element  of $\cT^{\tw}(n)$ is a planar tree $T$ equipped with the 
following data: 
\begin{itemize}

\item a partition of the set $V(T)$ of vertices
$$
V(T) = V_{\lab}(T)  \sqcup V_{\nu}(T) \sqcup V_{root}(T)  
$$  
into the singleton $V_{root}(T)$ consisting of the root vertex, 
the set $V_{\lab}(T)$ consisting of $n$ vertices, and the set $V_{\nu}(T)$
consisting of vertices which we call {\it neutral};

\item a bijection between the set $V_{\lab}(T)$ and the set $\{1,2, \dots, n\}$;

\end{itemize}
We also call elements of $\cT^{\tw}(n)$ brace trees.
Figures \ref{fig:T12}, \ref{fig:cup}, \ref{fig:cup-opp}
show examples of brace trees in $\cT^{\tw}(2)$\,.
Figures \ref{fig:Tbul1} and \ref{fig:T1bul} show examples of 
a brace tree in $\cT^{\tw}(1)$\,. 
\begin{figure}[htp] 
\begin{minipage}[t]{0.45\linewidth}
\centering 
\begin{tikzpicture}[scale=0.5]
\tikzstyle{lab}=[circle, draw, minimum size=5, inner sep=1]
\tikzstyle{n}=[circle, draw, fill, minimum size=6, inner sep=0]
\tikzstyle{root}=[circle, draw, fill, minimum size=0, inner sep=1]
\node[root] (r) at (1, 0) {};
\node[n] (v) at (1, 1) {};
\node [lab] (v1) at (0.3,2.3) {$1$};
\node [lab] (v2) at (1.7,2.3) {$2$};
\draw (r) edge (v);
\draw (v) edge (v1);
\draw (v) edge (v2);
\end{tikzpicture}
\caption{A brace tree $T_{\cup}\in \cT^{\tw}(2)$} \label{fig:cup}
\end{minipage} 
\begin{minipage}[t]{0.45\linewidth}
\centering 
\begin{tikzpicture}[scale=0.5]
\tikzstyle{lab}=[circle, draw, minimum size=5, inner sep=1]
\tikzstyle{root}=[circle, draw, fill, minimum size=0, inner sep=1]
\node[root] (r) at (1, 0) {};
\node[n] (v) at (1, 1) {};
\node [lab] (v2) at (0.3,2.3) {$2$};
\node [lab] (v1) at (1.7,2.3) {$1$};
\draw (r) edge (v);
\draw (v) edge (v1);
\draw (v) edge (v2);
\end{tikzpicture}
\caption{A brace tree $T_{\cup^{opp}}\in \cT^{\tw}(2)$} \label{fig:cup-opp}
\end{minipage} 
\end{figure} 
\begin{figure}[htp] 
\begin{minipage}[t]{0.45\linewidth}
\centering 
 \begin{tikzpicture}[scale=0.5]
\tikzstyle{lab}=[circle, draw, minimum size=5, inner sep=1]
\tikzstyle{root}=[circle, draw, fill, minimum size=0, inner sep=1]
\node[root] (r) at (1, 0) {};
\node [n] (v) at (1,1) {};
\node [lab] (vv) at (1,2.3) {$1$};
\draw (r) edge (v);
\draw (v) edge (vv);
\end{tikzpicture}
\caption{A brace tree $T_{\bul 1} \in \cT^{\tw}(1)$} \label{fig:Tbul1}
\end{minipage} 
\begin{minipage}[t]{0.45\linewidth}
\centering 
 \begin{tikzpicture}[scale=0.5]
\tikzstyle{lab}=[circle, draw, minimum size=5, inner sep=1]
\tikzstyle{root}=[circle, draw, fill, minimum size=0, inner sep=1]
\node[root] (r) at (1, 0) {};
\node [lab] (v) at (1,1) {$1$};
\node [n] (vv) at (1,2.3) {};
\draw (r) edge (v);
\draw (v) edge (vv);
\end{tikzpicture}
\caption{A brace tree $T_{1 \bul} \in \cT^{\tw}(1)$} \label{fig:T1bul}
\end{minipage} 
\end{figure}
On figures, neutral vertices of a brace tree in $\cT^{\tw}(n)$ 
are depicted by black circles and 
vertices in $V_{\lab}(T)$ are depicted by white circles  
with the corresponding numbers inscribed.

We have the obvious bijection between brace trees in 
$\cT^{\tw}(n)$ with $r$ neutral vertices and  
linear combinations \eqref{T-averaged}\,. This bijection 
assigns to a brace tree $T'$ with $r$ neutral vertices 
the linear combination \eqref{T-averaged} where $T$ is obtained 
from $T'$ by labeling the neutral vertices by $1,2, \dots, r$ in any 
possible way and shifting the labels for vertices in  $V_{\lab}(T)$ up by $r$\,.  

In virtue of this bijection, the $n$-th space $\Tw\BT(n)$ of $\Tw\BT$ is the 
space of (finite) linear combinations of
brace trees in $\cT^{\tw}(n)$\,. Furthermore,
each brace tree $T' \in \cT^{\tw}(n)$
carries the degree 
\begin{equation}
\label{deg-in-cTtw}
|T| = 2 |V_{\nu}(T)| - |E(T)| + 1\,, 
\end{equation}
where $E(T)$ is the set of all edges of $T$\,.
In other words,  each non-root
edge carries degree $-1$ and each neutral vertex carries degree $2$\,.

Using this description of $\Tw\BT$ it is easy to define
the elementary insertions in terms of brace trees. 
Indeed, let $T \in \cT^{\tw}(n) $, $T' \in \cT^{\tw}(k)$, $1\le i \le k$
and let $v_i$ be the vertex in $V_{\lab}(T')$ with label $i$\,.
If  $v_i$  is a leaf (i.e. $v_i$ does not have 
incoming edges) then the vector $T' \circ_i T \in \Tw\BT(n+k-1)$ is represented by a tree $T''$
which is obtained from $T'$ by erasing  the vertex $v_i$ and gluing the 
tree $T$ via identifying the root edge of $T$ with the edge originating 
at $v_i$\,. After this operation we relabel elements of the set 
$V_{\lab}(T'') = V_{\lab}(T) \sqcup (V_{\lab}(T') \setminus \{v_i\})$ in the 
obvious way. The sign factor in front of $T''$ is obtained by keeping 
track of the reordering of non-root edges of $T'$ and $T$.

Let us now consider the case when $v_i$ has $q \ge 1$ incoming 
edges. Since  the tree $T'$ is planar these $q$ incoming edges 
are totally ordered. So we denote them by $e_1, e_2, \dots, e_q$
keeping in mind that 
\begin{equation}
\label{e1e2dots}
e_1 < e_2 < \dots < e_q\,.
\end{equation}

The desired vector $T' \circ_i T \in \Tw\BT(n+k-1)$ is represented 
by the sum 
\begin{equation}
\label{circ-i-TwBT}
T' \circ_i T = \sum_{\al} (-1)^{f(\al)} T_{\al} 
\end{equation}
where $ T_{\al}$ is obtained from $T'$ and $T$ following these steps: 
\begin{itemize}

\item first, we erase the vertex $v_i \in V_{\lab}(T')$ and 
glue the tree $T$  via identifying the root edge $T$ with 
the edge originating from $v_i$\,,

\item second, we attach the edges $e_1, e_2, \dots, e_q$ to 
vertices in the set $V_{\lab}(T) \sqcup V_{\nu}(T)$   
 
\item finally, we relabel elements of the set 
$V_{\lab}(T'') = V_{\lab}(T) \sqcup (V_{\lab}(T') \setminus \{v_i\})$ in the 
obvious way. 
\end{itemize}
Ways of connecting the edges $e_1, e_2, \dots, e_q$ to 
vertices in the set $V_{\lab}(T) \sqcup V_{\nu}(T)$ should 
obey the obvious analog of Condition \ref{cond:connect-BT}:
\begin{cond} 
\label{cond:connect-Tw}
The restriction of the total order on the 
set $E(T_{\al})$ of $T_{\al}$ to the subset 
$\{e_1, e_2, \dots, e_q\}$ should coincide 
with the order \eqref{e1e2dots}\,.
\end{cond}   
The summation in \eqref{circ-i-TwBT} goes over all ways
$\al$ of connecting the edges $e_1, e_2, \dots, e_q$ to 
vertices in the set  $V_{\lab}(T) \sqcup V_{\nu}(T)$ satisfying 
Condition \ref{cond:connect-Tw}. 
The sign factor $(-1)^{f(\al)}$ is define in the same way as for 
elementary insertions in $\BT$\,. 

\begin{example}
\label{ex:TwBT-ins}
Let $T_{\cc}$ (resp. $T_{\cup}$) be the brace tree in $\cT^{\tw}(2)$ 
depicted on figure \ref{fig:T12} (resp. figure \ref{fig:cup}).
Then the vector $T \circ_1 T_{\cup} \in \Tw\BT(3)$ equals to 
the sum shown on figure \ref{fig:TwBT-ins}.
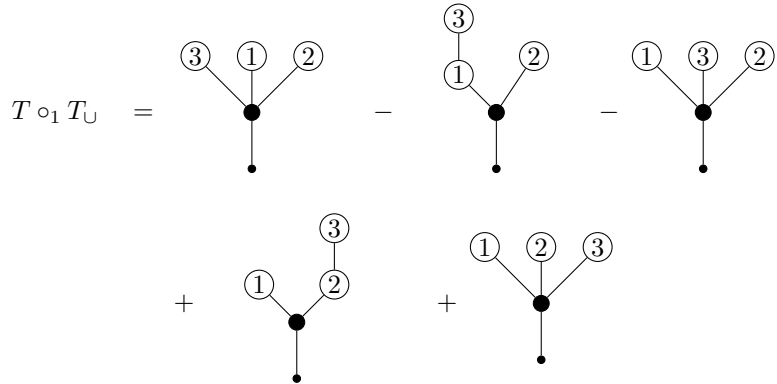
\begin{figure}[htp]
\begin{minipage}[t]{\linewidth}
\centering
\begin{tikzpicture}[scale=0.5]
\tikzstyle{lab}=[circle, draw, minimum size=5, inner sep=1]
\tikzstyle{root}=[circle, draw, fill, minimum size=0, inner sep=1]
\draw (-2,1.5) node[anchor=center] {{$T \circ_1 T_{\cup} \quad =$}};
\node[root] (r) at (2.5, 0) {};
\node[n] (v) at (2.5, 1.5) {};
\node [lab] (v3) at (1,3) {$3$};
\node [lab] (v1) at (2.5,3) {$1$};
\node [lab] (v2) at (4,3) {$2$};
\draw (r) edge (v);
\draw (v) edge (v1);
\draw (v) edge (v2);
\draw (v) edge (v3);
\draw (6,1.5) node[anchor=center] {{$-$}};
\node[root] (r) at (9, 0) {};
\node[n] (v) at (9, 1.5) {};
\node [lab] (v1) at (8,2.5) {$1$};
\node [lab] (v2) at (10,3) {$2$};
\node [lab] (v3) at (8,4) {$3$};
\draw (r) edge (v);
\draw (v) edge (v1);
\draw (v) edge (v2);
\draw (v1) edge (v3);
\draw (12,1.5) node[anchor=center] {{$-$}};
\node[root] (r) at (14.5, 0) {};
\node[n] (v) at (14.5, 1.5) {};
\node [lab] (v1) at (13,3) {$1$};
\node [lab] (v3) at (14.5,3) {$3$};
\node [lab] (v2) at (16,3) {$2$};
\draw (r) edge (v);
\draw (v) edge (v1);
\draw (v) edge (v2);
\draw (v) edge (v3);
\end{tikzpicture}
\end{minipage}
~\\[0.5cm]
\begin{minipage}[t]{\linewidth}
\centering
\begin{tikzpicture}[scale=0.5]
\tikzstyle{lab}=[circle, draw, minimum size=5, inner sep=1]
\tikzstyle{root}=[circle, draw, fill, minimum size=0, inner sep=1]
\draw (-1,1.5) node[anchor=center] {{$+$}};
\node[root] (r) at (2, -0.5) {};
\node[n] (v) at (2, 1) {};
\node [lab] (v1) at (1,2) {$1$};
\node [lab] (v2) at (3,2) {$2$};
\node [lab] (v3) at (3,3.5) {$3$};
\draw (r) edge (v);
\draw (v) edge (v1);
\draw (v) edge (v2);
\draw (v2) edge (v3);
\draw (6,1.5) node[anchor=center] {{$+$}};
\node[root] (r) at (8.5, 0) {};
\node[n] (v) at (8.5, 1.5) {};
\node [lab] (v1) at (7,3) {$1$};
\node [lab] (v2) at (8.5,3) {$2$};
\node [lab] (v3) at (10,3) {$3$};
\draw (r) edge (v);
\draw (v) edge (v1);
\draw (v) edge (v2);
\draw (v) edge (v3);
\end{tikzpicture}
\end{minipage}
\caption{The elementary insertion $T \circ_1 T_{\cup}$} \label{fig:TwBT-ins}
\end{figure}
\end{example}

\begin{remark}
\label{rem:cL-BT}
Let us recall that $\BT(0) = \bfzero$. Hence, due to Remark \ref{rem:Tw-cO-0}, 
the graded vector space of the dg Lie algebra
\begin{equation}
\label{cL-BT}
\cL_{\BT} = \Conv(\La^2 \coCom, \BT) 
\end{equation}
is canonically identified with $\bs^{-2} \Tw\BT(0)$
Thus,  we can use brace trees in $\cT^{\tw}(0)$ to 
represent vectors in the dg Lie algebra \eqref{cL-BT}. 

However, the degree formula should be adjusted as follows: 
a brace tree $T \in \cT^{\tw}(0)$ representing a vector in
$\cL_{\BT}$ carries the degree
\begin{equation}
\label{deg-T-cLBT}
|T| = 2 |V_{\nu}(T)|  - |E(T)| -1\,. 
\end{equation}
For example, the vector $T_{\bb}$ depicted on figure \ref{fig:bul-bul}
\begin{figure}
\centering
\begin{tikzpicture}[scale=0.5]
\tikzstyle{lab}=[circle, draw, minimum size=5, inner sep=1]
\tikzstyle{root}=[circle, draw, fill, minimum size=0, inner sep=1]
\node[root] (r) at (0, 0) {};
\node[n] (v) at (0, 1.5) {};
\node [n] (vv) at (0,3) {};
\draw (r) edge (v);
\draw (v) edge (vv);
\end{tikzpicture}
\caption{The brace tree $T_{\bb}\in \cT^{\tw}(0)$} \label{fig:bul-bul}
\end{figure}  
is the Maurer-Cartan element in $\cL_{\BT}$ corresponding to the 
map  \eqref{LaLie-BT}. This vector carries degree $1$\,.
\end{remark}

\subsection{The differential on \texorpdfstring{$\Tw\BT$}{Tw BT}}
Following the general procedure, the differential $\pa(T)$ of 
a brace tree $T \in \Tw\BT(n)$ is 
given by the formula
\begin{equation}
\label{pa-TwBT}
\pa (T) = T_{\bb} \cdot  T +  \ka(T_{\bb})  \circ_1  T 
-(-1)^{|T|} \sum_{i=1}^n T  \circ_i  \ka(T_{\bb})\,. 
\end{equation}
Here $T_{\bb} \cdot  \,$ denotes the auxiliary action  
\eqref{eq:cL-action1} of the Lie algebra $\cL_{\BT}$ on 
the operad $\Tw\BT$ and 
\begin{equation}
\label{ka-MC}
\ka(T_{\bb})  = T_{\bul1} + T_{1\bul}\,, 
\end{equation}
where $T_{\bul1}$ and $T_{1\bul}$ are the brace trees 
depicted on figures \ref{fig:Tbul1} and \ref{fig:T1bul}, respectively.
Unfolding the definition of the  auxiliary action \eqref{eq:cL-action1} 
and using \eqref{ka-MC} we get 
\begin{equation}
\label{pa-TwBT1}
\pa (T) =  
-(-1)^{|T|} \sum_{v \in V_{\nu}(T)}  
T_{v} \circ_{n+1} T_{\bb} + T_{1\bul} \circ_1 T + 
 T_{\bul1} \circ_1 T
\end{equation}
$$
-(-1)^{|T|} \sum_{i=1}^n T \circ_i   T_{1\bul} 
-(-1)^{|T|} \sum_{i=1}^n T \circ_i   T_{\bul1}\,. 
$$
where $T_v$ is a brace tree in $\cT^{\tw}(n+1)$ which is
obtained from $T$ by replacing the neutral vertex $v$ by 
a ``white'' vertex with label $n+1$\,. 
In other words, $T_{\bb} \cdot  T$ is  the sum 
over insertions of the brace tree $T_{\bb}$ (see fig. \ref{fig:bul-bul})
into neutral vertices of $T$ with appropriate signs.  
The expression $ T_{1\bul} \circ_1 T$ is the sum of brace 
trees which are obtained from $T$ by attaching the single edge with 
a neutral vertex on one end to all non-root vertices of $T$\,.
The brace tree $T_{\bul 1} \circ_1 T$ is obtained from the tree $T$
by dividing the root edge into two parts and inserting a neutral vertex 
in the middle. The expression $ T \circ_i   T_{1\bul} $ (resp. 
the expression $T \circ_i   T_{\bul1}$ ) is obtained from $T$ by 
inserting $T_{1\bul} $ (resp. $  T_{\bul1} $) into the vertex of $T$ with 
label $i$\,.

\begin{example}
\label{ex:pa-TwBT}
Let $T_{\cc}$, $T_{\cup}$, $T_{\cup^{opp}}$, $T_{\bul 1}$
be the brace trees depicted on figures \ref{fig:T12},
\ref{fig:cup}, \ref{fig:cup-opp}, and \ref{fig:Tbul1}.
Direct computations show that $T_{\cup}$ and $T_{\cup^{opp}}$ 
are $\pa$-closed and $T_{\cc}$ is not $\pa$-closed. Instead, we have
\begin{equation}
\label{pa-T-cc}
\pa(T_{\cc}) = T_{\cup^{opp}} - T_{\cup}\,.
\end{equation}
In other words, $T_{\cup}$ is cohomologous to $T_{\cup^{opp}}$\,.

For the brace tree  $T_{\bul 1}$ we have 
\begin{equation}
\label{pa-Tbul1}
\pa (T_{\bul 1}) = T_{\bul\bul 1}\,.
\end{equation}
where $T_{\bul\bul 1}$ is the brace tree depicted 
on figure \ref{fig:Tbulbul1}.
\begin{figure}
\centering
\begin{tikzpicture}[scale=0.5]
\tikzstyle{lab}=[circle, draw, minimum size=5, inner sep=1]
\tikzstyle{root}=[circle, draw, fill, minimum size=0, inner sep=1]
\node[root] (r) at (0, 0) {};
\node[n] (v) at (0, 1.3) {};
\node [n] (vv) at (0, 2.6) {};
\node [lab] (v1) at (0, 3.9) {$1$};
\draw (r) edge (v);
\draw (v) edge (vv);
\draw (vv) edge (v1);
\end{tikzpicture}
\caption{The brace tree $T_{\bul\bul 1}\in \cT^{\tw}(1)$} \label{fig:Tbulbul1}
\end{figure}  
\end{example}

The operad $\Tw \BT$ acts naturally on the Hochschild cochain complex 
$\Cbu(A)$ of any (flat) $A_{\infty}$-algebra $A$\,. 
This action is described in Appendix \ref{app:action}.

\section{The braces operad \texorpdfstring{$\Br$}{Br} \`a la Kontsevich-Soibelman}
\label{sec:Br} 
In this section we define a suboperad $\Br$ of $\Tw\BT$\,.
The operad $\Br$ coincides with what  M. Kontsevich and Y. Soibelman
called the ``minimal operad'' in \cite{K-Soi}. By slightly abusing notation
we will call $\Br$ the operad of braces and we will refer to algebras over 
$\Br$ as brace algebras.  At the end of this section we will 
show that the embedding $\Br\hookrightarrow \Tw\BT$ is a
quasi-isomorphism of operads. 

We will say that a brace tree $T \in \cT^{\tw}(n)$ is {\it admissible} if $T$ satisfies the 
following condition: 
\begin{cond}
\label{cond:braces-trees}
Every neutral vertex of $T$ has at least $2$ incoming edges.
\end{cond}
Thus figures \ref{fig:T12}, \ref{fig:cup}, and \ref{fig:cup-opp} 
show examples of admissible brace trees while
figures \ref{fig:Tbul1}, \ref{fig:T1bul}, \ref{fig:bul-bul} 
show examples of inadmissible brace trees.

We define $\Br(n)$ as the subspace of $\Tw\BT(n)$ which 
is spanned by admissible brace trees. 
It is easy to see that, in $\cT^{\tw}(0)$ there are no 
admissible brace trees. Hence, 
\begin{equation}
\label{Br-0}
\Br(0) = \bfzero.
\end{equation}
Furthermore, in $\cT^{\tw}(1)$, there is 
exactly one admissible brace tree. This brace tree is depicted on 
figure \ref{fig:id-BT}. Thus 
\begin{equation}
\label{Br-1}
\Br(1) = \bbK\,.
\end{equation}

We claim that 
\begin{prop}
\label{prop:Br}
The sub-collection $\big\{ \Br(n)\big\}_{n \ge 0}$ is a suboperad 
of $\Tw\BT$\,.
\end{prop}
\begin{proof}
It is obvious that for admissible brace trees $T'\in \cT^{\tw}(k)$ 
and $T \in \cT^{\tw}(n)$ 
$$
T' \circ_i T\,, \qquad 1\le i \le k
$$
are linear combinations of admissible brace trees. 
Thus it remains to prove that for every brace tree 
$\pa T$ is a linear combination of admissible brace trees. 
Concretely, we have to show that in $\pa T$ there occur no 
brace trees with 1- or 2-valent neutral vertices. 

Let $e(T)$ be the number of all edges of $T$.  
The vertex splitting operations produce
$2 e(T)$ terms with a 2-valent neutral vertex, 2 for each of the $e(T)$ edges. 
Concretely,
\[
\begin{tikzpicture}[scale=.7]
\node[ngr] (v1) at (0,0.3) {};
\draw (v1) -- +(135:1) (v1)-- +(-90:1) (v1) -- ++(45:1) ;
\node[ngr] (v2) at (0,1.7) {};
\draw (v2) -- +(135:1) (v2)-- +(90:1) (v2) -- ++(45:1) ;
\draw (v2) edge (v1);
\draw [->] (1.5,1) -- (2.5,1);
\begin{scope}[xshift=4cm]
\node[ngr] (v1) at (0,0) {};
\draw (v1) -- +(135:1) (v1)-- +(-90:1) (v1) -- ++(45:1) ;
\node[ngr] (v2) at (0,2) {};
\node[n] (v3) at (0,.8) {};
\draw (v2) -- +(135:1) (v2)-- +(90:1) (v2) -- ++(45:1) ;
\draw (v2) edge (v3) (v3) edge (v1);
\end{scope}
\begin{scope}[xshift=7cm]
\node[ngr] (v1) at (0,0) {};
\draw (v1) -- +(135:1) (v1)-- +(-90:1) (v1) -- ++(45:1);
\node[ngr] (v2) at (0,2) {};
\node[n] (v3) at (0,1.2) {};
\draw (v2) -- +(135:1) (v2)-- +(90:1) (v2) -- ++(45:1) ;
\draw (v2) edge (v3) (v3) edge (v1);
\end{scope}
\end{tikzpicture}
\]
Here the left term comes from splitting of the lower gray vertex and the right term comes from splitting of 
the upper gray vertex. 
The upper gray vertex can be either labeled or neutral, in particular, it may have 
valency $1$. The lower gray vertex may be either a labeled vertex, or a neutral vertex, or the root 
vertex. (Of course, in the latter case the lower vertex has valency $1$.) 
The remainder of the brace tree is omitted. Comparing the signs factors, we see that
these two terms cancel each other, so the differential does not produce valence 2 neutral vertices.

Next consider valence 1 neutral vertices. 
Again, the corresponding brace trees come in pairs:
\[
\begin{tikzpicture}[scale=.7]
\node[ngr] (v1) at (0,0) {};
\draw (v1) -- +(135:1.5) (v1)-- +(-90:1) (v1) -- ++(45:1.5);
\draw [->] (1.5,0) -- (2.5,0);
\begin{scope}[xshift=4cm]
\node[ngr] (v1) at (0,0) {};
\draw (v1) -- +(135:1.5) (v1)-- +(-90:1) (v1) -- ++(45:1.5) ;
\node[n] (v3) at (0,.8) {};
\draw (v3) edge (v1);
\end{scope}
\begin{scope}[xshift=7cm]
\node[ngr] (v1) at (0,0) {};
\draw (v1) -- +(135:1.5) (v1)-- +(-90:1) (v1) -- ++(45:1.5) ;
\node[n] (v3) at (0,1.8) {};
\draw (v3) edge (v1);
\end{scope}
\end{tikzpicture}
\]
Here the left term comes from splitting of the shown gray (i.e., white or black) vertex 
and the right term comes from 
the linear combination $\ka(T_{\bb})  \circ_1  T $ in \eqref{pa-TwBT}.
Again, checking the signs, we see that these two terms cancel each other and no valence 1 neutral vertex is produced.

\end{proof}

\subsection{A simpler description of the differential on \texorpdfstring{$\Br$}{Br}}
Let $T$ be an admissible brace tree in $\cT^{\tw}(n)$, 
$v$ be a vertex in $V_{\lab}(T)$ carrying label $i$\,, 
and $\ka(T_{\bb})$ be the sum $T_{\bul1} + T_{1\bul}$ of inadmissible brace 
trees depicted  on figures \ref{fig:Tbul1} and \ref{fig:T1bul}\,.
We denote by $\pa_{v} (T)$  a vector in $ \Br(n)$
which is obtained from the sum
$$
T \circ_i \ka(T_{\bb})
$$
by omitting all non-admissible brace trees. 
Let $v$ be a neutral vertex of $T$\,. 
To define $\pa_v(T)$ we  change the color of  vertex $v$  
to white and assign to $v$ label $n+1$\,.
We denote this new brace tree by $T_v$
and obtain $\pa_v(T)$ from the sum 
$$
T_v \circ_{n+1} T_{\bb} 
$$
via omitting all non-admissible brace trees.

Finally, we declare that the differential $\pa$ on $\Br$
is the vector $\pa(T)\in \Br(n)$ is represented by 
the sum 
\begin{equation}
\label{pa-Br}
\pa(T) = \sum_{v \in V_{\lab}(T)\sqcup  V_{\nu}(T) } \pm \pa_v (T)\,,
\end{equation}
where the sign factors are determined in the same way as in formula 
\eqref{pa-TwBT}.

\subsection{The inclusion \texorpdfstring{$\Br \hookrightarrow \Tw\BT$}{of Br into Tw BT} is a quasi-isomorphism}  
\label{sec:Br-TwBT}
The goal of this section is to prove the following theorem 
\begin{thm}
\label{thm:Br-TwBT}
The operad map 
$$
\Br \hookrightarrow \Tw\BT
$$
is a quasi-isomorphism.
\end{thm}
\begin{proof}

Let $T$ be a brace tree in $\cT^{\tw}(n)$ and let $\nu_{\le 2}(T)$ be the 
total number of neutral vertices of valence $\le 2$\,. For example, a brace 
tree $T$ is admissible if and only if
$$
\nu_{\le 2}(T) = 0\,.
$$

We observe that for every brace tree  $T \in \cT^{\tw}(n)$ the 
linear combination 
$$
\pa (T)  
$$
involves brace trees $T_i$ with $\nu_{\le 2}(T_i)  = \nu_{\le 2}(T)$ or 
$\nu_{\le 2}(T_i) = \nu_{\le 2}(T) + 1$\,. 

This observation allows us to introduce on $\Tw \BT$ the 
following increasing filtration 
\begin{equation}
\label{filtr-TwBT}
\dots \subset \cF^{m-1} \Tw\BT
\subset \cF^m \Tw\BT \subset \cF^{m+1} \Tw\BT   
\subset \dots\,,
\end{equation}
where $\cF^m \Tw\BT(n)$ consists of linear combinations of brace trees 
$T \in \cT^{\tw}(n)$ satisfying the condition:
\begin{equation}
\label{filtr-TwBT-dfn}
\nu_{\le 2}(T) - \deg(T) \le m\,.
\end{equation}

The restriction of \eqref{filtr-TwBT} on $\Br$ 
gives us the ``silly'' filtration
\begin{equation}
\label{silly-filtr-Br}
\cF^m \Br(n)^k = 
\begin{cases}
 \Br(n)^k  \qquad {\rm if} ~~  k \ge  -m\,, \\
   \bfzero \qquad {\rm otherwise}
\end{cases}
\end{equation}
with the associated graded complex carrying the zero 
differential. 

To describe the associated graded complex for $\Tw\BT$, we 
introduce the notion of a {\it core} for brace trees in $\cT^{\tw}(n)$\,.
Namely, for a brace tree $\Gamma  \in \cT^{\tw}(n)$, its core 
$\Gamma'$ obtained by deleting all neutral bivalent vertices and 
connecting their incident edges. Here is an example:
\begin{center}
\begin{tikzpicture}[scale=.7,
ext/.style={draw, circle, minimum size=5, inner sep=1, }, 
int/.style={draw, circle, fill, inner sep=1,  minimum size=3},
arr/.style={-triangle 60},]

\begin{scope}[]
\node [ext] (v3) at (-3.5,-2) {};
\node [ext] (v13) at (-5,-0.5) {};
\node [ext] (v11) at (-6,0.5) {};
\node [int] (v2) at (-3.5,-2.5) {};
\node [int] (v1) at (-3.5,-3) {};
\node [int] (v15) at (-4,-1.5) {};
\node [int] (v14) at (-4.5,-1) {};
\node [int] (v4) at (-3.5,-1.5) {};
\node [int] (v5) at (-3.5,-1) {};
\node [int] (v6) at (-3.5,-0.5) {};
\node [int] (v9) at (-4,0) {};
\node [int] (v10) at (-4.5,0.5) {};
\node [int] (v7) at (-3,0) {};
\node [int] (v8) at (-2.5,0.5) {};
\node [int] (v12) at (-5.5,0) {};
\draw (v1) edge (v2);
\draw (v2) edge (v3);
\draw (v3) edge (v4);
\draw (v4) edge (v5);
\draw (v5) edge (v6);
\draw (v6) edge (v7);
\draw (v7) edge (v8);
\draw (v6) edge (v9);
\draw (v9) edge (v10);
\draw (v11) edge (v12);
\draw (v12) edge (v13);
\draw (v13) edge (v14);
\draw (v14) edge (v15);
\draw (v15) edge (v3);
\draw(v1) -- (-3.5,-3.5);
\end{scope}
\begin{scope}[shift={(6.5,0)}]
\node [ext] (v3) at (-3.5,-2) {};
\node [ext] (v13) at (-5,-0.5) {};
\node [ext] (v11) at (-6,0.5) {};
\node [int] (v6) at (-3.5,-0.5) {};
\node [int] (v10) at (-4.5,0.5) {};
\node [int] (v8) at (-2.5,0.5) {};
\draw(v3) -- (-3.5,-3.5);

\end{scope}
\node at (-1,-1.5) {\huge$\Rightarrow$};
\draw (v3) edge (v13);
\draw (v13) edge (v11);
\draw (v3) edge (v6);
\draw (v6) edge (v10);
\draw (v6) edge (v8);
\draw(v3);
\node at (-4,1.25) {a brace tree};
\node at (2.5,1.25) {its \emph{core}};
\end{tikzpicture}
\end{center}

It is easy to see that the associated graded complex
$$
\Gr \Tw \BT(n)
$$
splits into the direct sum over all possible cores with 
$n$ labeled vertices
\begin{equation}
\label{direct-sum}
\bigoplus_{\text{cores }\Gamma'} \Tw\BT_{\Gamma'}\,,
\end{equation}
where $ \Tw\BT_{\Gamma'}$ is spanned by brace trees with 
the given core $\Gamma'$\,. 

Furthermore, for each core $\Gamma'$ the subcomplex $\Tw\BT_{\Gamma'}$ is 
isomorphic to the tensor product  
\[
\Tw\BT_{\Gamma'} \cong \bigotimes_{\textrm{edges }e \textrm{ of } \Gamma'} V_e\,,
\]
where
\[
V_e = 
\begin{cases}
\bbK \stackrel{\mathit{id}}{\to} \bbK \stackrel{0}{\to} \bbK \stackrel{\mathit{id}}{\to} \bbK \cdots & \quad \text{if $e$ connects to a univalent neutral vertex,} \\
\bbK \stackrel{0}{\to} \bbK \stackrel{\mathit{id}}{\to} \bbK \stackrel{0}{\to} \bbK \cdots & \quad \text{otherwise.} \\
\end{cases}
\]
The various copies of $\bbK$ correspond to strings of bivalent neutral vertices of various lengths:
\[
\begin{tikzpicture}[
xst/.style={draw, cross out, minimum size=5, }, 
int/.style={draw, circle, fill, inner sep=1, minimum size=5},
ext/.style={draw, circle, fill, gray, inner sep=1, minimum size=5},
arr/.style={-triangle 60},]

\node[ext] (v1) at (-3,-3) {};
\draw (v1) -- +(45:.5) (v1) -- +(135:.5) (v1) -- +(-90:.5) ;
\node [int] (v2) at (-3,-2) {};
\node [int] (v4) at (-2,-2) {};
\node [int] (v5) at (-2,-1) {};
\node [int] (v7) at (-1,-2) {};
\node [int] (v8) at (-1,-1) {};
\node [int] (v9) at (-1,0) {};
\node [int] (v13) at (2.5,-2) {};
\node [int] (v16) at (3.5,-1) {};
\node [int] (v17) at (3.5,-2) {};
\node [ext] (v3) at (-2,-3) {};
\node [ext] (v6) at (-1,-3) {};
\node [ext] (v12) at (2.5,-1) {};
\node [ext] (v15) at (3.5,0) {};
\node [ext] (v10) at (1.5,-2) {};
\node [ext] (v11) at (1.5,-3) {};
\node [ext] (v14) at (2.5,-3) {};
\node [ext] (v18) at (3.5,-3) {};
\node at (0,-2) {$\cdots\quad$  or};
\node at (4.25,-2) {$\cdots$ };
\draw (v1) edge (v2);
\draw (v3) edge (v4);
\draw (v4) edge (v5);
\draw (v6) edge (v7);
\draw (v7) edge (v8);
\draw (v8) edge (v9);
\draw (v10) edge (v11);
\draw (v12) edge (v13);
\draw (v13) edge (v14);
\draw (v15) edge (v16);
\draw (v16) edge (v17);
\draw (v17) edge (v18);

\draw (v1) -- +(45:.5) (v1) -- +(135:.5) (v1) -- +(-90:.5) ;
\draw (v3) -- +(45:.5) (v3) -- +(135:.5) (v3) -- +(-90:.5) ;
\draw (v6) -- +(45:.5) (v6) -- +(135:.5) (v6) -- +(-90:.5) ;
\draw (v11) -- +(45:.5) (v11) -- +(135:.5) (v11) -- +(-90:.5) ;
\draw (v14) -- +(45:.5) (v14) -- +(135:.5) (v14) -- +(-90:.5) ;
\draw (v18) -- +(45:.5) (v18) -- +(135:.5) (v18) -- +(-90:.5) ;
\draw (v10) -- +(45:.5) (v10) -- +(135:.5) (v10) -- +(90:.5) ;
\draw (v12) -- +(45:.5) (v12) -- +(135:.5) (v12) -- +(90:.5) ;
\draw (v15) -- +(45:.5) (v15) -- +(135:.5) (v15) -- +(90:.5) ;
\end{tikzpicture}
\]
Here, a gray vertex is either a labeled vertex or a neutral vertex with $\ge 2$ incoming 
edges. Any lower gray vertex may also be the root vertex.
The rest of the brace tree is omitted, and just indicated by some stubs at the gray vertices.
Clearly the cohomology of $V_e$ is trivial if $e$ connects to a univalent neutral vertex and $\bbK$ otherwise. 
The representative in the latter case is a single edge (i.e., a string of zero neutral vertices). 

Therefore the cochain complex $\Tw\BT_{\Gamma'}$ is acyclic if $\Gamma'$ contains at least one 
univalent neutral vertex and 
\begin{equation}
\label{H-for-TwBT-Gpr}
H^{\bul} \big( \Tw\BT_{\Gamma'} \big) \cong \bbK \L \{ \Gamma' \} \R
\end{equation}
if the core $\Gamma'$  does not have univalent neutral vertices.
Hence the embedding 
$$
\Br(n) \hookrightarrow \Tw \BT(n)
$$ 
induces a quasi-isomorphism on the level of associated graded complexes. 

Combining this observation with the fact that the filtrations \eqref{filtr-TwBT} 
and \eqref{silly-filtr-Br} are cocomplete and locally bounded, we deduce 
Theorem \ref{thm:Br-TwBT} from \cite[Lemma A.3]{notes}. 
\end{proof}

\subsection{The dg operad $\Br$ is a homotopy fixed point of $\Tw$}
\label{sec:Br-homotopyfixed}

Various solutions of the Deligne conjecture on Hochschild cochain complex 
\cite{Bat-Markl}, \cite{BF}, \cite{Swiss}, \cite{M-Smith}, \cite{M-Smith2}, \cite{K-Soi}, \cite{Dima-dg}, 
\cite{Vor} imply that the dg operad $\Br$ is quasi-isomorphic to the dg operad 
$$
C_{-\bul}(E_2, \bbK)
$$
of singular chains for the little disc operad $E_2$\,.

Combining this statement with the formality \cite{K-mot}, \cite{Dima-disc}
for the dg operad $C_{-\bul}(E_2, \bbK)$, we conclude that the dg operad 
$\Br$ (as well as $\Tw\BT$) is quasi-isomorphic to the operad 
$\Ger$. 

Thus Theorem \ref{thm:Twpreservesqiso} and Corollary \ref{cor:LaLie-Ger}
imply that 
\begin{corollary}
\label{cor:Br-homot-fixed}
The dg operads $\Br$ and $\Tw\BT$ are homotopy fixed points for $\Tw$\,. ~~~$\Box$
\end{corollary}

\subsection{$\Br$ is a $\Tw$-coalgebra}

Consider the composition
\[
\Br\to \Tw\BT \to \Tw\Tw\BT.
\]
We claim that the image actually lands in $\Tw\Br \subset \Tw\Tw\BT$. Indeed, neither the first nor the second map creates new neutral vertices, and hence no neutral vertices of valence 1 or 2 in particular. Hence we have a map 
\[
c : \Br\to \Tw\Br.
\]
It follows almost immediately from the $\Tw$ colagebra axioms for $\Tw\BT$ that this map endows $\Br$ with the structure of a $\Tw$-coalgebra. Summarizing, we arrive at the following result.
\begin{lemma}
$\Br$ is a sub $\Tw$-coalgebra of $\Tw\BT$.  $~~~\Box$
\end{lemma}

\section{Every solution of Deligne's conjecture is homotopic to a one that is compatible with twisting}
\label{sec:theproof}

We are now ready to prove Theorem \ref{thm:main}, which roughly states that  
every solution of Deligne's conjecture is homotopic to a one that is compatible 
with twisting.

We will deduce this result from the following more general statement:
\begin{thm}
\label{thm:general}
Let $\La \Lie_\infty\to \cO$, $\La \Lie_\infty\to \cO'$ be two objects of the 
under-category $\La \Lie_\infty\downarrow \Operads$ and let $\cO \stackrel{F}{\to} \Tw\cO'$ be 
a morphism of dg operads compatible with the maps from $\La \Lie_\infty$.
If $(\cO, c)$ is a $\Tw$-coalgebra and the dg operads
$\cO$, $\Tw\cO'$ are homotopy fixed points for $\Tw$, then the morphism 
\begin{equation}
\label{F-pr}
F' := (\Tw (\eta_{\cO'}\circ F)) \circ c \colon \cO\to \Tw\cO'
\end{equation}
is homotopy equivalent to $F$. Furthermore, $F'$ is a morphism of $\Tw$-coalgebras.
\end{thm}
This theorem can be restated as follows: if $\cO$ and $\Tw \cO'$ are 
homotopy fixed points for $\Tw$ (with $\cO$ being a $\Tw$-coalgebra) 
then each homotopy class of maps of dg operads $\cO \to \Tw \cO'$ has 
at least one representative which respects the $\Tw$-coalgebra structures
on $\cO$ and $\Tw \cO'$. Furthermore, such a representative may be written down explicitly.
\begin{remark}
\label{rem:homotopy-cat}
The homotopy category of dg operads can be constructed using 
the standard technics of closed model categories. For more details 
we refer the reader to  \cite{Hinich-hom},  \cite{Hinich-err}, 
and \cite[Part II, Appendix A.4]{MV-nado}. In the proof below we use merely
the existence of the homotopy category of dg operads. 
\end{remark}
\begin{proof} 
Abbreviate $f:=\eta_{\cO'}\circ F$ and consider the following commutative diagram:
\begin{equation}
\label{f-dfn}
\begin{tikzpicture}
\matrix (m) [matrix of math nodes, row sep=2em, column sep =3em]
{ \cO & \Tw\cO'   \\
      & \cO' \\ };
\path[->,font=\scriptsize]
(m-1-1) edge node[below] {$f$}(m-2-2) 
        edge node[auto] {$F$} (m-1-2)
(m-1-2) edge node[auto] {$\eta_{\cO'}$} (m-2-2);
\end{tikzpicture}
\end{equation}
Applying the functor $\Tw$ to \eqref{f-dfn}
we obtain the commutative diagram:
\begin{equation}
\label{Tw-f-dfn}
\begin{tikzpicture}
\matrix (m) [matrix of math nodes, row sep=2em, column sep =3em]
{ \Tw\cO & \Tw\Tw\cO'   \\
         & \Tw\cO' \\ };
\path[->,font=\scriptsize]
(m-1-1) edge node[below] {$\Tw f$}(m-2-2) 
        edge node[auto] {$\Tw F$} (m-1-2)
(m-1-2) edge node[auto] {$\Tw \eta_{\cO'}$} (m-2-2);
\end{tikzpicture}
\end{equation}
Next consider the following diagram:
\begin{equation}
\label{master-diag}
\begin{tikzpicture}[descr/.style={fill=white,inner sep=2.5pt},
ccell/.style={}]
\matrix (m) [matrix of math nodes, row sep=5em, column sep =5em]
{ \Tw\cO & \Tw\Tw\cO'   \\
  \cO    & \Tw\cO' \\ };
\path[->,font=\scriptsize]
(m-1-1) edge node[descr] {$\Tw f$}(m-2-2) 
        edge node[auto] {$\Tw F$} (m-1-2)
        edge[bend right] node[left] {$\eta_{\cO}$} (m-2-1)
(m-1-2) edge node[below, sloped] {$\Tw \eta_{\cO'}$} (m-2-2)
        edge[bend left] node[right] {$\eta_{\Tw\cO'}$} (m-2-2)
(m-2-1) edge node[right] {$c$} (m-1-1)
        edge node[auto] {$F$} (m-2-2);
\draw (m-2-1) +(.8,.8) node[ccell] {?};
\end{tikzpicture}
\end{equation}
Our goal is to show that the triangle marked with ``?'' is homotopy commutative.
We will do this by showing that all the other cells (including the big cell around the diagram) homotopy commute. We saw above that the upper right triangle commutes. Furthermore the big outer square commutes because $\eta$ is a natural transformation. The small cell on the left commutes since $\eta_\cO\circ c=\mathit{id}_\cO$ by the definition of a $\Tw$ coalgebra. Finally consider the small cell on the right. Note that, by the defining property of the counit, the following diagram commutes:
\[
\begin{tikzpicture}
\matrix (m) [matrix of math nodes, row sep=2em, column sep =3em]
{ \Tw\Tw\cO' & \Tw\cO'   \\
  \Tw\cO' & \Tw\Tw\cO' \\ };
\path[->,font=\scriptsize]
(m-1-1) edge node[auto] {$\Tw \eta_{\cO'}$}(m-1-2) 
(m-2-2) edge node[right] {$\eta_{\Tw\cO'}$} (m-1-2)
(m-2-1) edge node[auto] {$\mD_{\cO'}$} (m-1-1)
        edge node[auto] {$\mD_{\cO'}$} (m-2-2)
;
\end{tikzpicture}
\]
All arrows are quasi-isomorphisms since $\Tw\cO'$ is a homotopy fixed point of $\Tw$. Hence the morphisms $\Tw \eta_{\cO'}$ and $\eta_{\Tw\cO'}$ are homotopic.
This shows that the triangle ? in the previous diagram is homotopy commutative. 

It remains to prove that $F'$ \eqref{F-pr} is a morphism of $\Tw$-coalgebras. 

For this purpose we consider the following diagram
\begin{equation}
\label{diag-forTw-morph}
\begin{tikzpicture}[descr/.style={fill=white,inner sep=2.5pt},
ccell/.style={}, baseline=(current  bounding  box.center)]
\matrix (m) [matrix of math nodes, row sep=4em, column sep =4em]
{ \Tw\cO & \Tw\Tw\cO & \Tw\Tw \cO'   \\
  \cO    & \Tw\cO &  \Tw \cO' \\ };
\path[->,font=\scriptsize]
(m-2-1) edge node[auto] {$c$}(m-1-1)  edge node[auto] {$c$}(m-2-2)
(m-1-1)  edge node[auto] {$\Tw(c)$}(m-1-2)
(m-2-2)  edge node[auto] {$\mD_{\cO}$}(m-1-2)
 edge node[auto] {$\Tw(f)$}(m-2-3)
(m-1-2)  edge node[auto] {$\Tw\Tw(f)$} (m-1-3)
(m-2-3) edge node[auto] {$\mD_{\cO'}$}(m-1-3); 
\end{tikzpicture}
\end{equation}

The left square of diagram \eqref{diag-forTw-morph} 
commutes since $c: \cO \to \Tw\cO$ is the comultiplication map 
for the $\Tw$-coalgebra $\cO$. The right square of  diagram \eqref{diag-forTw-morph} 
commutes since $\mD$ is a natural transformation from $\Tw$ to $\Tw \Tw$\,. 
Thus the ambient rectangle forms a commutative diagram. 

On the other hand, the composition $\Tw (f) \circ c$ of the lower horizontal arrows 
is $F'$ and the composition $\Tw\Tw(f) \circ \Tw(c)$ of upper horizontal arrows 
is $\Tw(F')$\,. Hence $F'$ is indeed a morphism of $\Tw$-coalgebras. 

Theorem \ref{thm:general} is proved. 
\end{proof}

In order to prove Theorem \ref{thm:main} we need a technical 
property of operad maps from $\Ger_{\infty}$ to $\BT$\,. 

So let $f : \Ger_{\infty} \to \BT$ be a map of dg operads and let
\begin{equation}
\label{MC-eta-F}
\al \in \Conv(\Ger^{\vee}, \BT) \cong  \prod_{n \ge 2} 
\Big( \BT(n) \otimes \La^{-2}\Ger(n) \Big)^{S_n}
\end{equation}
be the Maurer-Cartan element which corresponds to $f$\,.

Using the fact that every vector in $\BT(n)$ carries degree $1-n$, we deduce 
that $\al$ has the form
\begin{equation}
\label{al-form}
\al = \sum_{n\ge 2,\, i} ~ \sum_{\si \in S_n}  \si \Big( X_{n,i} \otimes v^1_{n,i} v^2_{n,i} \Big)\,,   
\end{equation}
where $X_{n,i} \in \BT(n)$ and 
$$
v^1_{n,i} v^2_{n,i} \in \bs^2 S^2(\La^{-1}\Lie(b_1, \dots, b_n)) \cap \La^{-2}\Ger(n)\,. 
$$

We claim that 
\begin{prop}
\label{prop:Ger-infty-BT}
If  $f : \Ger_{\infty} \to \BT$ is an operad map satisfying the condition 
\begin{equation}
\label{f-Tcc}
f(\bs \, 1^{\mc}_2)  = T_{\cc} + \si_{12} (T_{\cc})\,,
\end{equation}
with $T_{\cc}$ being the brace tree shown on figure \ref{fig:T12}, then 
the Maurer-Cartan element $\al$ does not involve monomials of the 
form
$$
T \otimes b_1 v(b_2, \dots b_n)\,, \qquad n \ge 3\,,
$$
where $v$ is a $\La^{-1}\Lie$-monomial and $T$ is a brace tree
with vertex\phantom{a} \tikz{\node [circle, draw, minimum size=5, inner sep=1] (v) at (0,0) {$1$};} 
\phantom{a}having valency $1$ or $2$. 
\end{prop}
We will postpone the technical proof of this proposition to Section \ref{sec:Ger-infty-BT} 
and complete the proof of Theorem \ref{thm:main}.

According to Section \ref{sec:Ginf-Twcoalg}, $\Ger_{\infty}$ is a $\Tw$-coalgebra.
Furthermore, by Corollaries \ref{cor:LaLie-Ger-inf} and \ref{cor:Br-homot-fixed}
the dg operads $\Ger_{\infty}$ and $\Tw\BT$  are homotopy 
fixed points for $\Tw$. Thus, given an operad quasi-isomorphism
$F:  \Ger_{\infty} \to \Br$ we can apply Theorem \ref{thm:general} 
to $\cO=\Ger_\infty$, $\cO'=\BT$ and the composition 
$$
\wt{F} : \Ger_{\infty} \stackrel{F}{~\longrightarrow~} \Br  ~\hookrightarrow~ \Tw\BT\,.
$$

Since $\wt{F}$ is a quasi-isomorphism from $\Ger_{\infty}$ to $\Tw\BT$, 
the composition 
$$
f  = \eta_{\BT} \circ \wt{F} : \Ger_{\infty} \to \BT
$$
satisfies condition \eqref{f-Tcc}. Hence, applying Proposition \ref{prop:Ger-infty-BT}, 
we conclude that for every basis vector $v \in \La^{-1}\Lie(n-1)$ the 
linear combination
$$
 \eta_{\BT} \circ \wt{F}\big(\,\bs \,( b_1 v(b_2, \dots b_n) )^*\, \big) \in \BT(n)
$$
does not involve brace trees with vertex  
\phantom{a}\tikz{\node [circle, draw, minimum size=5, inner sep=1] (v) at (0,0) {$1$};}\phantom{a} 
having valency $\le 2$\,.

This observation implies that the composition 
$$
\wt{F}'  =  \Tw(f) \circ c : \Ger_{\infty} \to \Tw\BT
$$
factors through the sub-operad $\Br$\,.

Since $\wt{F}'$ is homotopy equivalent to $\wt{F}$, and 
$\wt{F}'$ is compatible with the $\Tw$-coalgebra structures, 
Theorem \ref{thm:main} follows.  

\hfill \qed

\subsection{The proof of Proposition \ref{prop:Ger-infty-BT}}
\label{sec:Ger-infty-BT}

The proof of Theorem \ref{thm:main} is based on technical Proposition \ref{prop:Ger-infty-BT}. 
Here we give a proof of this proposition. 

In the proof of this proposition, it is more convenient to use 
a different but equivalent description of the pre-Lie algebra $\Conv(Q,P)$ which 
makes sense if each $Q(n)$ is finite dimensional.  This description 
is given in Appendix \ref{app:invar-coinvar}. 

Condition \eqref{f-Tcc} implies that $\al$ has the form 
\begin{equation}
\label{al-form-new}
\al = \al_0 +
 \sum_{n\ge 3,\, i}~ \sum_{\si \in S_n}  \si \Big( X_{n,i} \otimes v^1_{n,i} v^2_{n,i} \Big)\,,   
\end{equation}
where, as above, $X_{n,i} \in \BT(n)$,
$$
v^1_{n,i} v^2_{n,i} \in \bs^2 S^2(\La^{-1}\Lie(b_1, \dots, b_n)) \cap \La^{-2}\Ger(n)\,,
$$
and $\al_0$ is the degree $1$ vector 
\begin{equation}
\label{al-0}
\al_0 : = T_{\cc} \otimes b_1 b_2 +  \si_{12}(T_{\cc}) \otimes b_1 b_2\,. 
\end{equation}

Recall that, due to Proposition \ref{prop:bullet-coinv} from Appendix \ref{app:invar-coinvar}, 
the map $\Av$ \eqref{Av-dfn} gives us an isomorphism to the pre-Lie algebra 
$$
 \Conv(\Ger^{\vee}, \BT) \cong  \prod_{n \ge 2} \Big( \BT(n) \otimes \La^{-2}\Ger(n) \Big)^{S_n}
$$
from the pre-Lie algebra
\begin{equation}
\label{BT-Ger-coinv}
\prod_{n \ge 2} \Big( \BT(n) \otimes \La^{-2}\Ger(n) \Big)_{S_n}
\end{equation}
with the pre-Lie product $\bullet'$ given by equation \eqref{bullet-coinv}. 
So, in our consideration, we may replace  $\Conv(\Ger^{\vee}, \BT)$
by the pre-Lie algebra \eqref{BT-Ger-coinv}.
 
Let us denote by $\cT'(n)$ the set of all brace trees for which the labeling agrees 
with the natural order on the set of vertices. For example, the brace 
tree $T_{\cc}$ on figure \ref{fig:T12} belongs to $\cT'(2)$ and 
$\si_{12} (T_{\cc})$ does not belong to $\cT'(2)$\,. 
 
Since $\BT(n)$ is a free $S_n$-module generated elements 
of $\cT'(n)$ we conclude that 
\begin{equation}
\label{BT-Ger-easy}
\Big( \BT(n) \otimes \La^{-2}\Ger(n) \Big)_{S_n} \cong 
\bbK\L \cT'(n) \R \otimes   \La^{-2}\Ger(n)\,. 
\end{equation}

Hence the vector space \eqref{BT-Ger-easy} has the following natural basis
\begin{equation}
\label{BT-Ger-basis}
\Big\{ T \otimes w_{n,i} \Big\}_{T \in \cT'(n),\, i \in I_n}\,, 
\end{equation}
where $w_{n,i}$ are vectors of the basis \eqref{the-basis-n} for $\La^{-2}\Ger(n)$\,.

We also observe that, the $\La^{-1} \Lie$-monomials 
\begin{equation}
\label{La-1Lie-basis}
\Big\{\, v_{\tau} = \{ b_{\tau(1)}, \{ b_{\tau(2)}, \dots  b_{\tau(n-2)}, b_{n-1}\}  \br\, \Big\}_{\tau \in S_{n-2}}
\end{equation}
form a basis in the vector space $\La^{-1} \Lie(n-1)$ and 
\begin{equation}
\label{b-v-i-tau}
\Big\{ b_i v_{\tau}(b_1, \dots, \wh{b_i}, \dots, b_n) \Big\}_{1 \le i \le n;~ \tau \in S_{n-2}} 
\end{equation}
is a subset of the basis \eqref{the-basis-n} for $\La^{-2}\Ger(n)$\,.

Our goal is to show that the decomposition of the Maurer-Cartan 
element $\Av^{-1}(\al)$ with respect to the basis \eqref{BT-Ger-basis}
does not involve vectors of the form
\begin{equation}
\label{bad-vector}
T \otimes b_i v_{\tau}(b_1, \dots, \wh{b_i}, \dots, b_n)\,, 
\end{equation}
where $T$ is a brace tree in $\cT'(n)$ with vertex 
\phantom{a}\tikz{\node [circle, draw, minimum size=5, inner sep=1] (v) at (0,0) {$i$};}\phantom{a} 
being either univalent or bivalent. So the ``forbidden'' brace trees locally have one of the three 
forms shown schematically in figures \ref{fig:bag-vert-top}, \ref{fig:bag-vert-middle}, 
and \ref{fig:bag-vert-bottom}. 
\begin{figure}[htp]
\centering
\begin{minipage}[t]{0.3\linewidth}
\centering
 \begin{tikzpicture}
   \node (vv) at (0,0.3) {$\dots$};
   \node[ngr] (v) at (0,1) {};
   \node[ext] (w) at (0,2) {$i$};
   \draw (vv) edge (v) (v) edge (w) (v) edge (-0.5,1.5)  (v) edge (0.5,1.5);
  \end{tikzpicture}
\caption{\label{fig:bag-vert-top}  Vertex $i$ is univalent}  
\end{minipage}
\begin{minipage}[t]{0.3\linewidth}
\centering
\begin{tikzpicture}
   \node (dots) at (0,0.3) {$\dots$};
   \node[ngr] (gr1) at (0,1) {};
   \node[ext] (v) at (0,2) {$i$};
   \node[ngr] (gr2) at (0,3) {};
     \node (dotss) at (0,4) {$\dots$};
   \draw (dots) edge (gr1) (gr1) edge (v) (v) edge (gr2) (gr1) edge (-0.5, 1.5) edge (0.5, 1.5); 
   \draw (gr2) edge (-0.5, 3.5) edge (0, 3.5) edge (0.5, 3.5);
 \end{tikzpicture}
\caption{\label{fig:bag-vert-middle}  Bivalent vertex $i$ is not adjacent to the root}  
\end{minipage} 
\begin{minipage}[t]{0.3\linewidth}
\centering
\begin{tikzpicture}
\tikzstyle{root}=[circle, draw, fill, minimum size=0, inner sep=1]
      \node[root] (r) at (0,1) {};
   \node[ext] (v) at (0,2) {$i$};
   \node[ngr] (gr) at (0,3) {};
   \node (dotss) at (0,4) {$\dots$};
   \draw (r) edge (v) (v) edge (gr); 
   \draw (gr) edge (-0.5, 3.5) edge (0, 3.5) edge (0.5, 3.5);
 \end{tikzpicture}
\caption{\label{fig:bag-vert-bottom}  Bivalent vertex $i$ is adjacent to the root }  
\end{minipage}
\end{figure} 
Gray vertices on these figures are some other labeled vertices.

Let us introduce an operation $V^{\cced}_{i}$ whose
input is a brace tree in $\cT'(n)$ with vertex 
\phantom{a}\tikz{\node [circle, draw, minimum size=5, inner sep=1] (v) at (0,0) {$i$};}\phantom{a} 
being either univalent or bivalent. The construction of the  
output  $V^{\cced}_{i}(T)$ depends on whether vertex 
\phantom{a}\tikz{\node [circle, draw, minimum size=5, inner sep=1] (v) at (0,0) {$i$};}\phantom{a} 
is univalent or bivalent. 

If vertex 
\phantom{a}\tikz{\node [circle, draw, minimum size=5, inner sep=1] (v) at (0,0) {$i$};}\phantom{a} 
of $T$ is univalent then  $V^{\cced}_{i}(T)$
is a brace tree in $\cT'(n+1)$ which is obtained 
from $T$ via
\begin{itemize}

\item shifting labels of $T$ which are $> i$ up by $1$,

\item creating a vertex with label $i+1$, and 

\item attaching the vertex with label $i+1$ to the vertex with label $i$ by 
a single edge.

\end{itemize}

If vertex 
\phantom{a}\tikz{\node [circle, draw, minimum size=5, inner sep=1] (v) at (0,0) {$i$};}\phantom{a} 
of $T$ is bivalent then  $V^{\cced}_{i}(T)$
is a brace tree in $\cT'(n+1)$ which is obtained from $T$ via 

\begin{itemize}

\item shifting labels of $T$ which are $> i$ up by $1$,

\item creating a vertex with label $i+1$,

\item detaching the edge terminating at the vertex with label $i$
and attaching it to the vertex with label $i+1$,

\item and finally, attaching the vertex with label $i+1$ to the label $i$ by 
a single edge.

\end{itemize}

In pictures, $V^{\cced}_{i}(T)$ locally looks as follows:
\begin{figure}[htp]
\centering
\begin{minipage}[t]{0.3\linewidth}
\centering
 \begin{tikzpicture}
  \node (vv) at (0,0.3) {$\dots$};
  \node[ngr] (v) at (0,1) {};
  \node[ext] (w) at (0,2) {$i$};
  \node[ext] (ww) at (0,3) {\scriptsize{$i$+1}};
  \draw (vv) edge (v) (v) edge (w) (v) edge (-0.5, 1.5) edge (0.5, 1.5)  (w) edge (ww);
 \end{tikzpicture}
\end{minipage}
\begin{minipage}[t]{0.3\linewidth}
\centering 
 \begin{tikzpicture}
  \node (vvv) at (0,-0.7) {$\dots$};
  \node[ngr] (vv) at (0,0) {};
  \node[ext] (v) at (0,1) {$i$};
  \node[ext] (ww) at (0,2) {\scriptsize{$i$+1}};
  \node[ngr] (w) at (0,3) {};
  \node (dots) at (0,4) {$\dots$};
  \draw  (vvv) edge (vv) (vv) edge (v) (v) edge (ww) (ww) edge (w);
  \draw  (vv) edge (-0.5,0.5) edge (0.5, 0.5) (w) edge (-0.5, 3.5) edge (0,3.5) edge (0.5, 3.5);
 \end{tikzpicture}
\end{minipage} 
\begin{minipage}[t]{0.3\linewidth}
\centering
 \begin{tikzpicture}
 \tikzstyle{root}=[circle, draw, fill, minimum size=0, inner sep=1] 
  \node[root] (r) at (0,0) {};
  \node[ext] (v) at (0,1) {$i$};
  \node[ext] (ww) at (0,2) {\scriptsize{$i$+1}};
  \node[ngr] (w) at (0,3) {};
  \node (dots) at (0,4) {$\dots$};
  \draw  (r) edge (v) (v) edge (ww) (ww) edge (w);
  \draw  (w) edge (-0.5, 3.5) edge (0,3.5) edge (0.5, 3.5);
 \end{tikzpicture}
\end{minipage}
\end{figure}

Unfolding the definition, it is not hard to see that 
for every vector of the form \eqref{bad-vector} 
\begin{equation}
\label{al-0-T-b-vbbb}
\al_0 \bullet' T \otimes b_i v_{\tau}(b_1, \dots, \wh{b_i}, \dots, b_n) + 
T \otimes b_i v_{\tau}(b_1, \dots, \wh{b_i}, \dots, b_n) \bullet' \al_0 = 
\end{equation}
$$
\pm\, V^{\cced}_i(T) \otimes  b_i b_{i+1} v_{\tau}(b_1, \dots, b_{i-1}, b_{i+2}, \dots, b_n) + 
\dots\,,
$$
where $\dots$ denote the linear combination of the remaining vectors in 
the basis for \eqref{BT-Ger-easy}. 

\begin{remark}
Note that graphically the operation $\bullet' \al_0$ is the summation over
splittings of a labeled vertex into two labeled vertices, 
while the operation $\al_0 \bullet'$ may graphically be interpreted as attaching a new labeled vertex in all possible ways 
``from the top'' and ``from the bottom'' of the tree. 
Hence the first term on the right hand side of \eqref{al-0-T-b-vbbb} is produced three times on 
the left hand side of \eqref{al-0-T-b-vbbb}.
If $T$ is of the form shown on figure \ref{fig:bag-vert-top} or of the form 
shown on figure \ref{fig:bag-vert-bottom} then the first contribution comes from the
operation $\al_0 \bullet'$, the second one comes from splitting vertex $i$ and the 
third one comes from splitting the labeled vertex adjacent to vertex $i$. 
If $T$ is of the form shown on figure \ref{fig:bag-vert-middle} then all these 
three contributions come from the operation $\bullet' \al_0$: the first one comes from 
the splitting vertex $i$ and the other two contributions come from splitting 
the labeled vertices adjacent to vertex $i$. 
We claim that in all three cases the signs let two of the three contributions cancel.
\end{remark}

It is clear that, if $v v'$ (resp. $w$) is a vector 
in $\La^{-2}\Ger(n)$ (resp. in   $\La^{-2}\Ger(m)$) with the 
$\La^{-1}\Lie$-words $v, v'$ having length $\ge 2$, 
then the pre-Lie products 
$$
\big(T_1 \otimes v v'\big) \bullet' \big(T_2 \otimes w \big)
$$
and
$$
\big(T_2 \otimes w \big) \bullet'  \big(T_1 \otimes v v'\big) 
$$
do not involve tensor factors of the form 
$b_i b_{i+1} v_{\tau}(b_1, \dots, b_{i-1}, b_{i+2}, \dots, b_n)$\,.

It is also clear that, if the vertex with label $j_1$ (resp. label $j_2$) of a 
brace tree $T_1 \in \cT'(n)$ (resp. $T_2 \in \cT'(m)$) has valency $\ge 3$, then 
for every pair of basis vectors 
$v_{\tau} \in \La^{-1}\Lie(n-1)$, $v_{\tau'} \in \La^{-1}\Lie(m-1)$,  the
decomposition of the pre-Lie product
$$
\big(T_1 \otimes b_{j_1} v_{\tau}(b_1, \dots , \wh{b_{j_1}}, \dots, b_n) \big) 
\bullet' \big(T_2 \otimes   b_{j_2} v_{\tau'}(b_1, \dots , \wh{b_{j_2}}, \dots, b_m)  \big)
$$
in the basis \eqref{BT-Ger-basis} does not involve vectors of the form
\begin{equation}
\label{the-form}
V^{\cced}_i(T) \otimes  b_i b_{i+1} v_{\tau}(b_1, \dots, b_{i-1}, b_{i+2}, \dots, b_n)\,.
\end{equation}

These observations imply that, if the decomposition of the 
element $\Av^{-1}(\al)$ in the basis \eqref{BT-Ger-basis}
involves a vector of the form \eqref{bad-vector}, then 
$$
\Av^{-1}(\al) \bullet' \Av^{-1}(\al) \neq 0\,.
$$

On the other hand,  $\Av^{-1}(\al)$ is a Maurer-Cartan element of 
\eqref{BT-Ger-easy}. Thus, the proposition follows. 

\hfill \qed

\subsection{Proof of Corollary \ref{cor:main}}
\label{sec:cor-main}

In this section, we give a proof of Corollary \ref{cor:main} of 
Theorem \ref{thm:main} stated in the Introduction. 

Let us assume that the quasi-isomorphism
$$
F : \Ger_{\infty} \to \Br
$$
is compatible with the $\Tw$-coalgebra structures on 
$\Ger_{\infty}$ and $\Br$\,. This means that the diagram 
\begin{equation}
\label{F-TwF}
\begin{tikzpicture}
\matrix (m) [matrix of math nodes, row sep=2em, column sep =3em]
{ \Tw\Ger_{\infty} & \Tw \Br   \\
  \Ger_{\infty} & \Br \\ };
\path[->,font=\scriptsize]
(m-1-1) edge node[auto] {$\Tw F$}(m-1-2)  
(m-2-1) edge node[auto] {$c_{\Ger_{\infty}}$} (m-1-1)
 edge node[auto] {$F$} (m-2-2)
 (m-2-2)  edge node[auto] {$c_{\Br}$} (m-1-2);
\end{tikzpicture}
\end{equation}
commutes. 

Following Theorem \ref{thm:twisting} and using the 
Maurer-Cartan element $\al$ of $\Cbu(A)$, we form an 
operad map
\begin{equation}
\label{TwBr-EndCbu}
\ma^{\al}_A : \Tw\Br \to \End_{\Cbu(A)^{\al}}\,.
\end{equation}

Unfolding definitions, we see that the composition 
$$
\ma^{\al}_A \circ  \Tw(F) \circ c_{\Ger_{\infty}} : \Ger_{\infty} \to 
 \End_{\Cbu(A)^{\al}}
$$
gives us the $\Ger_{\infty}$-structure $F^{\al}_A$ and the 
composition 
$$
\ma^{\al}_A \circ  c_{\Br} \circ F  : \Ger_{\infty} \to 
 \End_{\Cbu(A)^{\al}}
$$
give us the $\Ger_{\infty}$-structure $F_{A^{\al}}$\,. 

Thus, since diagram \eqref{F-TwF} commutes, the $\Ger_{\infty}$-structures 
$F_{A^{\al}}$ and  $F^{\al}_A$ coincide and the second statement of Corollary \ref{cor:main}
is proved. 

Let us now consider the case when 
$$
F : \Ger_{\infty} \to \Br
$$
is an arbitrary solution of Deligne's conjecture (not necessarily compatible with 
$\Tw$-coalgebra structures).

Due to Theorem \ref{thm:main}, $F$ is homotopy equivalent to a quasi-isomorphism 
$$
F' : \Ger_{\infty} \to \Br
$$ 
that is compatible with the $\Tw$-coalgebra structures on 
$\Ger_{\infty}$ and $\Br$\,. 

Hence diagram \eqref{F-TwF} commutes up to homotopy. 
Therefore, the two (possibly different) $\Ger_{\infty}$-structures $F_{A^{\al}}$ 
and  $F^{\al}_A$ are indeed homotopy equivalent. 

Corollary \ref{cor:main} is proved. 

\hfill \qed

\appendix

\section{The operad \texorpdfstring{$\Ger$}{Ger}}
\label{app:Ger}
A Gerstenhaber algebra is a graded vector space $V$
equipped with a commutative (and associative) product 
(without identity) and a degree $-1$ binary operation $\{\,,\,\}$ which 
satisfies the following relations: 
\begin{gather}
\label{Ger-axiom}
\{v_1 , v_2 \}  = (-1)^{|v_1||v_2|} \{ v_2, v_1\}\,,
\\
\label{Ger-axiom1}
\{v , v_1 v_2 \}  = \{v, v_1\} v_2 + (-1)^{|v_1||v| + |v_1|} v_1 \{v, v_2\} \,,
\\
\label{Ger-axiom-Jac}
\{\{v_1, v_2\} , v_3\} +  
(-1)^{|v_1|(|v_2|+|v_3|)} \{\{v_2, v_3\} , v_1\} +
(-1)^{|v_3|(|v_1|+|v_2|)} \{\{v_3, v_1\} , v_2\}  = 0\,.
\end{gather}

To define spaces of the operad $\Ger$ governing Gerstenhaber algebras 
we introduce the free Gerstenhaber algebra $\Ger_n$ in $n$ dummy 
variables  $a_1, a_2, \dots, a_n$ of degree $0$\,.  Next we set
$\Ger(0) =  \bfzero$ and $\Ger(1) = \bbK$\,. And then we declare that, 
for $n\ge 2$, $\Ger(n)$ is spanned by monomials of $\Ger_n$ in which 
each dummy variable $a_i$ appears exactly once. 

The symmetric group $S_n$ acts on $\Ger(n)$ in the obvious 
way by permuting the dummy variables. It is also clear how 
to define elementary insertions.  
\begin{example}
\label{exam:Ger-insertions}
Let us consider the monomials $u = \{a_2, a_3\} a_1 \{a_4,a_5\} \in \Ger(5)$ and
$w =\{a_1,a_2\}\in \Ger(2)$ and compute the insertions 
$u \circ_2 w$, $u \circ_4 w$ and $w \circ_1 u$\,. We get 
\begin{align*}
u \circ_2 w &=  - \{\{a_2, a_3\}, a_4\} a_1 \{a_5,a_6\} 
\\
u \circ_4 w &=  \{a_2, a_3\} a_1 \{\{a_4, a_5 \},a_6\}
\\
w \circ_1 u &=  \{ \{a_2, a_3\} a_1 \{a_4,a_5\} , a_6\} 
=  \{a_6,  \{a_2, a_3\} a_1 \{a_4,a_5\} \}
\\
& =  \{a_6,  \{a_2, a_3\} \} a_1 \{a_4,a_5\} -  \{a_2, a_3\} \{a_6, a_1\} \{a_4,a_5\}
\\
& \qquad - \{a_2, a_3\} a_1 \{a_6, \{a_4,a_5\}\}\,.  
\end{align*}
\end{example}
It is easy to see that the operad $\Ger$ is generated by the vectors
$a_1 a_2, \{a_1,a_2\} \in \Ger(2)$\,.

\section{Action of the operad \texorpdfstring{$\Tw\BT$}{Tw BT} on the Hochschild cochain 
complex of an \texorpdfstring{$A_{\infty}$}{A-infinity}-algebra}
\label{app:action}

Let $A$ be a cochain complex. We form another cochain complex
\begin{equation}
\label{Cbu-A}
\Cbu(A) = \bigoplus_{m\ge 0} \bs^{m}\, \Hom(A^{\otimes \,m}, A)
\end{equation}
with the differential $\pa_A$ coming from $A$\,.

Let us show that \eqref{Cbu-A} is a naturally an algebra over 
$\BT$\,. To define an action  $\BT$ on  \eqref{Cbu-A} we observe that 
the collection 
\begin{equation}
\label{End-A}
\{\Hom(A^{\otimes \,m}, A)\}_{m \ge 0}
\end{equation}
is an naturally an operad, i.e. the endomorphism operad of $A$\,.

Using this observation, we will construct an auxiliary map
\begin{equation}
\label{BTCbu-to-A}
\vr : \bigoplus_{N \ge 1} \left( \BT(N) \otimes \big(\Cbu(A) \big)^{\otimes\, N} \right)_{S_N}  \to  A
\end{equation}

Given a brace tree $T \in \cT(N)$ and $N$ homogeneous 
vectors 
\begin{equation}
\label{P1P2dotsPN}
P_i \in  \bs^{m_i}\, \Hom(A^{\otimes \,m_i}, A)\,, \qquad 1 \le i \le N
\end{equation}
we decorated non-root vertices of $T$ with vectors \eqref{P1P2dotsPN}
following this rule: the vertex with label $i$ is decorated by 
the vector $P_i$\,. 

We say that such a decoration is {\it admissible} 
if for every $1 \le i  \le N$ the vertex with label $i$ has 
exactly $m_i$ incoming edges. Otherwise, we say that 
the decoration is {\it inadmissible}. In particular, if a decoration 
is admissible, then each leaf of $T$ is decorated by a 
vector in 
$$
\Hom(A^{\otimes 0}, A) =  A\,.
$$

Since \eqref{End-A} is an operad, each brace 
tree $T\in \cT(N)$ with an admissible decoration gives us a vector in $A$\,.
We denote this vector in $A$ by 
\begin{equation}
\label{m-TPPP}
\mm (T; P_1, P_2, \dots, P_N)\,. 
\end{equation}
For example, if $T$ the brace tree depicted on figure \ref{fig:T}, 
$P_1 \in \bs \Hom(A, A)$, $P_2 \in A$, $P_3 \in \bs^2 \Hom(A^{\otimes 2}, A)$, 
and $P_4 \in A$
\begin{figure}
\centering
\begin{tikzpicture}[scale=0.5]
\tikzstyle{lab}=[circle, draw, minimum size=5, inner sep=1]
\tikzstyle{labb}=[circle, draw, minimum size=5, inner sep=3]
\tikzstyle{n}=[circle, draw, fill, minimum size=5]
\tikzstyle{root}=[circle, draw, fill, minimum size=0, inner sep=1]
\node[root] (r) at (0, 0) {};
\node[lab] (v1) at (0, 1) {$3$};
\node[lab] (v2) at (-0.5, 2.3) {$1$};
\node[lab] (v3) at (-1, 3.6) {$2$};
\node[lab] (v4) at (1, 2) {$4$};
\draw (r) edge (v1);
\draw (v1) edge (v2);
\draw (v2) edge (v3);
\draw (v1) edge (v4);
\end{tikzpicture}
\caption{A brace tree $T \in \cT(4)$} \label{fig:T}
\end{figure}  
then we have 
$$
\mm(T; P_1, P_2, P_3, P_4) = P_3( P_1(P_2), P_4)\,.
$$

We can now define the map $\vr$ \eqref{BTCbu-to-A}.

If $T$ is a brace tree in $\cT(N)$ and cochains $P_1, P_2, \dots, P_N$ give 
us an inadmissible decoration of $T$ then we set 
\begin{equation}
\label{T-P-zero}
\vr(T, P_1, \dots, P_N) =0\,.
\end{equation}
Otherwise, we declare that
\begin{equation}
\label{T-PPP}
\vr(T, P_1, \dots, P_N) : =  
(-1)^{\ve(T; P_1, \dots, P_N)} \mm(T; P_1, \dots, P_N)  
\end{equation}
where the sign factor $(-1)^{\ve(T; P_1, \dots, P_N)}$ comes from 
permutation on the set 
$$
\left\{ E(T) \setminus \{\textrm{root edge}\} \right\} \sqcup \{P_1, \dots, P_N\} 
$$
which we perform when we decorate the tree $T$ with vectors 
$P_1, P_2, \dots, P_N$\,.

For example, if $T$ is the brace tree depicted on figure \ref{fig:T},  
 $P_1 \in \bs \Hom(A, A)$, $P_2 \in A$, $P_3 \in \bs^2 \Hom(A^{\otimes 2}, A)$, 
and $P_4 \in A$ then 
$$
\ve(T; P_1, \dots, P_N) =  |P_3| (|P_1| + |P_2|)  +  3 |P_3| + 2 |P_1| + |P_2|\,.
$$ 
The sign factor $(-1)^{ |P_3| (|P_1| + |P_2|)}$ appears because we 
need to switch from the order $(P_1, P_2, P_3, P_4)$ to 
the order  $(P_3, P_1, P_2, P_4)$ ; the sign factor $(-1)^{ 3 |P_3| }$
appears because $P_3$ ``jumps over'' three edges of brace tree $T$\,. 
Similarly, $P_1$ (resp. $P_2$)  ``jumps over'' two edges (resp. one edge) 
of the brace tree $T$\,.

We can now define how a brace tree $T \in \cT(n)$ acts on 
$n$ homogeneous vectors 
\begin{equation}
\label{PPdotsP}
P_i \in  \bs^{m_i}\, \Hom(A^{\otimes \,m_i}, A)\,, \qquad 1 \le i \le n\,.
\end{equation}

For this purpose, we form the linear combinations ($k \ge 0$) 
\begin{equation}
\label{Tk-circ-T}
T_k \circ_1 T\,,
\end{equation}
where $T_k$ is the brace corolla shown on figure \ref{fig:T-k}.
\begin{figure}
\centering
\begin{tikzpicture}[scale=0.5]
\tikzstyle{lab}=[circle, draw, minimum size=5, inner sep=1]
\tikzstyle{labb}=[circle, draw, minimum size=5, inner sep=3]
\tikzstyle{n}=[circle, draw, fill, minimum size=5]
\tikzstyle{root}=[circle, draw, fill, minimum size=0, inner sep=1]
\node[root] (r) at (0, 0) {};
\node[labb] (v1) at (0, 1.5) {$1$};
\node[labb] (v2) at (-2.5, 3.3) {$2$};
\draw (0.5,3.3) node[anchor=center] {{\small $\dots$}};
\node[labb] (v3) at (-1, 3.3) {$3$};
\node[lab] (vk1) at (2.5, 3.3) {{\small $k+1$}};
\draw (r) edge (v1);
\draw (v1) edge (v2);
\draw (v1) edge (v3);
\draw (v1) edge (vk1);
\end{tikzpicture}
\caption{The brace corolla $T_{k}\in \cT(k+1)$} \label{fig:T-k}
\end{figure}  
and 
\begin{equation}
\label{k}
k = \sum_{i=1}^{n} m_i + 1 -n\,.
\end{equation}

Next we set 
\begin{equation}
\label{T-acts}
T(P_1, \dots, P_n; a_1, \dots, a_k) = 
\vr (T_k \circ_1 T,  P_1, \dots, P_n, a_1, \dots, a_k ),
\end{equation}
where $ a_1, \dots, a_k $ are viewed as vectors 
in
$$
\Hom(A^{\otimes\, 0},  A)\,.
$$
Note that, the vectors  $a_1, \dots, a_k $ will 
decorate leaves (of brace trees in the linear combination 
$T_k \circ_1 T$)
with labels $n+1, n+2, \dots, n+k$\,. Moreover, if 
equation \eqref{k} were not satisfied then all decorations 
of $T_k \circ_1 T$ would be inadmissible.

Tedious but straightforward computations show that 
formula  \eqref{T-acts} indeed defines an action of the 
operad $\BT$ on \eqref{Cbu-A}. 

\begin{example}
\label{ex:action}
Let $T$ be the brace tree in $\cT(4)$ depicted on
figure \ref{fig:T} and 
\begin{equation}
\label{PPPP}
\begin{array}{cc}
P_1 \in \bs^2\, \Hom(A \otimes A, A)\,, &
P_2 \in \bs\, \Hom(A , A)\,, \\[0.3cm] 
P_3 \in \bs^3\, \Hom(A^{\otimes\, 3}, A)\,, &
P_4 \in   \Hom(A^{\otimes\, 0}, A) = A\,.
\end{array}
\end{equation}
The vector $T (P_1, P_2, P_3, P_4) $ belongs to 
$$
\bs^3 \Hom(A^{\otimes \, 3}, A)
$$

The linear combination
$T_3 \circ_1 T$ contains a lot of terms. On figure 
\ref{fig:list} we list all brace trees (with signs) which 
get admissible decorations by the vectors 
\begin{equation}
\label{position}
P_1, P_2, P_3, P_4, a_1, a_2, a_3
\end{equation}
\begin{figure}
\begin{minipage}[t]{0.3\linewidth} 
\centering
\begin{tikzpicture}[scale=0.5]
\tikzstyle{lab}=[circle, draw, minimum size=5, inner sep=1]
\tikzstyle{labb}=[circle, draw, minimum size=5, inner sep=3]
\tikzstyle{n}=[circle, draw, fill, minimum size=5]
\tikzstyle{root}=[circle, draw, fill, minimum size=0, inner sep=1]
\draw (-1, 2) node[anchor=center] {{$-$}};
\node[lab] (v7) at (3, 5) {{\small $7$}};
\node[lab] (v6) at (1, 3.5) {{\small $6$}};
\node[lab] (v2) at (3, 3.5) {{\small $2$}};
\node[lab] (v5) at (0.5, 2) {{\small $5$}};
\node[lab] (v1) at (2, 2) {{\small $1$}};
\node[lab] (v4) at (3.5, 2) {{\small $4$}};
\node[lab] (v3) at (2, 0.5) {{\small $3$}};
\node[root] (r) at (2, -0.5) {};
\draw (r) edge (v3);
\draw (v3) edge (v5);
\draw (v3) edge (v1);
\draw (v3) edge (v4);
\draw (v1) edge (v6);
\draw (v1) edge (v2);
\draw (v2) edge (v7);
\end{tikzpicture}
\end{minipage}
\begin{minipage}[t]{0.3\linewidth} 
\centering
\begin{tikzpicture}[scale=0.5]
\tikzstyle{lab}=[circle, draw, minimum size=5, inner sep=1]
\tikzstyle{labb}=[circle, draw, minimum size=5, inner sep=3]
\tikzstyle{n}=[circle, draw, fill, minimum size=5]
\tikzstyle{root}=[circle, draw, fill, minimum size=0, inner sep=1]
\draw (-1, 2) node[anchor=center] {{$+$}};
\node[lab] (v7) at (3, 3.5) {\small ${ 7}$};
\node[lab] (v6) at (1, 5) {\small ${ 6}$};
\node[lab] (v2) at (1, 3.5) {\small ${2}$};
\node[lab] (v5) at (0.5, 2) {\small ${5}$};
\node[lab] (v1) at (2, 2) {\small ${1}$};
\node[lab] (v4) at (3.5, 2) {\small ${4}$};
\node[lab] (v3) at (2, 0.5) {\small ${3}$};
\node[root] (r) at (2, -0.5) {};
\draw (r) edge (v3);
\draw (v3) edge (v5);
\draw (v3) edge (v1);
\draw (v3) edge (v4);
\draw (v1) edge (v2);
\draw (v1) edge (v7);
\draw (v2) edge (v6);
\end{tikzpicture}
\end{minipage}
\begin{minipage}[t]{0.3\linewidth} 
\centering
\begin{tikzpicture}[scale=0.5]
\tikzstyle{lab}=[circle, draw, minimum size=5, inner sep=1]
\tikzstyle{labb}=[circle, draw, minimum size=5, inner sep=3]
\tikzstyle{n}=[circle, draw, fill, minimum size=5]
\tikzstyle{root}=[circle, draw, fill, minimum size=0, inner sep=1]
\draw (-2, 2) node[anchor=center] {{$-$}};
\node[lab] (v6) at (1, 5) {\small ${6}$};
\node[lab] (v5) at (-1, 3.5) {\small ${5}$};
\node[lab] (v2) at (1, 3.5) {\small ${2}$};
\node[lab] (v1) at (0, 2) {\small ${1}$};
\node[lab] (v7) at (2, 2) {\small ${7}$};
\node[lab] (v4) at (3.5, 2) {\small ${4}$};
\node[lab] (v3) at (2, 0.5) {\small ${3}$};
\node[root] (r) at (2, -0.5) {};
\draw (r) edge (v3);
\draw (v3) edge (v1);
\draw (v3) edge (v7);
\draw (v3) edge (v4);
\draw (v1) edge (v5);
\draw (v1) edge (v2);
\draw (v2) edge (v6);
\end{tikzpicture}
\end{minipage}
\caption{The brace trees which contribute to 
$T (P_1, P_2, P_3, P_4)$} \label{fig:list}
\end{figure}
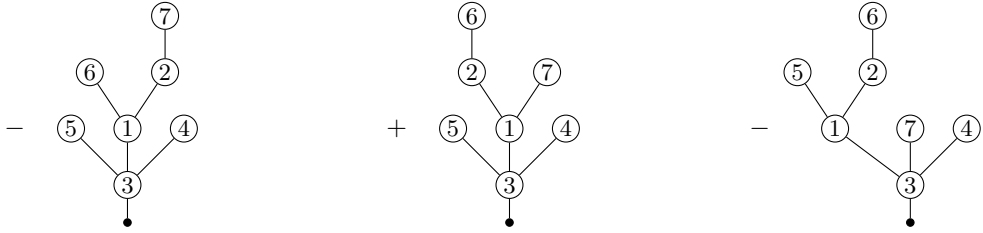    
Thus we get 
\begin{multline}
\label{T-PPPP}
T \big(P_1, P_2, P_3, P_4 \big) (a_1, a_2, a_3)   
= -(-1)^{\ve_1} 
P_3(a_1, P_1(a_2, P_2 (a_3)), P_4)+ \\
(-1)^{\ve_2}  P_3( a_1, P_1 ( P_2(a_2), a_3),  P_4) -
(-1)^{\ve_3}  P_3( P_1( a_1, P_2(a_2)), a_3, P_4)\,,
\end{multline}
where 
\begin{gather*}
\ve_1 = \ve(P_3, a_1, P_1, a_2, P_2, a_3, P_4) 
+  6 |P_3| + 5|a_1| + 4|P_1| + 3 |a_2| + 
2 |P_2| + |a_3|\,, 
\\
\ve_2 = \ve(P_3, a_1, P_1, P_2, a_2, a_3, P_4)
 +  6 |P_3| + 5|a_1| + 4|P_1| + 3 |P_2| + 
2 |a_2| + |a_3|\,,
\\
\ve_3 = \ve( P_3, P_1, a_1, P_2, a_2, a_3, P_4)
+  6 |P_3| +  5 |P_1| + 4 |a_1| + 3 |P_2| + 
2 |a_2| + |a_3|\,,
\end{gather*}
and $\ve(\dots)$ is the sign of the permutation of vectors 
$P_1, P_2, P_3, P_4, a_1, a_2, a_3$ from their standard 
position \eqref{position}.
\end{example}

Since $\BT$ receives the operad map \eqref{LaLie-BT} from $\La\Lie$,
the cochain complex \eqref{Cbu-A} is naturally a $\La\Lie$ algebra. 
The bracket is the famous Gerstenhaber bracket introduced 
in \cite{Ger}. 

Let us observe that 
the cochain complex $\Cbu(A)$ \eqref{Cbu-A} is equipped 
with the obvious decreasing filtration:
\begin{equation}
\label{filtr-Cbu}
\cF_q\Cbu(A) = 
\Big\{ P \in \Cbu(A) ~\big|~ P(a_1, a_2, \dots, a_k) = 0 ~~ \forall ~ k \le q  \Big\}
\end{equation}
and the action of $\BT$ on $ \Cbu(A)$ is compatible with 
this filtration.
Thus we may apply  the general procedure of twisting described 
in Section \ref{sec:twist} to the completion 
\begin{equation}
\label{Hoch-completed}
\hat{C}^{\bul}(A) : =  \prod_{m\ge 0} \bs^{m}\, \Hom(A^{\otimes \,m}, A) 
\end{equation}
of the $\BT$-algebra $\Cbu(A)$ with respect to the filtration \eqref{filtr-Cbu}.

According to Theorem \ref{thm:twisting-cont}, filtered $\Tw\BT$-algebra structures 
on $\hat{C}^{\bul}(A)$ are in bijection with Maurer-Cartan elements $\al$ in $\hat{C}^{\bul}(A)$ 
of the form 
\begin{equation}
\label{MC-Ainfty}
\al = \sum_{k \ge 2} \al_k\,, \qquad 
\al_k \in \bs^k \Hom(A^{\otimes k}, A)
\end{equation}
i.e.  flat $A_{\infty}$-structures on the cochain
complex $A$\,. 

Moreover, for every Maurer-Cartan element \eqref{MC-Ainfty}, 
the cochain complex  $\hat{C}^{\bul}(A)$ with the twisted differential \eqref{tw-diff} is exactly 
the (completed) Hochschild cochain complex of the  $A_{\infty}$-algebra $A$\,.

Thus we conclude that, for every flat $A_{\infty}$-algebra 
$A$, its completed Hochschild cochain complex is equipped with 
a natural action of the operad $\Tw\BT$ and hence 
with a natural action of the operad $\Br$\,.

\begin{remark}
\label{rem:non-complete}
Sometimes the Maurer-Cartan element  \eqref{MC-Ainfty}
satisfies the condition 
$$
\al_{k} = 0 
$$
for all $k \ge N$ for some given number $N$\,. 
(For example, if $A$ is an associative algebra.)
In this case, the action of $\Tw\BT$ (and hence the action 
of $\Br$) is well defined before completing the Hochschild cochain 
complex. 
\end{remark}

\section{The equivalent definition of the pre-Lie product $\bullet$ on $\Conv(Q,P)$ in terms of coinvariants}
\label{app:invar-coinvar}

Let us assume that the pseudo-cooperad $Q$
satisfies the following property.
\begin{pty}
\label{P:fin-dim}
For each $n$ the graded vector space $Q(n)$ is 
finite dimensional. 
\end{pty}

Due to this property we have 
\begin{equation}
\label{Conv-fin}
\Conv(Q, P)  \cong \prod_{n \ge 0} 
\big( P(n) \otimes Q^*(n) \big)^{S_n}\,.
\end{equation}
where $Q^*(n)$ denotes the linear dual of 
the vector space $Q(n)$\,.

The collection $Q^* := \{Q^*(n)\}_{n \ge 0}$ is naturally 
a pseudo-operad and we can express the pre-Lie structure  
\eqref{Conv-bullet} in terms of elementary insertions on 
$P$ and $Q^*$. Namely, given two vectors 
$$
X = \sum_{n \ge 0} v_n \otimes w_n, \qquad 
X' = \sum_{n \ge 0} v'_n \otimes w'_n
$$
in 
$$
\prod_{n \ge 0} \big( P(n) \otimes Q^*(n) \big)^{S_n}\,,
$$
we have 
\begin{equation}
\label{bullet-fin}
X \bullet X' = \sum_{n \ge 1, m \ge 0} (-1)^{|v'_m| |w_n|}
\sum_{\si \in \Sh_{m, n-1}}
\si(v_n \circ_1 v'_m) \otimes \si(w_n \circ_1 w'_m)\,.
\end{equation}

Let us now observe that the formula 
\begin{equation}
\label{Av-dfn}
\Av (v \otimes w) = \sum_{\si \in S_n} \si(v) \otimes \si(w)
\end{equation}
defines a $\bbK$-linear map 
$$
\Av :  P(n) \otimes Q^*(n) \to  \big( P(n) \otimes Q^*(n) \big)^{S_n} \,.
$$
Furthermore, since $\bbK$ has characteristic zero, $\Av$ induces 
an isomorphism (which we denote by the same letter): 
\begin{equation}
\label{Av-iso}
\Av :  \big( P(n) \otimes Q^*(n) \big)_{S_n} \stackrel{\cong}{\,\longrightarrow\,} 
  \big( P(n) \otimes Q^*(n) \big)^{S_n}
\end{equation}
from the space
\begin{equation}
\label{coinvar}
 \big( P(n) \otimes Q^*(n) \big)_{S_n} 
\end{equation}
of $S_n$-coinvariants to the space $  \big( P(n) \otimes Q^*(n) \big)^{S_n}$ of 
invariants.

We have the following proposition.
\begin{prop}
\label{prop:bullet-coinv}
Let $v, v', w, w'$ be homogeneous vectors in $P(n)$, $P(m)$, 
$Q(n)$, and $Q'(m)$, respectively. Then the formula 
\begin{equation}
\label{bullet-coinv}
(v \otimes w) \bullet' (v' \otimes w') =
\sum_{i=1}^n (-1)^{|v'| |w|} 
v \circ_i v' \otimes w \circ_i w'
\end{equation}
defines a binary operation on 
$$
\bigoplus_{n \ge 0}   \big( P(n) \otimes Q^*(n) \big)_{S_n}\,.
$$
This operation extends in the obvious way to the infinite sum 
\begin{equation}
\label{prod-coinv}
\prod_{n \ge 0}   \big( P(n) \otimes Q^*(n) \big)_{S_n}
\end{equation}
and 
\begin{equation}
\label{bul-pr-Av-bul}
\Av\big( \,(v \otimes w) \bullet' (v' \otimes w') \, \big) = 
\Av(v \otimes w) \bullet  \Av (v' \otimes w')\,.
\end{equation}
\end{prop}
\begin{proof}
Since the map \eqref{Av-iso} is an isomorphism of 
graded vector spaces, it suffices to prove that equation
\eqref{bul-pr-Av-bul} holds. 

This equation can be verified by a straightforward computation using the obvious 
identity\footnote{Here $S_{\{2,3, \dots, n\}}$ is the subgroup of elements $\si \in S_n$
satisfying the condition $\si(1) =1$\,.}
$$
\sum_{\si \in S_n} \si (v) = \sum_{\si' \in S_{\{2,3, \dots, n\}} } \sum_{i=1}^n \si' \circ (1,2, \dots, i)(v)  
$$
and properties of the elementary insertions. 
\end{proof}

Proposition \ref{prop:bullet-coinv} implies that the map 
\eqref{Av-iso} establishes an isomorphism between the pre-Lie 
algebra \eqref{Conv-fin} with the pre-Lie product  \eqref{bullet-fin}
and the pre-Lie algebra \eqref{prod-coinv} with the pre-Lie product 
given by equation \eqref{bullet-coinv}.

\section{A Lemma on a filtered complex}
\label{app:filtered-lem}

\begin{lemma}
\label{lem:filtered}
Let $C$ be a cochain complex with a complete descending filtration which is 
bounded above. I.e., 
\[
C= \mF_{-N} \supset  \mF_{-N+1} \supset \dots  \supset \mF_{0}   \supset \mF_1 \supset \cdots
\]
and $C= \lim_p C / \mF_p$. If the associated graded complex $\Gr C$ is acyclic then 
so is $C$\,.
\end{lemma}
\begin{proof}
Suppose $c\in C$ is a cocycle. Then there exists $j$ such that $c\in \mF_j$. 
Then by acyclicity of the associated graded complex, there is a $b_j\in \mF_j$ such that $c-db_j\in \mF_{j+1}$. Proceed in this manner to define $b_{j+1}$, $b_{j+2}$ etc. Then set $b:= \sum_j b_j$. The sum converges since $C$ is complete 
with respect to the filtration.
Again by completeness 
\[
c - db=\lim_{N\to \infty} c-\sum_{j\leq N} db_j=0.
\]
Hence $c$ is exact.
\end{proof}

\section{\texorpdfstring{$\Br$}{Br} is a homotopy \texorpdfstring{$\Tw$}{Tw} fixed point -- 
combinatorial argument}
\label{app:brhomfixed}
We showed in Section \ref{sec:Br-homotopyfixed} above that the operad $\Br$ is a homotopy fixed point of $\Tw$, i. e., that $\Tw\Br\to \Br$ is a quasi-isomorphism. One may give an alternative proof of that fact by an elementary combinatorial argument, which we briefly sketch in this appendix. 

Elements of $\Tw\Br$ can be seen as series of brace trees, for which some of the 
labeled vertices have been colored in, say, gray. 
The differential is schematically depicted in Figure \ref{fig:g1diff}.
The number of neutral (``black'') vertices of brace trees yields a descending complete filtration 
\[
\Tw\Br = \mathcal{F}^0 \supset \mathcal{F}^1 \supset \mathcal{F}^2 \supset \cdots
\]
where $\mathcal{F}^p$ is composed of series in trees with $\geq p$ neutral (black) vertices. Let us consider the associated graded $\Gr \mathcal{F}$. Its differential misses those terms of Figure \ref{fig:g1diff} that produce a black vertex.
We claim that $V_p:=\mathcal{F}^p / \mathcal{F}^{p+1}$ is acyclic for $p\geq 1$.
To show this we need to use additional notation. 
For a brace tree, the \emph{first} neutral vertex is the one hit first when going around he tree in clockwise order, see Figure \ref{fig:firstint}.
We filter $V_p$ by the number of valence zero gray vertices attached to the very left of the first internal vertex, see Figure \ref{fig:veryleft}.
Taking a spectral sequence, the first differential increases that number by one, see Figure \ref{fig:graydiff}. It is easy to see that the cohomology under this differential is zero. It follows that $V_p$ is acyclic as claimed.
Hence the projection $\Gr \mathcal{F}\to \mathcal{F}^0 / \mathcal{F}^1$ is a quasi-isomorphism. From this one can see that also the projection $\Tw\Br\to \Tw \BT$ is a quasi-isomorphism. But since $\Br \to \Tw \BT$ is a quasi-isomorphism by section \ref{sec:Br-TwBT}, we are done.

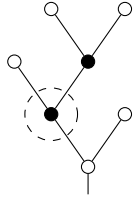
\begin{figure}
\centering
\begin{tikzpicture}
[
lab/.style={draw, circle, minimum size=5, inner sep=1}, 
n/.style={draw, circle, fill, minimum size=5, inner sep=1},
yscale=-1, scale=.7
]
\draw (5,7) node[lab] (v1) {} 
+(-.7,-1)  node[n] (v2) {} 
+(.7,-1) node[lab] (v3) {} 
+(0,.7) node (root) {};
\draw (v2)+(-.7,-1) node[lab] (v4) {}
+(.7, -1) node[n] (v5) {};
\draw (v5)+(-.7,-1) node[lab] (v6) {}
+(.7, -1) node[lab] (v7) {};
\draw (v1) edge (v2) edge (v3) edge (root) 
(v2) edge(v4) edge (v5)
(v5) edge (v6) edge (v7);
\draw[dashed] (v2) circle (.5);
\end{tikzpicture}
\caption{\label{fig:firstint} The \emph{first} neutral vertex of a brace tree.}
\end{figure}

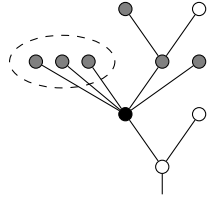
\begin{figure}
\centering
\begin{tikzpicture}
[lab/.style={draw, circle, minimum size=5, inner sep=1},
n/.style={draw, circle, fill,minimum size=5, inner sep=1},
gray/.style={draw, circle, fill=gray, minimum size=5, inner sep=1}, yscale=-1, scale=.7
]
\draw (5,7) node[lab] (v1) {}
+(-.7,-1)  node[n] (v2) {}
+(.7,-1) node[lab] (v3) {}
+(0,.7) node (root) {};
\draw (v2)+(-.7,-1) node[gray] (v4) {}
+(-1.2, -1) node[gray] (v21) {}
+(-1.7, -1) node[gray] (v22) {}
+(+1.4, -1) node[gray] (v23) {}
+(.7, -1) node[gray] (v5) {};
\draw (v5)+(-.7,-1) node[gray] (v6) {}
+(.7, -1) node[lab] (v7) {};
\draw (v1) edge (v2) edge (v3) edge (root) 
(v2) edge(v4) edge (v5) edge (v21) edge (v22) edge (v23)
(v5) edge (v6) edge (v7);
\draw[dashed] (v21) ellipse (1 and .5);
\end{tikzpicture}
\caption{\label{fig:veryleft}  The filtration on $V_p$ we use comes from the number of valence zero gray vertices attached to to the very left of the first neutral vertex. In this example, that number is three.}
\end{figure}

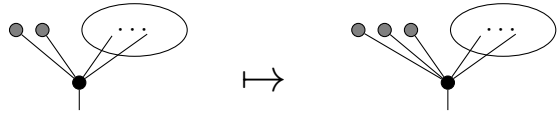
\begin{figure}
\centering
\begin{tikzpicture}
[lab/.style={draw, circle, minimum size=5, inner sep=1},
n/.style={draw, circle, fill,minimum size=5, inner sep=1},
gray/.style={draw, circle, fill=gray, minimum size=5, inner sep=1},
yscale=-1, scale=.7
]
\begin{scope}[xshift=-3.5cm]
\draw (5,7)  node[n] (v2) {}
+(0,.7) node (root) {};
\draw (v2)+(-.7,-1) node[gray] (v4) {}
+(-1.2, -1) node[gray] (v21) {}
+(+1.4, -1) node[gray] (v23) {}
+(.7, -1) node[gray] (v5) {};
\draw[fill=white] (v2)+(1.05,-1) ellipse (1 and .5) ;
\draw  (v2)+(1.05,-1) node {$\cdots$};
\draw 
(v2) edge(v4) edge (v5) edge (v21)  edge (v23)  edge (root);
\end{scope}
\node at (5,7) {\huge $\mapsto $};
\begin{scope}[xshift=3.5cm]
\draw (5,7)  node[n] (v2) {}
+(0,.7) node (root) {};
\draw (v2)+(-.7,-1) node[gray] (v4) {}
+(-1.2, -1) node[gray] (v21) {}
+(-1.7, -1) node[gray] (v22) {}
+(+1.4, -1) node[gray] (v23) {}
+(.7, -1) node[gray] (v5) {};
\draw[fill=white] (v2)+(1.05,-1) ellipse (1 and .5) ;
\draw  (v2)+(1.05,-1) node {$\cdots$};
\draw 
(v2) edge(v4) edge (v5) edge (v21) edge (v22) edge (v23)  edge (root);
\end{scope}
\end{tikzpicture}
\caption{\label{fig:graydiff} The differential increases the number of valence zero gray vertices attached to to the very left of the first neutral vertex by at most one. The component that increases this number by exactly one is shown here. Note that if this number is odd, the differential acts as zero.}
\end{figure}

\begin{figure}
\centering

\includegraphics[scale=1]{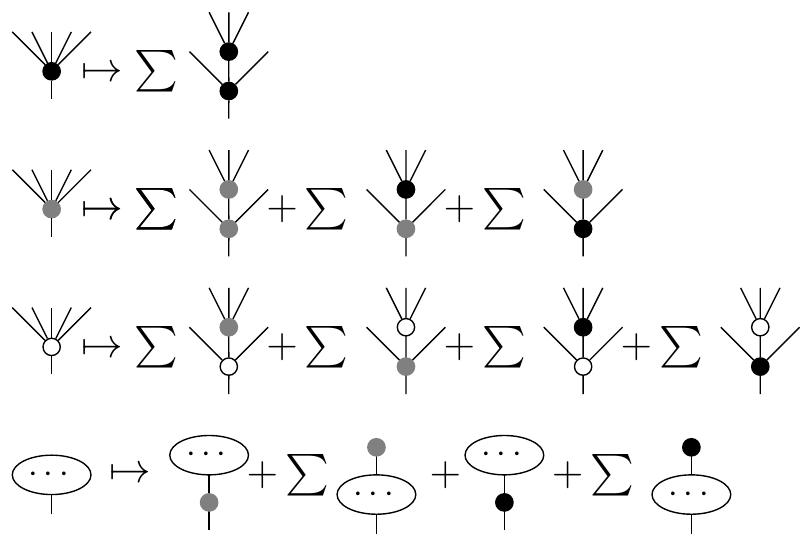}
\caption{\label{fig:g1diff}
A pictorial description of the differential on $\Tw \Br$. On the associated graded $\Gr \mathcal{F}$ the terms that create black vertices (i. e., the first row and the last two terms in the remaining rows) are absent. The terms depicted in Figure \ref{fig:graydiff} are contained in the second term of the last row. 
}
\end{figure}

\vspace{0.1cm}

\noindent\textsc{Department of Mathematics,
Temple University, \\
Wachman Hall Rm. 638\\
1805 N. Broad St.,\\
Philadelphia, PA, 19122 USA \\
\emph{E-mail address:} {\bf vald@temple.edu}}

\vspace{0.5cm}

\noindent\textsc{University of Z\"urich, \\
Institute of Mathematics, \\
Winterthurerstrasse 190,\\
8057 Z\"urich, Switzerland  \\
\emph{E-mail address:} {\bf thomas.willwacher@math.uzh.ch}}


\begin{thebibliography}{xx}

\bibitem{Bat-Markl} M. Batanin and M. Markl, 
Crossed interval groups and operations on the Hochschild cohomology, arXiv:0803.2249.

\bibitem{BF} C. Berger and B. Fresse, 
Combinatorial operad actions on cochains, 
Math. Proc. Cambridge Philos. Soc., {\bf 137}, 1 (2004) 135--174.

\bibitem{BM} C. Berger and I. Moerdijk, 
Axiomatic homotopy theory for operads,
Comment. Math. Helv. {\bf 78}, 4 (2003) 805--831;
arXiv:math/0206094.

\bibitem{C-Lazarev09} J. Chuang and A. Lazarev, 
L-infinity maps and twistings, Homology Homotopy Appl. {\bf 13}, 2 (2011) 175--195.


\bibitem{C-Lazarev} J. Chuang and A. Lazarev, 
Combinatorics and formal geometry of the master equation,
Lett. Math. Phys. {\bf 103}, 1 (2013) 79--112;  
arXiv:1205.5970.

\bibitem{CEFT}  V.A. Dolgushev, Covariant and equivariant formality theorems, 
Adv. Math. {\bf 191}, 1 (2005)147--177.

\bibitem{thesis} V.A. Dolgushev, A proof of Tsygan's formality conjecture for an arbitrary 
smooth manifold, Ph.D. thesis, M.I.T. (2005); arXiv:math/0504420.

\bibitem{notes} V.A. Dolgushev and C.L. Rogers, 
Notes on Algebraic Operads, Graph Complexes, and Willwacher's Construction, 
{\it Mathematical aspects of quantization}, 25--145, Contemp. Math., {\bf 583}, Amer. Math. Soc., 
Providence, RI, 2012; arXiv:1202.2937.

\bibitem{Swiss} V.A. Dolgushev, D.E. Tamarkin, and B.L. Tsygan, 
Proof of Swiss Cheese Version of Deligne's Conjecture, 
Int. Math. Res. Notices {\bf 2011} (2011) 4666--4746.



\bibitem{Fresse98} B. Fresse, Lie theory of formal groups over an operad, J. of Algebra {\bf 202} (1998), 455--511.

\bibitem{Fresse} B. Fresse,
Koszul duality of operads and homology of partition posets,
in "Homotopy theory and its applications
(Evanston, 2002)", Contemp. Math. {\bf 346} (2004)
115--215.

\bibitem{BV} I. Galvez-Carrillo, A. Tonks, and B. Vallette, 
Homotopy Batalin-Vilkovisky algebras, 
J. Noncomm. Geom. {\bf 6}, 3 (2012) 539--602; 
arXiv:0907.2246. 

\bibitem{Ger} M. Gerstenhaber, The cohomology structure of
an associative ring,  Annals of Math., {\bf 78} (1963)
267--288.

\bibitem{GJ} E. Getzler and J.D.S. Jones,
Operads, homotopy algebra and iterated integrals
for double loop spaces, hep-th/9403055.

\bibitem{GK} V. Ginzburg and M. Kapranov,
Koszul duality for operads,  Duke Math. J.
{\bf 76}, 1 (1994) 203--272.

\bibitem{Hinich-hom} V. Hinich, Homological algebra of homotopy algebras, Comm. Algebra {\bf 25}
(1997), 3291--3323.

\bibitem{Hinich-err} V. Hinich,  Erratum to ``Homological algebra of homotopy algebras'', 
arXiv:math/0309453.

\bibitem{Hinich} V. Hinich, Tamarkin's proof of
Kontsevich formality theorem, Forum Math. {\bf 15},
4 (2003) 591--614; math.QA/0003052.

\bibitem{KapranovManin} M. Kapranov and Y. Manin, 
Modules and Morita theorem for operads, 
Amer. J. Math. {\bf 123} (2001), 811--838.

\bibitem{K-mot} M. Kontsevich,  Operads and motives in deformation quantization, 
Lett. Math. Phys. {\bf 48}, 1 (1999) 35--72.

\bibitem{K-Soi} M. Kontsevich and Y. Soibelman,
Deformations of algebras over operads and the Deligne conjecture,
{\it Proceedings of the Mosh\'e Flato Conference
Math. Phys. Stud.} {\bf 21},
255--307, Kluwer Acad. Publ., Dordrecht, 2000.

\bibitem{Loday-Vallette} J.-L. Loday and B. Vallette, Algebraic Operads, 
{\it Grundlehren der mathematischen Wissenschaften}, {\bf 346}, Springer-Verlag 2012. 

\bibitem{Markl} M. Markl, Operads and PROPs, {\it Handbook of algebra.} 
Vol. 5, 87--140, Elsevier/North-Holland, Amsterdam, 2008; 
arXiv:math/0601129.

\bibitem{Markl-distr} M. Markl, Distributive laws and Koszulness, 
Ann. Inst. Fourier (Grenoble) {\bf 46}, 2 (1996) 307--323.

\bibitem{MSS-book} M. Markl, S. Shnider, and J. Stasheff, Operads in algebra, topology and physics,
{\it Mathematical Surveys and Monographs}, {\bf 96}. AMS, Providence, RI, 2002. x+349 pp.

\bibitem{M-Smith} J. E. McClure and J. H. Smith,
A solution of Deligne's Hochschild cohomology conjecture,
Contemp. Math. {\bf 293} (2002) 153--193,
Amer. Math. Soc., Providence, RI;
math.QA/9910126.

\bibitem{M-Smith2} J. E. McClure and J. H. Smith,
Multivariable cochain operations and little {$n$}-cubes,
J. Amer. Math. Soc. {\bf 16}, 3 (2003) 681--704.

\bibitem{MV-nado} S. Merkulov and B. Vallette, 
Deformation theory of representations of prop(erad)s, I, II,
J. Reine Angew. Math. {\bf 634} (2009), 51--106. and J. Reine Angew. Math. {\bf 636} (2009), 123--174,
arXiv:0707.0889. 

\bibitem{Quillen} D. Quillen, Rational homotopy theory,
Ann. of Math. (2) {\bf 90} (1969) 205--295. 

\bibitem{Dima-disc}  D. Tamarkin,  Formality of chain operad of little discs, 
Lett. Math. Phys. {\bf 66}, 1-2 (2003) 65--72. 

\bibitem{Dima-dg} D. Tamarkin, What do DG categories form? 
Compos. Math. {\bf 143}, 5 (2007) 1335--1358; math.CT/0606553.

\bibitem{Vor} A.A. Voronov, Homotopy Gerstenhaber algebras, 
{\it Proceedings of the Mosh\'e Flato Conference} Math. Phys. Stud.,
22, 307-331. Kluwer Acad. Publ., Dordrecht, 2000.

\bibitem{grt} T. Willwacher, M. Kontsevich's graph complex and the 
Grothendieck-Teichm\"uller Lie algebra, arXiv:1009.1654.


\bibitem{Yek} A. Yekutieli, Continuous and Twisted L-infinity Morphisms, J. Pure Appl. Algebra {\bf 207}, (2006), 575--606.

\end{thebibliography}
\end{document}